\documentclass[11pt]{amsart} \textwidth=14.5cm \oddsidemargin=1cm
\usepackage{amsmath,amsthm,amssymb,latexsym,epic,bbm,comment,yfonts,mathrsfs}
\usepackage{graphicx,enumerate,stmaryrd,rotating,xcolor}
\usepackage[all]{xy}
\xyoption{2cell}

\allowdisplaybreaks

\usepackage[active]{srcltx}
\usepackage{youngtab}

\usepackage[latin1]{inputenc}
\usepackage{amsxtra}
\usepackage{amsmath}
\usepackage{amscd}
\usepackage{amssymb}
\usepackage{amsfonts}
\usepackage[all]{xy}
\usepackage{mathrsfs}
\usepackage{amsthm}
\usepackage{color}
\usepackage{upgreek}
\usepackage{verbatim}  
\usepackage{enumerate}
\usepackage{latexsym}
\usepackage{graphicx}
\usepackage{tikz}
\usepackage{todonotes}
\usepackage{setspace}
\usepackage{hyperref}
\usepackage{stmaryrd}
\usepackage{nicefrac}
\usepackage{euscript}
\usetikzlibrary{decorations.pathmorphing}

 \definecolor{darkgreen}{HTML}{336633}
 \definecolor{darkred}{HTML}{993333}
\makeatletter

\newcommand{\arxiv}[1]{\href{http://arxiv.org/abs/#1}{\tt
    arXiv:\nolinkurl{#1}}}

\theoremstyle{plain}
\newtheorem{thm}{Theorem}
\newtheorem*{thm*}{Theorem}

\newtheorem{lem}[thm]{Lemma}
\newtheorem{prop}[thm]{Proposition}

\newtheorem{cor}[thm]{Corollary}
\newtheorem{df-prop}[thm]{Definition-Proposition}

\theoremstyle{definition}

\theoremstyle{remark}
\newtheorem{rem}[thm]{Remark}
\newtheorem{ex}[thm]{Example}

\def\bbQ{\mathbb{Q}}

\def\bbT{\mathbb{T}}

\def\onto{\twoheadrightarrow}

\def\mod{\operatorname{-mod}\nolimits}

\newcommand{\hf}{{\Small \frac12}}

\def\Hom{\operatorname{Hom}\nolimits}
\def\End{\operatorname{End}\nolimits}
\def\Res{\operatorname{Res}\nolimits}
\def\Ind{\operatorname{Ind}\nolimits}
\def\Ext{\operatorname{Ext}\limits}

\def\ep{\epsilon}
\def\gl{\mathfrak{gl}}

\def\la{\lambda}

\def\pn{\mf{pe} (n)}
\def\ov{\overline}

\newcommand{\mc}{\mathcal}
\newcommand{\mf}{\mathfrak}
\newcommand{\C}{\mathbb C}

\newcommand{\oa}{{\bar 0}}
\newcommand{\ob}{{\bar 1}}

\newcommand{\cF}{\mathcal{F}}

\newcommand{\fg}{\mathfrak{g}}

\newcommand{\fh}{\mathfrak{h}}

\newcommand{\cO}{\mathcal{O}}

\newcommand{\h}{\mathfrak{h}}
\newcommand{\N}{\mathbb{Z}_{\geq 0}}
\newcommand{\ch}{\mathrm{ch}}
\newcommand{\rad}{\mathrm{rad}}

\newcommand{\Coind}{{\rm Coind}}

\newcommand{\g}{\mathfrak{g}}
\newcommand{\fl}{\mathfrak{l}}
\newcommand{\fp}{\mathfrak{p}}
\newcommand{\fu}{\mathfrak{u}}
\newcommand{\Real}{\mathrm{Re}}

\newcommand{\nb}{{\nabla}}
\newcommand{\zmnz}{{\mathbb Z^{m|n}_{\zeta-}}}
\newcommand{\Z}{{\mathbb Z}}

  \newcommand{\add}{{\mathrm{add}}}

     \def\Ann{{\text{Ann}}}

       \def\Id{{\text{Id}}}

                \def\vpre{{\nu}\text{-pres}}

    \def\Opres{\mc O^{\nu\text{-pres}}}

        \def\mofA{{\text{mof-}A}}
                       \def\gmofA{{\text{gmof-}A}}
                   \def\mofB{{\text{mof-}B}}
                    \def\gmofB{{\text{gmof-}B}}
                    \def\dotoo{{\dot{\ov{\mc O}}}}
   
          \def\OI{{\ov{\mc O}}}
         \def\Lnua{{\Lambda(\nu)}}
          \def\Lnub{{\Lambda^+(\nu)}}

          \def\xnula{{\Lnua_{>\la}}}
\begin{document}

\numberwithin{equation}{section}

\title[Whittaker categories for Lie superalgebras]{Whittaker categories, properly stratified categories and Fock space categorification for Lie superalgebras}

\author[Chen]{Chih-Whi Chen}
\address{Department of Mathematics, National Central University, Chung-Li, Taiwan 32054} \email{cwchen@math.ncu.edu.tw}
\author[Cheng]{Shun-Jen Cheng}
\address{Institute of Mathematics, Academia Sinica, and National Center of Theoretical Sciences, Taipei, Taiwan 10617} \email{chengsj@math.sinica.edu.tw}
\author[Mazorchuk]{Volodymyr Mazorchuk}
\address{Department of Mathematics, Uppsala University, P.O. Box 480, SE-75106 Uppsala, Sweden}
\email{mazor@math.uu.se}

\date{}

\begin{abstract}
  	 We study various categories of Whittaker modules over a type I	Lie superalgebra realized as cokernel categories that fit into the framework of  properly stratified categories. These categories are the target of the Backelin functor $\Gamma_\zeta$. We show that these categories can be described, up  to equivalence,  as Serre quotients of the BGG category $\mc O$ and of certain singular categories of Harish-Chandra $(\g,\g_\oa)$-bimodules. We also show that $\Gamma_\zeta$ is a realization of the Serre quotient functor. We further investigate a $q$-symmetrized Fock space over a quantum group of type A and prove that, for general linear Lie superalgebras our Whittaker categories, the functor $\Gamma_\zeta$ and various realizations of Serre quotients and Serre quotient functors categorify this $q$-symmetrized Fock space and its $q$-symmetrizer. In this picture, the canonical and dual canonical bases in this $q$-symmetrized Fock space correspond to tilting and simple objects in these Whittaker categories, respectively.
\end{abstract}

\maketitle

\tableofcontents

\noindent
\textbf{MSC 2010:} 17B10 17B55

\noindent
\textbf{Keywords:}  Lie algebra, Lie superalgebra, Whittaker module, Whittaker category, properly stratified category, Fock space categorification.
\vspace{5mm}

\section{Introduction}\label{sec1}

\subsection{Background}

Representation theory of Lie algebras is an important and rapidly developing branch of modern algebra.
One of the most interesting objects of study in this theory is the Bernstein-Gelfand-Gelfand (BGG) category $\mathcal{O}$
associated to a triangular decomposition of the Lie algebra in question. For a complex semisimple Lie algebra,
this category was introduced in \cite{BGG0, BGG}. It has numerous remarkable properties and applications; see
\cite{Hu08} and references therein. Many of these properties and applications have their origin in the
Kazhdan-Lusztig conjecture formulated in  \cite{KL1} and proved in \cite{BB, BK}.
This remarkable combinatorial point of view of category $\mathcal{O}$ is based on certain properties of the Hecke
algebra of the Weyl group associated to the Lie algebra in question and it establishes quite surprising and
deep connections between representations of semisimple Lie algebras and other areas like combinatorics and geometry.

Motivated by the theory of Whittaker models, Kostant initiated in \cite{Ko78} the study of
the so-called {\em Whittaker vectors} and the corresponding modules over complex semisimple
Lie algebras. Each Whittaker vector is associated to a character $\zeta$ of the nilpotent
radical in the triangular decomposition.
For each {\em non-singular} character $\zeta$, that is, a character that does not vanish on any simple root vector, Kostant constructed a family of simple modules $Y_{\zeta,\chi}$
additionally indexed by a central character $\chi$. The action of the Cartan subalgebra on these simple modules is not
semisimple and hence these simple modules do not belong to category $\mathcal{O}$. Nevertheless, there is a
relationship to category $\mc O$ as these modules can be realized, for example, as uniquely defined submodules of the vector space duals
of Verma modules. Whittaker modules have been extensively studied in various contexts, see
\cite{Mc, Mc2, MS, B, BM, CoM, ALZ, Bro, R1} and references therein.

In \cite{MS},  Mili{\v{c}}i{\'c} and Soergel introduced a certain category $\mc N(\zeta)$ for semisimple Lie algebras
which is especially suitable for  the study and classification of simple Whittaker modules. The categories $\mc N(\zeta)$, with varying parameter $\zeta$, contain both
Kostant's simple modules and modules in category $\mc O$. This approach
resulted in  a complete classification of simple Whittaker modules, based on two earlier papers \cite{Mc, Mc2} by McDowell.
An important role in this approach was played by the so-called {\em standard Whittaker modules} $M(\la,\zeta)$, which  specialize to
the classical Verma modules in the case $\zeta=0$.

Similarly to category $\mc O$, each object in $\mc N(\zeta)$ has finite length; see \cite[Theorem~2.6]{MS}.
In \cite{MS}, it is shown that the composition multiplicities for  $M(\la,\zeta)$ associated to
an integral weight $\la$ can be computed using Kazhdan-Lusztig polynomials, just like for Verma  modules in category $\mc O$.
The corresponding result in full generality was obtained by Backelin in \cite{B} using a certain exact functor
$\Gamma_\zeta$ from $\mc O$ to $\mc N(\zeta)$.

\subsection{Whittaker categories}
While the multiplicities of simple Whittaker modules in standard Whittaker modules  in the case of semisimple Lie algebras are, by now, well-understood,
the general structural properties of various Whittaker categories require more study. Often it is helpful
to restrict  one's attention to some full subcategory of $\mc N(\zeta)$ which has strong connections with
other classical objects of study in representation theory.

The category $\mc N(\zeta)$ itself was studied in detail in \cite{BR}. In particular, in \cite[Sections~7,8]{BR}
	it was shown that the full subcategories in the regular blocks of $\mc N(\zeta)$ consisting of modules annihilated by kernels of central characters are highest weight categories. 
	However, the corresponding   subcategories in singular blocks of $\mc N(\zeta)$ have more complicated structures. In the present paper, we provide a suitable full subcategory of $\mc N(\zeta)$, containing these subcategories and the category $\mc O$ and its generalizations, for which such a reciprocity holds.

To explain the definition of our Whittaker category, we start by
 recalling the  equivalence between
$\mc N(\zeta)$ and a certain category of Harish-Chandra bimodules constructed by  Mili{\v{c}}i{\'c} and Soergel in \cite{MS}. The principal notion behind this construction
is that of a {\em cokernel category}. It  connects various full subcategories of $\mc O$ with certain singular
categories of Harish-Chandra bimodules. The cokernel category in question is the full  subcategory of $\mc N(\zeta)$
which consists of all modules admitting a two-step presentation by summands of modules of the form
$E\otimes M(\la,\zeta)$, where $E$ is  a finite-dimensional $\g$-module and $\la$ is a dominant weight.
We denote this cokernel category by $\text{Coker}(\mc F\otimes M(\la,\zeta))$.

Such cokernel subcategories inside category $\mc O$ were studied, in particular, in  \cite{FKM00} and \cite{MaSt04,MaSt08},
under the name of {\em $\mathcal{S}$-subcategories of $\mc O$}, or alternatively, under the name of
{\em categories of  projectively presentable modules} and denoted by  $\mc O^{\nu\text{-pres}}$, see Section \ref{sect::41}.
The stuctural properties of blocks of  $\mc O^{\nu\text{-pres}}$ are governed by the notion of
{\em (properly) stratified algebras},  see  \cite{Dl,CPS2} for the latter and also \cite{BS} for 
generalizations. The framework of cokernel  categories
and  stratified algebras provides  a uniform description for a number of generalizations of category $\mc O$.
In particular, the cokernel category $\text{Coker}(\mc F\otimes M(\la,\zeta))$  turns out to provide a suitable
setup for our study of Whittaker modules.

\subsection{Whittaker modules for Lie superalgebras}

A Cartan-Killing type classification of finite-dimensional complex simple Lie superalgebras
was obtained by Kac in \cite{Ka1, Ka2}. Over the last decades the representation  theory of Lie superalgebras has been a rapidly developing part of representation theory, see, e.g.,
\cite{Se96, Br1, CLW1, BW, CLW2, Br04q} and references therein, for various aspects of
the theory of finite-dimensional representations and, more generally, category $\mc O$.

While Whittaker modules for Lie algebras received a great amount of attention in the last decade,
their super analogues were studied in detail only very recently,
see \cite{BCW, BrG, Ch21, Ch212}. The category $\mc N(\zeta)$ and standard
Whittaker modules can be defined, in a natural way, in the general setup of quasireductive
Lie superalgebras of type I; see Section~\ref{sect::21} for the definitions.  As in the Lie algebra case,
   Backelin's functor $\Gamma_\zeta$, which extends naturally to an exact functor from $\mc O$ to $\mc N(\zeta)$ for Lie superalgebras, plays a similar role in the multiplicity problem for standard Whittaker modules. More precisely, $\Gamma_\zeta$ maps Verma modules to standard Whittaker modules and simple highest weight modules to simple Whittaker modules (or zero) over Lie superalgebras; see \cite[Theorem 20]{Ch21}. It turns out that the problem about composition multiplicities for standard Whittaker modules can be reduced to that for Verma modules, as established in \cite[Theorem C]{Ch21}.


One of the motivations for the present paper is to describe a suitable analogue of
the cokernel category $\text{Coker}\left(\mc F\otimes M(\la,\zeta)\right)$ in the setup of
Lie superalgebras. This gives rise to a category which we denote by $\mc W(\zeta)$.
We study various aspects of the structure theory of this category, in particular,
we describe its stratified structure in detail. The category $\mc W(\zeta)$ turns out to
have a remarkably rich representation theory.

Whittaker modules and their generalization were also studied in the setup of
(finite) $W$-algebras; see, e.g., \cite{BrG,BGK,Lo09, Lo10, Pr07,Xi,ZS} and references therein.   It is   worth pointing out that  the category $\mc N(\zeta)$ above is equivalent to the category $\mc O$ for a $W$-algebra in the sense of  \cite{Lo09} when $\g$ is a semisimple Lie algebra; see  \cite[Theorem 4.1, Proposition 4.2]{Lo09} and Section \ref{rem::40}.

\subsection{Kazhdan-Lusztig conjecture and categorifications}

The Kazhdan-Lusztig conjecture for a semisimple  Lie algebra $\mathfrak{g}$ can be
formulated using the categorification of the right regular module over the Hecke algebra
for the Weyl group of $\mathfrak{g}$ via the action of projective functors on
the Grothendieck group of the principal block of category $\mc O$.
For a fixed parabolic subgroup of $W$,
there are two families of parabolic Hecke modules, studied in \cite[Section~3]{So}.
One of these families specializes to the permutation module
while the other one specializes to the induced sign module.
The family that specializes to the induced sign module decategorifies the
action of projective functors on the regular block of the corresponding
parabolic category $\mc O$, see \cite{RC}. Via parabolic category $\mc O$,
this family is further related to the theory
parabolic Kazhdan-Lusztig  polynomials; see \cite{KL1, Deo87, Deo91, BGS96}.

The family that specializes to the permutation module decategorifies the
action of projective functors on the regular block of the corresponding
category  $\mc O^{\nu\text{-pres}}$ of projectively presentable modules,
as shown in \cite{MaSt08}.

In  \cite{FKK}, Frenkel, Khovanov and Kirillov gave a formulation of the Kazhdan-Lusztig conjectures for Lie algebras of type $A$ in terms of canonical bases \`a la Lusztig on tensor powers of the fundamental representation of the quantum group of type $A$. As a consequence, the Kazhdan-Lusztig polynomials can be read off from the coefficients of the canonical bases.

In the setup of Lie superalgebras, there have been numerous attempts to obtain
results towards the  solution of the irreducible character problem in category $\mc O$.
The most significant step is Brundan's formulation of a Kazhdan-Lusztig type conjecture for the general linear Lie superalgebra $\gl(m|n)$ in \cite{Br1}. This conjecture was proved by
Lam, Wang and the second author in \cite{CLW2}, see also  \cite{BLW}.
The conjecture is formulated as a Fock space categorification of the super category $\mc O$ in which
certain canonical bases then play a crucial role. We refer to \cite{Br04q, BW, Bao17, CLW1, ChWa12} for treatments
of some of the other simple classical Lie superalgebras.

\subsection{Goals}
In the present paper, we construct a $q$-symmetrized Fock space consisting of a tensor product of the $q$-symmetric tensors of the natural and its (restricted) dual module of the infinite-rank quantum group of type $A$. Equivalently, it is the image of a certain $q$-symmetrizer $S_\zeta$ acting on the Fock space consisting of a tensor product of a tensor power of the natural module and a tensor power of its dual module. This $q$-symmetrization can be intuitively paralleled with the construction of
a permutation module. Making another parallel with Lie algebras, it is natural to expect that
the $q$-symmetrized Fock space should be related to combinatorics of various cokernel categories.
In particular, it should be important for the study of Whittaker modules.

In the paper, we construct and study various categories and modules over
quasireductive Lie superalgebras that are  equivalent to $\mc W(\zeta)$.
We show how, in the case of the Lie superalgebra $\gl(m|n)$, the combinatorics of these categories can be understood via the $q$-symmetrized
Fock space  and how the action of projective functors on these categories can be used to categorify
the $q$-symmetrized Fock space. As a part of this story, we develop a theory of (dual)
canonical bases arising from $q$-symmetric tensors. We also use our categorification results
and various categorical equivalences to formulate and prove Brundan-Kazhdan-Lusztig type conjectures
for these categories.

Throughout the paper,  for a given quasireductive Lie superalgebra of type I, we analyze the  stratified structure of a certain {\em Serre quotient category}
$\ov{\mc O}$ of category $\mc O$ in the sense of Gabriel \cite{Gabriel}. This Serre quotient gives one of
the equivalent incarnations of $\mc W(\zeta)$, generalizing some results in \cite{MaSt04, FKM00} to the
setup of Lie superalgebras. There is a canonical {\em Serre quotient functor} $\pi$ from $\mc O$ to
$\ov{\mc O}$. which is determined uniquely by certain universal property.

We prove that the Backelin  functor $\Gamma_\zeta$ is a realization of the functor $\pi$, see also \cite{BrG}
where the case of  $\mf{gl}(m|n)$ with a  non-singular $\zeta$ was considered. Let us explain the strategy of our proof in more
detail. The first step is to study the functor $F_\zeta$ obtained by by composing the functors
used in \cite[Theorem~5.9]{BG} and \cite[Theorem~5.1]{MS}. In particular, we show that the functor $F_\zeta$
has the universal property which determines $\pi$. We further show that the functors $F_\zeta$ and
$\Gamma_\zeta$ have the same codomain. Then, we  employ the tools from projective functors, developed
in \cite{MaMe12} and \cite{Kh}, to prove that $F_\zeta$ and $\Gamma_\zeta$ are isomorphic. By
uniqueness of the Serre quotient functor, we conclude that all $F_\zeta$, $\Gamma_\zeta$ and $\pi$ have
equivalent codomain categories, and, finally, we show that they are isomorphic, up to equivalences of
these categories.

Using this explicit realization of $\Gamma_\zeta$ via the Serre quotient, the construction of simple,
standard, costandard, projective and tilting modules in $\mc W(\zeta)$ can be realized as the images
of certain modules in $\mc O$ under $\Gamma_\zeta$. The realization is also new for reductive Lie algebras and results in a number of interesting observations
about $\Gamma_\zeta$ and other applications    and generalizations, see  Sections~\ref{re::37}, \ref{rem::40},  \ref{ref::41} and Remark \ref{rem::44}.   In particular, we investigate the adjoints of $\Gamma_\zeta$ and  the connection with the representation theory of finite $W$-algebras and develop an analogue of Soergel's Struktursatz for $\Gamma_\zeta$.

We also construct tilting modules in $\mc W(\zeta)$ with respect  to its stratified structure  for any quasireductive Lie superalgebra of type I. In our categorification picture  for $\mf{gl}(m|n)$, the representation theoretical counterparts of  the
canonical and dual canonical bases in the $q$-symmetrized Fock space are given by the tilting and irreducible Whittaker modules in $\mc W(\zeta)$, respectively. Therefore, the corresponding Kazhdan-Lusztig polynomials determined by these bases admit the usual interpretation as
composition multiplicities, alternatively, multiplicities of standard subquotients
in standard filtrations. We further prove that both $\Gamma_\zeta$ and $\pi$
categorify the $q$-symmetrizer $S_\zeta$.  This also provides a new categorification picture for $\gl(m)$ by means of their corresponding Whittaker categories.

\subsection{Structure of the paper}

This paper is organized as follows. In Section \ref{sect::prel}, we provide
some background material on quasireductive Lie superalgebras. In Section \ref{sect::3}, we develop
the theory of $q$-symmetric tensors in Fock space for quantum groups and the theory of
(dual) canonical bases on $q$-symmetric tensors.

Section~\ref{sect::4} is devoted to the description, study and comparison of several
(equivalent) abelian incarnations of the Whittaker category $\mc W(\zeta)$.
In Section \ref{sect::5}, we  describe the stratified structure of these categories in detail.

In Section \ref{sect::6}, we focus on the study of tilting modules in both $\OI$ and $\mc W(\zeta)$.
We describe explicitly all tilting modules as objects of $\OI$ and use this description to
establish a Ringel self-duality  result for $\OI$. In Section \ref{sect::7}, we combine all the
above results to establish the categorification of (a topological completion of)
the $q$-symmetrized Fock space via the action of projective functors on the categories studied
in Section~\ref{sect::4}.  In particular, we use the functor $\Gamma_\zeta$, the
tilting modules and the simple modules in $\mc W(\zeta)$ to categorify the $q$-symmetrizer,
the canonical basis  and the dual canonical bases, respectively.   In Section \ref{sect::81}, we give a description of structural modules in $\mc O^{\nu\text{-pres}}$. Section \ref{App::9} offers a realization of $\OI$ in terms of the category of  finite-dimensional locally unital modules over a locally unital algebra.

\vskip 0.5cm
{\bf Acknowledgments}. The first two authors are partially supported by MOST grants of the R.O.C. They also
	acknowledge support from the National Center for Theoretical Sciences of the R.O.C. The third author is supported by the Swedish Research Council. {We thank A.~Brown and A.~Romanov for pointing out an inaccuracy in the first version.}

\section{Preliminaries} \label{sect::prel}

Throughout the paper, the symbols $\Z$, $\Z_{\geq 0}$, and $\Z_{\leq 0}$ stand
for the sets of all, non-negative and non-positive integers, respectively. All vector spaces,
algebras and tensor products are assumed to be over the field $\C$ of complex numbers.

\subsection{} \label{sect::21}
For a complex Lie superalgebra $\g$ we denote the category of (left) $\g$-modules by $\g$-Mod. Furthermore, set $\g\mod$ to be the full subcategory of all finitely generated $\g$-modules. We denote the universal enveloping algebra of $\g$ by $U(\g)$ and its center by $Z(\g)$.

Throughout, we let $\g =\g_\oa \oplus \g_\ob$ be a finite-dimensional quasireductive Lie superalgebra. This means that $\g_\oa$ is a reductive Lie algebra
and $\g_\ob$ is semisimple as a $\g_\oa$-module. We note that such Lie superalgebras are called classical in \cite{CCC,Ma}. Also, we assume that $\g$ is equipped with a compatible $\Z$-grading:
\begin{align}
	&\g=\g_{-1}\oplus \g_{0}\oplus \g_{1},\label{typeI}
	\end{align} which we refer to as a {\em type-I grading} in the paper. Note that $\g_{0}=\g_\oa$. We further set $\mf g_{\leq 0}:=\g_0\oplus \g_{-1}$ and $\g_{\geq 0}:=\g_0\oplus\g_{1}$.

 One of the most interesting classes in Kac's list \cite{Ka1} is the following series of type I Lie superalgebras:
\begin{align}
	&(\text{Type } {\bf A}):~\gl(m|n),~\mf{sl}(m|n),~\mf{psl}(n|n), \label{eq::claA}  \\
	&(\text{Type } {\bf C}):~\mf{osp}(2|2n),\label{eq::claC} \\
	&(\text{Type } {\bf P}):~\mf{p}(n),~[\mf{p}(n),\mf{p}(n)]. \label{eq::claP}
\end{align}

\subsection{} We fix  a triangular decomposition
\begin{align}
&\mf g_\oa=\mf n^-_\oa\oplus \mf h \oplus \mf n_\oa^+ \label{eqtraigoa}
\end{align}   of $\mf g_\oa$ with the corresponding Borel subalgebra $\mf b_\oa = \mf h\oplus \mf n_\oa^+$ such that $\mf b:=\mf b_\oa+\mf g_1$ is a Borel subalgebra of $\g$ in the sense of \cite{Ma} (see also  \cite[Section 1.3]{CCC}). This means that,
setting $\mf n^\pm_\ob =\mf g_{\pm 1}$, we obtain a triangular decomposition
\begin{align}
&\mf g=\mf n^-\oplus \mf h \oplus \mf n^+ \label{supereqtraigoa}
\end{align}
of $\g$ that extends \eqref{eqtraigoa}.
Here $\mf n^\pm = \mf n^\pm_\oa\oplus \mf n^\pm_\ob$.
The  subalgebra $\fh$ of $\g_\oa$ is referred to as a {\em Cartan subalgebra} of $\g$.  We then have the set $\Phi\subseteq \h^\ast$ of roots of $\mf g$
with respect to $\mf h $ and the corresponding root space decomposition
\begin{align}
&\g=\bigoplus_{\alpha\in \Phi\cup \{0\}}\g_\alpha,\label{eq::rootdec}
\end{align} where $\g_\alpha:=\{X\in \g|~[h,X] =\alpha(h)X, \forall h\in \h\}$, for $\alpha \in \h^\ast$.

 For a given subalgebra $\mf s\subseteq \g$, let $\Phi(\mf s)$ denote the set of roots in $\mf s$. We let $$\Phi^\pm:=\Phi(\mf n^\pm)\quad\text{and}\quad\Phi^\pm_\oa:=\Phi(\mf n_\oa^\pm)$$ denote the sets of positive/negative and even positive/negative roots, respectively. The basis of simple
roots in $\Phi^+_\oa$  is denoted by $\Pi_0$.

Furthermore, we define the partial order $\leq$ on $\h^\ast$ to be the transitive closure of the relations
\begin{align*}
&\la\pm \alpha \leq \la, \text{ for }\alpha\in \Phi^\mp.
\end{align*}

\subsection{}\label{sec:domtyp}  We denote by $\rho$ the Weyl vector, i.e., $\rho=\rho_\oa-\rho_\ob$, where
$\rho_\oa:=\hf\sum_{\alpha\in\Phi^+_\oa}\alpha$ and $\rho_\ob=\hf\sum_{\beta\in\Phi^+_\ob}\beta$. Also, we set $\rho_{\pm 1}$ to be the half of the sum of all roots in $\g_{\pm 1}$.  We remark that $\rho_{\ov 1} = \rho_1$. The Weyl group $W$ of $\mf g$ is defined to be the Weyl group of $\g_\oa$ with respect to the triangular decomposition \eqref{eqtraigoa}. The group $W$ acts on $\h^\ast$ by definition. The corresponding {\em dot-action} of $W$ on $\h^\ast$ is defined as \begin{align}
	&w\cdot \la :=  w(\la+\rho) - \rho= w(\la+\rho_\oa) - \rho_\oa, \label{dotaction}
	\end{align}   for any $\la \in \h^\ast$. For any $\alpha \in \Pi_0$, we let $s_\alpha\in W$ denote the corresponding simple reflection.

We fix a $W$-invariant non-degenerate bilinear form $\langle \cdot, \cdot \rangle$ on $\h^\ast$, which we assume to be induced from an invariant non-degenerate bilinear form on $\g$, provided that the latter exists.
We set $\alpha^\vee:= 2\alpha/\langle\alpha,\alpha\rangle$, for any even root $\alpha$.  A weight $\la\in \h^\ast$ is called {\em integral} if $\langle \la, \alpha^\vee\rangle \in \Z$, for all simple even roots $\alpha$. We denote by  $\Lambda\subseteq \h^\ast$ the set of all integral weights.  An integral weight $\la$ is called {\em dominant} (respectively, {\em anti-dominant}) if $\langle \la+\rho_\oa, \alpha^\vee \rangle \in \Z_{\geq 0}$ (respectively, $\langle \la+\rho_\oa, \alpha^\vee \rangle \in \Z_{\leq 0}$) and  {\em regular} if and only if  $\langle \la+\rho_\oa, \alpha^\vee \rangle \neq 0$, for all $\alpha \in \Phi_\oa^+$. In the case when $\langle\cdot,\cdot\rangle$ is induced from a form on $\g$, a weight $\la\in\h^\ast$ is called {\em typical}, if $\langle\la+\rho,\alpha\rangle\not=0$, for all $\alpha\in\Phi_\ob$ with $\langle\alpha,\alpha\rangle=0$.

\subsection{}\label{sec:lambdanu}
Throughout the present paper, we fix a  dominant and integral weight $\nu\in \h^\ast$ and set	
 \begin{align*}
&{\Pi_\nu}:=\{\alpha \in \Pi_0|~\langle \nu+\rho_\oa, \alpha \rangle =0\}.
\end{align*}
 Let $\mf l_\nu$ be the Levi subalgebra generated by $\mf h$ and all $\g_{\pm \alpha}$, where $\alpha\in {\Pi_\nu}$. Denote by $W_\nu$ the Weyl group of $\mf l_\nu$, that is, $W_\nu$ is generated by all $s_\alpha$, where $\alpha \in {\Pi_\nu}$.
 Let $\Lnua$ be the set of all integral and $W_\nu$-anti-dominant  weights, namely,
 \begin{align*}
 	&\Lnua:=\{\la\in \Lambda|~\langle \la+\rho_\oa, \alpha^\vee\rangle \in \Z_{\leq 0},\text{ for all }\alpha \in {\Pi_\nu}\}.
 	\end{align*}

Let   ${\bf ch} \mf n^+_\oa:=\{\zeta\in (\mf n_\oa^+)^\ast|~\zeta([\mf n_\oa^+, \mf n_\oa^+])=0\}$ be the set of characters of $\mf n_\oa^+$.
Following \cite{Ch21}, we fix a character $\zeta \in {\bf ch} \mf n_\oa^+$ with
\begin{align*}
	&\zeta(\g_\alpha) \not= 0 \Longleftrightarrow \alpha\in \Pi_\nu, \text{ for any } \alpha\in \Pi_0.
\end{align*}
Define $\mf l_\zeta:=\mf l_\nu$  to be the Levi subalgebra in a parabolic decomposition of $\g$ in the sense of \cite{Ma} (see also \cite[Section 1.4]{CCC}). Namely, there is a vector $H\in \h^\ast$ such that
\begin{align*}
	&\fl_\nu=\bigoplus_{\Real\, \alpha(H)=0}\g_\alpha,\quad\fu^+:=\bigoplus_{\Real\,\alpha(H)>0} \fg_\alpha \subseteq \mf n^+, \quad \fu^-:=\bigoplus_{\Real\, \alpha(H)<0} \fg_\alpha \subseteq \mf n^-.
\end{align*}
We denote the parabolic subalgebra $\mf p:=\mf l_\nu\oplus\mf u^+$. Note that $\fp_\ob= \mf g_1$.
In the paper, we define $W_\zeta:=W_\nu$.


\subsection{} \label{sect::24}
We consider  the BGG category $\mc O = \mc O(\mf g, \mf b)$ of $\g$-modules with respect to the triangular decomposition \eqref{supereqtraigoa}. In this paper we only consider the ``integral'' subcategory of $\mc O$, i.e., the full subcategory consisting of all modules having integral weight spaces. Henceforth, unless otherwise specified, by $\mc O$ and other similar notations we shall always mean the corresponding integral subcategories. We denote by $\mc O_\oa$ the corresponding BGG category of $\g_\oa$-modules with respect to the triangular decomposition \eqref{eqtraigoa}. We denote by $\mc F$ and $\mc F_\oa$ the categories of finite-dimensional $\g$- and $\g_\oa$-modules, respectively.

Let $\mf a$ be a subalgebra of $\g$ such  that $\mf g_\oa \subseteq \mf a \subseteq \mf g$.  We denote by $\Res^{\mf g}_{\mf a}(-)$ the restriction functor from $\mf g$-Mod to $\mf a$-Mod. Then $\Res^{\mf g}_{\mf a}(-)$ has the following left and right adjoint functors, respectively:
$$\Ind^{\g}_{\mf a}(-)=U(\g)\otimes_{U(\mf a)}-\qquad\mbox{and}\qquad \Coind^{\g}_{\mf a}(-)=\Hom_{U(\mf a)}(U(\g),-).$$
By \cite[Theorem~2.2]{BF}, see also \cite{Go}, we have the following isomorphism of functors:
\begin{align}
	&\Ind_{\mf a}^{\mf g}(-)\;\cong\; \Coind_{\mf a}^{\g}(\wedge^{\dim(\mf g/\mf a)}(\mf g/\mf a)\otimes -). \label{eq::adj}
\end{align} 
Suppose that $\mf a$ has a Borel subalgebra  $\mf b^{\mf a}\subseteq \mf a$ such that $\mf b^{\mf a}_\oa = \mf b_\oa$. Then we use $\mc O(\mf a)\cong\mc O(\mf a, \mf b^{\mf a})$ to denote the BGG category for $\mf a$ with respect to $\mf b^{\mf a}$, if the latter is clear from the context. In this case,  $\Ind_{\mf a}^\g, ~\Coind_{\mf a}^\g$ and $\Res_{\mf a}^{\g}$ are well-defined exact functors between $\mc O(\mf a)$ and $\mc O$. In the case of a reductive Lie algebra $\g=\g_\oa$, we have an anti-involution of
$\g$ described in \cite[Section 0.5]{Hu08}. We denote by $(\cdot)^\vee$ the
corresponding simple-preserving duality on $\mathcal{O}$;  see also \cite[Section 3.2]{Hu08}.

Let $ M^{\mf a}(\la)$ and $L^{\mf a}(\la)\in \mc O(\mf a)$ denote the Verma module
and its simple quotient corresponding to $\la$, respectively. Also, let $P^{\mf a}(\la)$ be the projective cover of $L^{\mf a}(\la)$ in $\mc O(\mf a)$. We use the notation $P(\la), L(\la), M(\la)$ (respectively, $P_0(\la), L_0(\la), M_0(\la)$) in the case when $\mf a =\mf g$ (respectively, $\mf a = \mf g_\oa$). Also, we let $T(\la)$ and $T_0(\la)$ denote the tilting modules with highest weight $\la$ in $\mc O$ and $\mc O_\oa$, respectively. We refer to \cite[Section 3.3]{CCC} for more details.

For any $\g$-module $M$ having a composition series and any simple
$\g$-module $L$, by $[M : L]$ we denote the composition
multiplicity of $L$ in $M$. For given $M\in \mc O$ and $\la \in \h^\ast$, we denote the  weight space
of $M$ corresponding to $\lambda$ by $$M_\la: = \{m\in M|~hm=\la(h)m, \forall h\in \h\}.$$  We  let $\ch M$ denote the character of $M$, namely,
\begin{align}
&\ch M:= \sum_{\la \in \h^\ast} \dim M_\la e^{\la}. \label{eq::wtspace}
\end{align}

For $\alpha\in \Pi_0$, a simple module $L(\la)$ is said to be {\em $\alpha$-free} if every non-zero vector $f_{\alpha}\in \g_{-\alpha}$ acts freely on $L(\la)$. We note that $\la \in \Lnua$ if and only if $L(\la)$ is $\alpha$-free for all $\alpha\in {\Pi_\nu}$.

\subsection{}
For a $\g$-module $M$, we denote by $\mc F\otimes M$ the full subcategory of
the category of $\g$-modules consisting of all $\g$-modules of the form
$V\otimes M$, where $V\in \mc F$.

Given a full subcategory $\mc C$ of $\g$-modules, we denote by
\begin{itemize}
\item $\langle \mc C\rangle$ the full subcategory of
the category of all $\g$-modules consisting of all subquotients
of all modules in~$\mc C$;
\item $\add(\mc C)$ the {\em additive closure} of $C$, that is
the full subcategory of the category of all $\g$-modules
consisting of all modules isomorphic to finite direct sums
of direct summands of objects in $\mc C$;
\item $\text{Coker}(\mc C)$ the full subcategory of the category
of all $\g$-modules, which consists of all modules $M$ that
have a presentation of the form
$X\rightarrow Y\rightarrow M\rightarrow 0,$
where $X$ and $Y$ are in $\mc C$.
\end{itemize}

For a given $\g$-module $M$ that has a composition series, the socle $\text{soc}(M)$ of $M$ is defined as the sum of all simple submodules in $M$. Dually, the radical $\rad(M)$ of $M$ is defined to be the intersection of all maximal submodules. The top of $M$ is defined as $\text{top}(M):=M/\rad M$.

\subsection{}
Let $\la\in \h^\ast$ and let $\chi_\la: Z(\g)\rightarrow \C$ {(respectively $\chi^{\g_\oa}_\la: Z(\g_\oa)\rightarrow \C$)} denote the central character of $\g$ {(respectively $\g_\oa$)} associated to $\la$.

Denote by $\g\mod_Z$  the full subcategory of $\g\mod$ on which $Z(\g)$ acts locally finitely. As in the Lie algebra situation, we define {\em projective (endo)functors} on $\g\mod_Z$ as direct summands of functors of the form $V\otimes -$, where $V\in \mc F$.  The category of projective functors is denoted by ${\mc Proj}$. In the case of a semisimple Lie algebra $\g=\g_\oa$, projective functors have been investigated in \cite{BG}.

We note that every functor in ${\mc Proj}$ preserves $\mc O$. Let $\la \in \h^\ast$ be a generic, typical, dominant, integral and regular weight as in \cite[Section 5.3]{MaMe12}.  In \cite[Theorem 5.1]{MaMe12}, it was proved that the triple ($\mc O$,\,$M(\la)$,\,${\mc Proj}$) is a {\em category with full projective functors}  in the sense of \cite[Definition 1]{Kh}:
\begin{itemize}
	\item[(i)] Every object $M\in \mc O$ is a quotient of $T(M(\la))$, for some $T\in {\mc Proj}$.
	\item[(ii)] For $T, T'\in {\mc Proj}$, the evaluation map $$\text{ev}_{M(\la)}: \Hom_{\mc Proj}(T,T')\rightarrow\Hom_{\mc O}(T(M(\la)), T'(M(\la)))$$ is surjective.
\end{itemize}

A functor $G$ from $\mc O$ to a full abelian subcategory $\mc C$ of $\g\mod_Z$ is said to functorially commute with projective functors if for every $T\in {\mc Proj}$ there is an isomorphism $\eta_T: T\circ G \rightarrow G\circ T$ such that for any natural transformation $\alpha: T\rightarrow T'$ of projective functors the following diagram is commutative (see also \cite[Section 3.3]{MaMe12}):
	$$\xymatrixcolsep{3pc} \xymatrix{
	T\circ G  \ar[r]^-{\alpha_{G}}  \ar@<-2pt>[d]_{\eta_T} &   T'\circ G   \ar@<-2pt>[d]_{\eta_{T'}} \\
	G\circ T \ar[r]^{G(\alpha)} & G\circ T'}$$
The following lemma is a consequence of \cite[Proposition 4]{Kh}.
\begin{lem} \label{lem::fpfO} Let $\la$ be as above. 	Let $\mc C$ be a full abelian subcategory of $\g\mod_Z$ that is invariant under projective functors. Suppose that $S,S':\mc O\rightarrow \mc C$ are two exact functors that functorially commute with projective functors. Then $S\cong S'$ if and only if $S(M(\la))\cong S'(M(\la))$.
\end{lem}

\section{Quantum groups and canonical bases} \label{sect::3}

\subsection{Quantum group $U_q(\gl_\infty)$}

Let $q$ be an indeterminate. The quantum group $U_q({\mathfrak
g\mathfrak l}_\infty)$ is the associative algebra over
$\bbQ(q)$ generated by $E_a, F_a, K_a, K^{-1}_a$, where $a \in \Z$, subject to
the following relations ($a,b\in\Z$):
\begin{eqnarray*}
 K_a K_a^{-1} &=& K_a^{-1} K_a =1, \\
 K_a K_b &=& K_b K_a, \\
 K_a E_b K_a^{-1} &=& q^{\delta_{a,b} -\delta_{a,b+1}} E_b, \\
 K_a F_b K_a^{-1} &=& q^{\delta_{a,b+1}-\delta_{a,b}}
 F_b, \\
 E_a F_b -F_b E_a &=& \delta_{a,b} \frac{K_{a,a+1}
 -K_{a+1,a}}{q-q^{-1}}, \\
 E_a^2 E_b +E_b E_a^2 &=& (q+q^{-1}) E_a E_b E_a,  \quad \text{if } |a-b|=1, \\
 E_a E_b &=& E_b E_a,  \,\qquad\qquad\qquad \text{if } |a-b|>1, \\
 F_a^2 F_b +F_b F_a^2 &=& (q+q^{-1}) F_a F_b F_a,  \quad\, \text{if } |a-b|=1,\\
 F_a F_b &=& F_b F_a,  \qquad\ \qquad\qquad \text{if } |a-b|>1.
\end{eqnarray*}
Here $K_{a,a+1} :=K_aK_{a+1}^{-1}$. For $r\ge 1$, we have the
divided powers $E_a^{(r)} =E_a^{r}/[r]!$ and
$F_a^{(r)}=F_a^{r}/[r]!$, where $[r] =(q^r-q^{-r})/(q-q^{-1})$ and
$[r]!=[1][2]\cdots [r]$. The quantum group $U_q(\mathfrak{sl}_\infty)$ is the subalgebra of $U_q({\mathfrak g\mathfrak l}_\infty)$ generated by $E_a$ and $F_a$, for $a\in\Z$. Of relevance for constructing canonical bases later on is the $\Z[q,q^{-1}]$-form $U_q(\mathfrak{sl}_\infty)_\Z$ containing the divided powers as generators in \cite[Sections 3.1.13 and 23.2.1]{Lu}.

Setting $\ov{q}=q^{-1}$ induces an automorphism on $\bbQ(q)$ denoted
by $^{-}$. Define the bar involution on $U_q({\mathfrak g\mathfrak
l}_\infty)$ to be the anti-linear automorphism ${}^-: U_q({\mathfrak
g\mathfrak l}_\infty)\rightarrow U_q({\mathfrak g\mathfrak
l}_\infty)$ determined by $\ov{E_a}= E_a$, $\ov{F_a}=F_a$, and
$\ov{K_a}=K_a^{-1}$. Here {\em anti-linear} means that
$\ov{fu}=\ov{f}\ov{u}$, for $f\in\bbQ(q)$ and $u\in U_q(\gl_\infty)$.

Let $\mathbb V$ be the natural $U_q({\mathfrak g\mathfrak
l}_\infty)$-module with basis $\{v_a\}_{a\in\Z}$ and let $\mathbb W
:=\mathbb V^*$, be the restricted dual module with basis
$\{w_a\}_{a\in\Z}$ such that $\langle w_a ,v_b\rangle = (-q)^{-a}
\delta_{a,b}$.  The actions of $U_q(\gl_\infty)$ on $\mathbb V$ and
$\mathbb W$ are then given as follows:
\begin{align*}
&K_av_b=q^{\delta_{ab}}v_b,\qquad E_av_b=\delta_{a+1,b}v_a,\ \ \quad
F_av_b=\delta_{a,b}v_{a+1},\\
&K_aw_b=q^{-\delta_{ab}}w_b,\quad E_aw_b=\delta_{a,b}w_{a+1},\quad
F_aw_b=\delta_{a+1,b}w_{a}.
\end{align*}
We shall use the following comultiplication $\Delta$ on $U_q(\gl_\infty)$:
\begin{eqnarray*}
 \Delta (E_a) &=& 1 \otimes E_a + E_a \otimes K_{a+1, a}, \\
 \Delta (F_a) &=& F_a \otimes 1 +  K_{a, a+1} \otimes F_a,\\
 \Delta (K_a) &=& K_a \otimes K_{ a},
\end{eqnarray*}
This $\Delta$ is consistent with the one in \cite{Kas}, but
differs from the one in \cite{Lu}.

For $m, n\in\N$, we denote
the set of integer-valued functions on $\{1,\ldots,m+n\}$ by $\Z^{m|n}$. We
shall also identify $f\in\Z^{m|n}$ with the $(m+n)$-tuple
$\left(f(1),f(2),\ldots,f(m+n)\right)$. We call $f$ {\em dominant} provided that
\begin{displaymath}
f(1)\ge f(2)\ge\cdots\ge f(m)\quad\text{ and }\quad f(m+1)\le f(m+2)\le\cdots\le f(m+n).
\end{displaymath}
A function in $\Z^{m|n}$ satisfying the above conditions with $\le$ and $\ge$ interchanged, will be called {\em anti-dominant}. Denote the subsets of dominant and anti-dominant functions by $\Z^{m|n}_+$ and $\Z^{m|n}_-$, respectively.

\subsection{Fock space $\bbT^{m|n}$}

For $m,n\in\N$, we define the following
tensor space over $\bbQ(q)$:
\begin{equation}  \label{eq:Fock}
{\bbT}^{m|n} :={\mathbb V}^{\otimes m}\otimes {\mathbb
W}^{\otimes n}.
\end{equation}
The quantum group $U_q(\gl_\infty)$ acts on $\mathbb T^{m|n}$ via the
co-multiplication $\Delta$.

For $f\in \Z^{m|n}$, we define
\begin{equation}  \label{eq:Mf}
M_f :=v_{f(1)}\otimes
v_{f(2)}\otimes\cdots \otimes v_{f(m)}\otimes
w_{f(m+1)}\otimes\cdots\otimes w_{f(m+n)}.
\end{equation}
We refer to $\{M_f |f\in
\Z^{m|n}\}$ as the {\em standard monomial basis} for ${\mathbb
T}^{m|n}$.

\subsection{Bruhat order}\label{super:Bruhat}
In this section, we introduce a partial order on $\Z^{m|n}$ that, through the bijection defined in \eqref{eq:wt:bij}, translates to the Bruhat order for  the Lie superalgebra $\gl(m|n)$.

Let $\texttt{P}$ be the free abelian group with basis
$\{\varepsilon_r\vert r\in\Z\}$. We define a partial order on $\texttt{P}$ by
declaring $\nu\ge\mu$, for $\nu,\mu\in \texttt{P}$, if $\nu-\mu$ is
a non-negative integral linear combination of
$\varepsilon_r-\varepsilon_{r+1}$, $r \in \Z$.

Set $b_i=0$, for $i=1,\cdots,m$, and $b_j=1$, for $j=m+1,\cdots,m+n$. For $f\in\Z^{m|n}$ and $1\le j\le m+n$, we define
\begin{align*}
\text{wt}^j(f)
 :=\sum_{j\le i}(-1)^{b_i}\varepsilon_{f(i)}\in \texttt{P},\quad
\text{wt}(f):=\text{wt}^{1}(f)\in \texttt{P}.
\end{align*}
We define the {\em Bruhat ordering} $\preceq$ on the integral lattice $\Z^{m|n}$,
in terms of the partially ordered set
$(\texttt{P}, \leq)$, as follows: $g\preceq f$ if and only if
$\text{wt}(g)=\text{wt}(f)$ and $\text{wt}^j(g)\le\text{wt}^j(f)$, for all $j$.
Note that, when $n=0$,
this is the usual Bruhat ordering on the integral weight lattice of $\gl(m)$.

\subsection{Hecke algebra of type $A$}

For $k\in\N$, denote by $\mf{S}_{k}$ the symmetric group on $\{{1},{2},\ldots,{k}\}$. The group $\mf{S}_{k}$ is generated
by the simple transpositions $s_{1} =({1},{2})$, $s_{2}
=({2},{3}), \ldots$, $s_{k-1} =({k-1},{k})$.

The Iwahori-Hecke algebra associated to $\mf{S}_k$ is the associative $\mathbb Q(q)$-algebra
$\mathcal H_{k}$ generated by $H_i$, where $1 \le i \le k-1$, subject to
the relations
\begin{align*}
&(H_i -q^{-1})(H_i +q) = 0,\\
&H_i H_{i+1} H_i = H_{i+1} H_i H_{i+1},\\
&H_i H_j = H_j H_i, \quad\text{for } |i-j| >1.
\end{align*}
Associated to $\sigma \in \mf{S}_{k}$ with a reduced
expression $\sigma=s_{i_1} \cdots s_{i_r}$, we define the element
 $
H_\sigma :=H_{i_1} \cdots H_{i_r}. $ The bar involution $^{\overline{\ }}$ on $\mathcal
H_{k}$ is the unique anti-linear automorphism such that
$\overline{q} =q^{-1}$ and $\overline{H_\sigma} =H_{\sigma^{-1}}^{-1}$,
for all $\sigma \in \mf{S}_{k}$. Set
\begin{align}\label{formula:S0}
S_0:= \sum_{\sigma \in \mf{S}_{k}}
q^{\ell(w_0)-\ell(\sigma)} H_\sigma,
\end{align}
where $w_0$ denotes the longest element in $\mf{S}_k$. It is
well known, see, e.g., \cite[Proposition~2.9]{So}, that
\begin{equation}  \label{eq:S0inv}
\ov{S_0}=S_0.
\end{equation}
Furthermore, for all $\sigma \in \mf{S}_{k}$, we have
\begin{align}\label{eq:S0H}
S_0 H_\sigma = q^{-\ell(\sigma)} S_0 = H_\sigma S_0.
\end{align}

Let $\texttt{V}$ be either $\mathbb V$ or $\mathbb W$ and consider the tensor spaces $\texttt V^{\otimes k}$. In either case, we index the tensor
factors by $\{1,2,\ldots,k\}$. Now, for a function $f: \{1,2,\ldots,k\}\rightarrow\Z$, recall from
\eqref{eq:Mf} that
$M_f = \texttt{v}_{f(1)} \otimes \cdots \otimes
\texttt{v}_{f(k)},$ where $\texttt{v}=v$ for $\texttt{V}=\mathbb
V$ and $\texttt{v}=w$ for $\texttt{V}=\mathbb W$. The
algebra $\mathcal H_{k}$ acts on $\texttt V^{\otimes k}$ on the right by following formulas
\begin{eqnarray} \label{eq:heckeaction}
 M_f H_i = \left\{
 \begin{array}{ll}
 M_{f\cdot s_i}, & \text{if } f \prec f \cdot s_i,  \\
 q^{-1} M_{f}, & \text{if } f = f \cdot s_i, \\
 M_{f \cdot s_i} - (q-q^{-1}) M_f, & \text{if } f \succ f \cdot
 s_i.
 \end{array}
 \right.
\end{eqnarray}

\begin{rem}
Suppose that $f:\{1,2,\ldots,k\}\rightarrow\Z$ with $f(i)>f(i+1)$. Then $f\cdot s_i\prec f$ in the case of $\mathbb V^{\otimes k}$, while $f\prec f\cdot s_i$ in the case of $\mathbb W^{\otimes k}$.
\end{rem}

It is known that there is a quantized version of Schur-Weyl duality (\cite{Jim}), that is, the actions of $U_q(\gl_\infty)$ and $\mathcal H_{k}$ on the tensor
spaces $\mathbb V^{\otimes k}$ (or $\mathbb W^{\otimes k}$) commute with each other.

\begin{ex}\label{ex:sym:can}  Let $\mathbb T^k = \mathbb T^{k|0}$ or $\mathbb T^{0|k}$.
Consider the subspace $\mathbb T^{k}S_0$ of $\mathbb T^{k}$. Because of the commuting actions of $U_q(\gl_\infty)$ and $\mc H_k$ on $\mathbb T^k$, we see that $\mathbb T^{k}S_0$ is a $U_q(\gl_\infty)$-module.  We denote it by $\texttt{S}^k(\texttt V)$. Let $f$ be anti-dominant with stabilizer $W_f:=\{\sigma\in\mf{S}_k|f\cdot\sigma=f\}$ and let $w_0^f$ denote the longest element in $W_f$. {We have $W_f=\prod_{i}\mf S_{m_i}$ and we set $[|W_f|]:=\prod_{i}[m_i]!$. Similar notations apply in the sequel.} Furthermore,
 let $W^f$ denote the set of shortest length representatives of the cosets $W_f\backslash\mf{S}_k$ with longest element $^fw_0$. We compute
\begin{align*}
M_f S_0&=M_f \sum_{\sigma\in W_f}q^{\ell(w_0^f)-\ell(\sigma)}H_\sigma\sum_{\tau\in W^f}q^{\ell(^fw_0)-\ell(\tau)}H_\tau\\
&=M_f \sum_{\sigma\in W_f}q^{\ell(w_0^f)-2\ell(\sigma)}\sum_{\tau\in W^f}q^{\ell(^fw_0)-\ell(\tau)}H_\tau\\
&=[|W_f|]M_f \sum_{\tau\in W^f}q^{\ell(^fw_0)-\ell(\tau)}H_\tau\\
&=[|W_f|]\sum_{\tau\in W^f}q^{\ell(^fw_0)-\ell(\tau)}M_{f\cdot \tau}.
\end{align*}
In the penultimate equality above, we have used the formula
\begin{align*}
\sum_{\sigma\in\mf S_k}q^{\ell(w_0)-2\ell(\sigma)}=[k]!,
\end{align*}
which, in turn, follows from the formula $\sum_{\sigma\in\mf S_k}q^{\ell(\sigma)}=\prod_{i=1}^{k-1}(1+q+\cdots+q^i)$.
Thus,
\begin{align}\label{eq:can:sym}
\widetilde{M}_f:=\frac{1}{[|W_f|]} M_f S_0=\sum_{\tau\in W^f}q^{\ell(^fw_0)-\ell(\tau)}M_{f\cdot\tau}
\end{align}
lies in the $Z[q,q^{-1}]$-span of the $M_f$ and we see that $\{\widetilde{M}_f|f\in\Z^k_-\}$ is a basis for $\texttt{S}^k{\texttt V}$.
\end{ex}

\subsection{Bar involution and canonical bases on $\mathbb T^{m|n}$}\label{sec:bar:inv}

Combining the right actions of the Hecke algebra $\mc{H}_m$ and $\mc{H}_n$ on $\mathbb V^{\otimes m}$ and $\mathbb W^{\otimes n}$, respectively, we obtain a right action of $\mc{H}_{m|n}:=\mc{H}_m\times\mc{H}_n$ on $\mathbb T^{m|n}$.

The space $\mathbb T^{m|n}$ has a bar involution that can be defined via Lusztig's quasi-$\mc R$-matrix as in \cite[Proposition 3.6]{CLW2} (see also \cite[Theorem 2.14]{Br1}). To formulate precise statements on (dual) canonical bases we need to work with a certain topological completion of $\mathbb T^{m|n}$ as in \cite[\S2-d]{Br1} (or \cite[Section 3.2]{CLW2}). This completion $\widehat{\mathbb T}^{m|n}$ allows for certain forms of infinite linear combinations of the monomial basis elements $M_f$ respecting the Bruhat order. As this will not be important for our purpose we shall not define this completion precisely here, and instead refer the reader to loc.~cit.~for the details. A consequence is that the standard monomial basis is a topological basis with respect to this topological completion.

Indeed, the bar involution leaves invariant the $\Z[q,q^{-1}]$-lattice  topologically spanned by the monomial basis $M_f$, $f\in\Z^{m|n}$, which we state below (\cite[Theorem 2.14]{Br1}, \cite[Proposition 3.6]{CLW2}).

\begin{prop}
Let $f\in\Z^{m|n}$. There exists an anti-linear bar map ${}^{\overline{\ }}:\widehat{\mathbb T}^{m|n}\rightarrow \widehat{\mathbb T}^{m|n}$ such that
\begin{align*}
\overline{M_f}=M_f+\sum_{g\prec f}r_{gf}(q) M_g,
\end{align*}
where $r_{gf}(q)\in\Z[q,q^{-1}]$ and the sum is possibly infinite. Furthermore, the bar involution is compatible with the bar involutions of $U_q(\gl_\infty)$ and $\mc{H}$, i.e., for $m\in\widehat{\mathbb T}^{m|n}$, $u\in U_q(\gl_\infty)$ and $H\in\mc{H}_{m|n}$ we have $\overline{umH}=\bar{u}\overline{m}\overline{H}$.
\end{prop}

The Bruhat order defined in Section \ref{super:Bruhat} satisfies the so-called finite interval property (see, e.g., \cite[Lemma 2.4]{CLW2}) so that, by \cite[Lemma 24.2.1]{Lu}, we have the following (see, e.g., \cite{Br1,CLW2}).

\begin{prop}\label{prop:ex:can1}
The $\bbQ(q)$-vector space $\widehat{\mathbb T}^{m|n}$ has unique
bar-invariant topological bases
\begin{align*}
\{T_f|f\in\Z^{m|n}\}\text{ and }\{L_f|f\in\Z^{m|n}\}
\end{align*}
such that
\begin{align*}
T_f=M_f+\sum_{g\prec f}t_{gf}(q)
M_g,
 \qquad
L_f=M_f+\sum_{g\prec f}\ell_{gf}(q)
M_g,
\end{align*}
with $t_{gf}(q)\in q\Z[q]$, and $\ell_{gf}(q)\in
q^{-1}\Z[q^{-1}]$, for $g\prec f$. (We will also write
$t_{ff}(q)=\ell_{ff}(q)=1$, $t_{gf}=\ell_{gf}=0$ for $g\npreceq f$.)
\end{prop}

The bases $\{T_f|f\in\Z^{m|n}\}$ and $\{L_f|f\in\Z^{m|n}\}$ are referred to as {\em canonical basis} and {\em dual canonical basis} for $\mathbb T^{m|n}$ or rather $\widehat{\mathbb T}^{m|n}$,
respectively, and $t_{gf}(q)$ and $\ell_{gf}(q)$
are called {\em Brundan-Kazhdan-Lusztig (BKL) polynomials}. It can be shown that the canonical basis elements indeed lie in $\mathbb{T}^{m|n}$.

\begin{rem}\label{rem:can:sk}
Let $f\in\Z_-^{k|0}$ (or $f\in\Z_-^{0|k}$) with stabilizer $W_f$. Proposition \ref{prop:ex:can1} implies that the element $\widetilde{M}_f$ in \eqref{eq:can:sym} is the canonical basis element $T_{f\cdot w_0}$ in $\mathbb T^{k|0}$ (or $\mathbb T^{0|k}$). Since the $\Z[q,q^{-1}]$-lattice spanned by the standard monomial basis elements is invariant under $U_q(\mf{sl}_\infty)_\Z$, it follows that the $\Z[q,q^{-1}]$-lattice spanned by the canonical basis elements is also invariant under $U_q(\mf{sl}_\infty)_\Z$. Since $\texttt{S}^k(\texttt V)$ is a $U_q(\gl_\infty)$-module, it follows that the $\Z[q,q^{-1}]$-span of $\{\widetilde{M}_f|f\text{ anti-dominant}\}$ is a $\Z[q,q^{-1}]$-lattice inside $\texttt{S}^k(\texttt V)$ invariant under $U_q(\mf{sl}_\infty)_\Z$.
\end{rem}

\subsection{$q$-symmetric tensors}\label{sec:qsymm}
 Let $W_\zeta$ be a (parabolic) subgroup of $\mf{S}_m\times\mf{S}_n$ generated by simple reflections.  The notation $W_\zeta$ match that introduced in Section \ref{sec:lambdanu} for the case $\g=\mf{gl}(m|n)$, and it will be used for our corresponding categorification picture.

Recall the Hecke algebra $\mc{H}_{m|n}$ from Section \ref{sec:bar:inv} and let $\mc H_\zeta$ be the Hecke subalgebra associated with $W_\zeta$. Let
\begin{align}\label{eq:q:symm}
S_\zeta:=\sum_{\sigma\in {W_\zeta}} q^{\ell(w_0^\zeta)-\ell(\sigma)}H_\sigma,
\end{align}
where $w_0^\zeta$ is the longest element in $W_\zeta$. The bar involution on $\mc{H}_{m|n}$ restricts to the bar involution on $\mc{H}_\zeta$, and hence, regarding $S_\zeta$ as an element in $\mc{H}_{m|n}$, we see from \eqref{eq:S0inv} and \eqref{eq:S0H} that
\begin{align*}
\ov{S_\zeta}=S_\zeta\quad \text{and}\quad S_\zeta H_\sigma=q^{-\ell(\sigma)}S_\zeta=H_\sigma S_\zeta ,\ \sigma\in W_\zeta.
\end{align*}

We consider the following $U_q(\gl_\infty)$-submodule of $\mathbb T^{m|n}$:
\begin{align}\label{eq:Fock:z}
\mathbb T^{m|n}_\zeta:=\mathbb T^{m|n} S_\zeta.
\end{align}
Let $\Z^{m|n}_{\zeta-}$ denote the set of elements in $\Z^{m|n}$ that are anti-dominant with respect to $W_\zeta$. For $f\in\Z^{m|n}_{\zeta-}$, we let $W_f:=\{\sigma\in W_\zeta|f\cdot\sigma=f\}$ be the stabilizer subgroup inside $W_\zeta$ and $W^f$ the set of shortest length representatives of the coset $W_f\backslash W_\zeta$.

Define two sets of monomial bases for $\mathbb T^{m|n}_\zeta$:
\begin{align}\label{eq:Nf}
\widetilde{N}_f:=M_f S_\zeta;\qquad \widetilde{M}_f:=\frac{1}{[|W_f|]}M_f S_\zeta,\quad f\in\Z^{m|n}_{\zeta-}.
\end{align}
Both bases lie in the $\Z[q,q^{-1}]$-span of $\{M_g|g\in\Z^{m|n}\}$ and are invariant under $U_q(\mf{sl}_\infty)_\Z$.
For $f\in\Z^{m|n}_{\zeta-}$, we also define
\begin{align}\label{def:Nf}
N_f:=\frac{[|W_\zeta|]}{[|W_f|]}\widetilde{N}_f.
\end{align}

Suppose that $f\in\Z^{m|n}$, but not necessarily in $\Z^{m|n}_{\zeta-}$. Let $\tau\in W_\zeta$ be of shortest length such that $f\cdot\tau\in\Z^{m|n}_{\zeta-}$. Then, by \eqref{eq:S0H}, we have
\begin{align}\label{eq:MfS}
M_f S_\zeta= M_{f\cdot\tau} H_{\tau^{-1}} S_\zeta = q^{-\ell(\tau)} M_{f\cdot\tau}S_\zeta = q^{-\ell(\tau)} \widetilde{N}_{f\cdot\tau}.
\end{align}
Similarly, we have the identity $\frac{1}{[|W_f|]}M_f S_\zeta=q^{-\ell(\tau)}\widetilde{M}_{f\cdot\tau}$.

Since the $\Z[q,q^{-1}]$-lattices spanned by $\{\widetilde{N}_f|f\in\Z^{m|n}_{\zeta-}\}$ and $\{\widetilde{M}_f|f\in\Z^{m|n}_{\zeta-}\}$ inside $\mathbb T^{m|n}_\zeta$ are invariant under $U_q(\mf{sl}_\infty)_\Z$, the bar involution on $\widehat{\mathbb T}^{m|n}$ restricts to a bar involution on the similar topological completion $\widehat{\mathbb T}^{m|n}_\zeta$ of $\mathbb T^{m|n}$. We have the following (c.f.~\cite[Section 5.2]{CLW2}).

\begin{lem} The bar involution on $\widehat{\mathbb T}^{m|n}$ restricts to a bar involution on $\widehat{\mathbb T}^{m|n}_\zeta$ such that,
for $f\in\Z^{m|n}_{\zeta-}$, we have:
\begin{itemize}
\item[(1)] $\overline{\widetilde{N}}_f\in \widetilde{N}_f+\sum_{g\prec f}\Z[q,q^{-1}] \widetilde{N}_g$,
\item[(2)] $\overline{\widetilde{M}}_f\in \widetilde{M}_f+\sum_{g\prec f}\Z[q,q^{-1}] \widetilde{M}_g$,
\item[(3)] $\overline{N}_f\in {N}_f+\sum_{g\prec f}\Z[q,q^{-1}] {N}_g.$
\end{itemize}
\end{lem}

\begin{proof}
For $f\in\Z^{m|n}_{\zeta-}$, suppose that
$\overline{M}_f = M_f+\sum_{g\prec f}r_{gf}M_g$ with $r_{gf}\in\Z[q,q^{-1}]$.
Then, by \eqref{eq:MfS}, we have
\begin{align*}
\overline{\widetilde{N}}_f&=\overline{M_f S_\zeta}=\overline{M}_f S_\zeta = M_fS_\zeta+\sum_{g\prec f}r_{gf}M_gS_\zeta
=\widetilde{N}_f + \sum_{g\prec f; g\in\Z^{m|n}_{\zeta -}}r'_{gf} \widetilde{N}_g.
\end{align*}
Since $r_{gf}\in\Z[q,q^{-1}]$, we have $r'_{fg}\in\Z[q,q^{-1}]$ by \eqref{eq:MfS}. This gives (1).

To prove (2), we note that $\widetilde{M}_f$, for $f\in\Z^{m|n}_{\zeta-}$, is a decomposable tensor with tensor factors consisting of either standard monomial basis elements in $\texttt{V}^{\otimes p}$ or canonical basis elements in $\texttt{S}^q(\texttt V)$, where again $\texttt{V}$ denotes either $\mathbb V$ or $\mathbb W$. Now, by Remark \ref{rem:can:sk}, the $\Z[q,q^{-1}]$-lattice spanned by either such factors is invariant under $U_q(\mf{sl}_\infty)_\Z$. From the explicit form of the quasi-R-matrix in \cite[Section 3.1]{CLW2} used to construct the bar involution on tensor products of ``based'' modules, it follows that $\overline{\widetilde{M}}_f$ lies in the $\Z[q,q^{-1}]$-lattice of the $\widetilde{M}_g$'s.

Finally, (3) follows now by applying the same argument to prove (1) using (2).
\end{proof}

We can now invoke \cite[Lemma 24.2.1]{Lu} again and obtain the following.

\begin{prop}
The $\bbQ(q)$-vector space $\widehat{\mathbb T}^{m|n}_\zeta$ has unique
bar-invariant topological bases of the form
\begin{align*}
\{\mc{ T}'_f|f\in\Z^{m|n}_{\zeta-}\},\quad\{\mc{ T}_f|f\in\Z^{m|n}_{\zeta-}\}\quad\text{and}\quad\{\mc{ L}_f|f\in\Z^{m|n}_{\zeta-}\}
\end{align*}
such that
\begin{align*}
\mc T'_f=\widetilde{M}_f+\sum_{g\prec f}\texttt{t}'_{gf}(q)\widetilde{M}_g,
\quad
\mc T_f={N}_f+\sum_{g\prec f}\texttt{t}_{gf}(q){N}_g,
 \quad
\mc L_f=\widetilde{N}_f+\sum_{g\prec f}\texttt{l}_{gf}(q)\widetilde{N}_g,
\end{align*}
with $\texttt{t}'_{gf}(q),\texttt{t}_{gf}\in q\Z[q]$, and $\texttt{l}_{gf}(q)\in
q^{-1}\Z[q^{-1}]$, for $g\prec f$ in $\Z^{m|n}_{\zeta-}$.
\end{prop}

As usual, we adopt the convention
$\texttt{t}'_{ff}(q)=\texttt{t}_{ff}(q)=\texttt{l}_{ff}(q)=1$, $\texttt{t}'_{gf}=\texttt{t}_{gf}=\texttt{l}_{gf}=0$ for $g\npreceq f$, and also refer these polynomials as BKL-polynomials.

\subsection{Canonical Bases on $q$-symmetric tensors}

In this section we shall compare the bases $\{T_f\}$ and $\{L_f\}$ with the bases $\{\mathcal T_f\}$ and $\{\mathcal L_f\}$.

Suppose that $f\in\Z^{m|n}_{\zeta-}$ and $\mc{T}'_f$ is the corresponding canonical basis element in $\widehat{\mathbb T}^{m|n}_\zeta$. We have
\begin{align}\label{can:Mtilde}
\begin{split}
\mc T'_f &= \widetilde{M}_f+\sum_{g\prec f}\texttt{t}'_{gf}(q)\widetilde{M}_g\\
&=\sum_{x\in W^f}q^{\ell(^fw_0)-\ell(x)}M_{fx} + \sum_{g\prec f}\sum_{y\in W^g} \texttt{t}'_{gf}(q)q^{\ell(^gw_0)-\ell(y)}M_{gx}.
\end{split}
\end{align}
Since $\texttt{t}'_{gf}(q)\in q\Z[q]$, it follows that all coefficients lie in $q\Z[q]$, except for the coefficient of $M_{f ^fw_0}=M_{fw^\zeta_0}$. Since it is bar-invariant, this implies, by the uniqueness property of canonical basis, that $\mc T'_f = T_{fw_0^\zeta}$, i.e., it is the canonical basis element of $\widehat{\mathbb T}^{m|n}$ corresponding to $fw_0^\zeta$.

Now, if we write
\begin{align*}
T_{fw^\zeta_0}=M_{fw_0^\zeta} + \sum_{g' \prec fw_0^\zeta} t_{g'fw_0^\zeta}(q) M_{g'},
\end{align*}
then, comparing with \eqref{can:Mtilde}, we conclude that
\begin{align*}
\texttt{t}'_{gf} = t_{g\cdot w_0^\zeta,f\cdot w_0^\zeta},\qquad f,g\in\Z^{m|n}_{\zeta-}.
\end{align*}

We have the following.
\begin{prop}\label{pi:can:basis}
Let $f\in\Z^{m|n}_{\zeta-}$. Recalling the definition \eqref{def:Nf}, we have
$$T_{fw_0^\zeta}S_\zeta=\mc T'_f S_\zeta=N_f+\sum_{g\prec f}t_{g\cdot w_0^\zeta,f\cdot w_0^\zeta}N_g,\quad g\in\Z^{m|n}_{\zeta-}.$$
In particular, $\mc T'_fS_\zeta=\mc T_f$ and we have $\texttt{t}_{gf} = t_{g\cdot w_0^\zeta,f\cdot w_0^\zeta}$, for $f,g\in\Z^{m|n}_{\zeta-}$.
\end{prop}
\begin{proof}
We compute
\begin{align*}
T_{fw_0^\zeta}S_\zeta=&{\mc T'_f} S_\zeta=\widetilde{M}_f S_\zeta+\sum_{g\prec f}\texttt{t}'_{gf}\widetilde{M}_g S_\zeta
=\frac{1}{[|W_f|]}{M}_fS^2_\zeta+\sum_{g\prec f}\texttt{t}'_{gf}\frac{1}{[|W_g|]}{M}_g S^2_\zeta\\
=&\frac{1}{[|W_f|]}{M}_fS^2_\zeta+\sum_{g\prec f}\texttt{t}'_{gf}\frac{1}{[|W_g|]}{M}_g S^2_\zeta
=\frac{[|W_\zeta|]}{[|W_f|]}{M}_fS_\zeta+\sum_{g\prec f}\texttt{t}'_{gf}\frac{[|W_\zeta|]}{[|W_g|]}{M}_g S_\zeta\\
=&\frac{[|W_\zeta|]}{[|W_f|]}\widetilde{N}_f+\sum_{g\prec f}\texttt{t}'_{gf}\frac{[|W_\zeta|]}{[|W_g|]}\widetilde{N}_g=N_f+\sum_{g\prec f}\texttt{t}'_{gf}N_g =N_f+\sum_{g\prec f}t_{g\cdot w_0^\zeta,f\cdot w_0^\zeta}N_g.
\end{align*}
\,
The identity $\mc T'_fS_\zeta=\mc T_f$ follows from the bar-invariance of $\mc T'_fS_\zeta$ and the characterization of the canonical basis element $\mc T_f$.
\end{proof}

Now let $L_f$ be the dual canonical basis element in $\mathbb T^{m|n}$ corresponding to a $W_\zeta$-anti-dominant weight, i.e., $f\in\Z^{m|n}_{\zeta-}$. We write
\begin{align*}
L_f=M_f+\sum_{g\prec f}\ell_{gf}(q) M_g,\qquad \ell_{gf}(q)\in q^{-1}\Z[q^{-1}].
\end{align*}
By \eqref{eq:MfS}, we have
\begin{align*}
L_fS_\zeta&=M_fS_\zeta+\sum_{g\prec f}\ell_{gf}(q) M_gS_\zeta\\
&= \widetilde{N}_f + \sum_{g\prec f;g\in\Z^{m|n}}\ell_{gf}(q)q^{-\ell(\tau_g)}\widetilde{N}_{g\cdot\tau(g)},
\end{align*}
where $\tau_g\in W^g$ is such that $g\cdot\tau_g\in\Z^{m|n}_{\zeta-}$. Since $L_fS_\zeta$ is bar-invariant, it follows by uniqueness of dual canonical basis that $L_fS_\zeta = \mc{L}_f$, for $f\in\Z^{m|n}_{\zeta-}$. A similar calculation as before for the canonical basis, shows that
\begin{align}\label{formula:ll}
\texttt{l}_{gf} = \sum_{x\in W^g}\ell_{g\cdot x,f}q^{-\ell(x)},\qquad f,g\in\Z^{m|n}_{\zeta-}.
\end{align}

Now suppose that $f\in\Z^{m|n}$, but not in $\zmnz$. Let $e\not=\tau\in W_\zeta$ be the shortest length element such $f\cdot \tau\in\zmnz$ and write the corresponding canonical basis element in $\mathbb T^{m|n}$ as
\begin{align*}
L_{f\cdot\tau}=M_{f\cdot\tau}+\sum_{g\prec f\cdot x}\ell_{g,f\cdot\tau}M_g.
\end{align*}
Applying $S_\zeta$ to the right, we get a bar-invariant element in $\mathbb T^{m|n}_\zeta$ of the form
\begin{align*}
L_{f\cdot\tau} S_\zeta=q^{-\ell(\tau)}\widetilde{N}_{f\cdot\tau}+\sum_{g\prec f\cdot\tau}\ell_{g,f\cdot\tau}M_gS_\zeta,
\end{align*}
which, by \eqref{eq:MfS}, lies in the $q^{-1}\Z[{q^{-1}}]$-span of the $\widetilde{N}_g$, $g\in\zmnz$. But such an expression can only be bar-invariant, if it is zero.

Let $\phi_\zeta:\widehat{\mathbb T}^{m|n}\rightarrow\widehat{\mathbb T}^{m|n}_\zeta$ be the $U_q(\gl_\infty)$-module homomorphism defined by
\begin{align}\label{def:phiz}
\phi_\zeta(M_f):=M_fS_\zeta,\qquad f\in\Z^{m|n}.
\end{align}
We summarize our discussion in the following.

\begin{thm}\label{thm:szeta:fock} Let $\phi_\zeta:\widehat{\mathbb T}^{m|n}\rightarrow\widehat{\mathbb T}^{m|n}_\zeta$ be the map in \eqref{def:phiz}. For $f\in\Z^{m|n}$, let $\tau$ be the minimum length element in $W_\zeta$ such that $f\cdot\tau\in\zmnz$. Then we have
\begin{itemize}
\item[(1)] $\phi_\zeta\left( M_f\right)=q^{-\ell(\tau)}\widetilde{N}_{f\cdot\tau}$ and $\phi(\widetilde{M}_f)=q^{-\ell(\tau)}N_f$.
\item[(2)] $\phi_\zeta(T_{f\cdot w_0^\zeta})=\mc T_f$, for $f\in\Z^{m|n}_{\zeta-}$.
\item[(3)] $\phi_\zeta\left(L_f\right)=\begin{cases}
\mc L_f,&\text{ if } f\in\zmnz,\\
0,&\text{ otherwise.}
\end{cases}$
\end{itemize}
\end{thm}

\section{Equivalence of categories} \label{sect::4}

In this section, we establish equivalences between several abelian categories of $\g$-modules.  These categories will provide categorifications of $q$-symmetric tensors introduced in Section \ref{sect::3}.  Recall that $\nu$ denotes a dominant and integral weight.
Also, recall the set $\Lambda(\nu)$ from Section~\ref{sec:lambdanu}.

\subsection{Projectively presentable modules} \label{sect::41}

A projective module $Q \in \mc O$ is said to be $\nu$-{\em admissible} if
each indecomposable direct summand of $Q$ is isomorphic to
$P(\la)$, for some  $\la \in \Lnua$.  Let $\Opres$ denote the full subcategory of $\mc O$ consisting of modules $M$ that have  a two step presentation of the form
\begin{align*}
&Q_2\rightarrow Q_1\rightarrow M\rightarrow 0,
\end{align*}
where both $Q_1$ and $Q_2$ are $\nu$-admissible projectives in $\mc O$. Similarly, we define the full subcategory $\mc O^{\vpre}_\oa$ of $\mc O_\oa$. The category $\Opres_\oa$ has been studied in detail in \cite{MaSt04}, see also the recent papers \cite{MaPW20,KoMa21}.

The category $\Opres$ is closed under taking cokernels but it is not closed
under taking kernels. However, $\Opres$ admits the structure of an abelian
category via an equivalence to a certain Serre quotient category
which we introduce in Section~\ref{Subsect::Fro}; see also \cite[Section~2.3]{MaSt04} for the abelian structure of $\Opres_\oa$.

\subsection{Serre quotients}
\subsubsection{General definition}
For a given abelian category $\mc C$ with Serre subcategory $\mc B$, we recall the {\em Serre quotient category} $\mc C/\mc B$ from  \cite[Chapitre III]{Gabriel}. Namely,  the  objects  in $\mc C/\mc B$ are the same as those of $\mc C$, and, for $X,Y\in \mc C$, the morphisms from $X$ to $Y$ in $\mc C/\mc B$ are defined as
\begin{align*}
&\Hom_{\mc C/\mc B}(X,Y):= \lim\limits_{\longrightarrow}\Hom_{\mc C}(X', Y/Y'),
\end{align*}    where $X'$ (respectively $Y'$) runs over all subobjects of $X$ (respectively $Y$) such that $X/X', Y'\in \mc B$. We refer to \cite{Gabriel} for the definition of the composition of two morphisms in $\mc C/\mc B$.

By \cite[Proposition III.1.1]{Gabriel}, the category $\mc C/\mc B$ is abelian and there is an exact canonical quotient functor $Q: \mc C\rightarrow \mc C/\mc B$ such that
\begin{itemize}
\item $Q$ is the identity on objects,
\item $Q$ acts on morphisms from $X$ to $Y$ via the canonical map
from $\Hom_{\mc C}(X,Y)$ to $\lim\limits_{\longrightarrow}\Hom_{\mc C}(X', Y/Y')$.
\end{itemize}
We record the following universal property of Serre quotient categories (see also
\cite[Lemma~A.1.2]{CP19}).

\begin{lem}\label{lem::unipro}
	\cite[Corollaires III.1.2 and III.1.3]{Gabriel}
	Let $\mc D$ be an abelian category with an exact functor $F: \mc C\rightarrow \mc D$ such that $F(X) = 0$ for any $X\in \mc B$. Then there is a unique exact functor $F': \mc C/\mc B\rightarrow \mc D$ such that $F= F'\circ Q$. Moreover, for any $f\in \Hom_{\mc C/\mc B}(X,Y)$ we have $F'(f)=F(g)$ if $f$ is represented by $g\in \Hom_{\mc C}(X',Y/Y')$.
\end{lem}

\subsubsection{Serre quotients of $\mc O$}

Let $\mc I^\nu$  be the Serre subcategory of $\mc O$ generated by all simple modules $L(\la)$, where $\la \in \Lambda \backslash\Lnua$.  From now on, we denote
by $\ov{\mc O}:=\mc O/\mc I^\nu$ the corresponding  Serre quotient of $\mc O$ and
by $\pi\equiv\pi^{\nu}: \mc O\rightarrow \ov{\mc O}$ the corresponding  canonical
functor. Similarly, we can consider the Serre subcategory $\mc I^\nu_\oa$  of
$\mc O_\oa$, the corresponding Serre quotient
$\overline{\mc O}_\oa:= \mc O_\oa/\mc I^\nu_\oa$ and the canonical
functor $\pi_0: \mc O_\oa \rightarrow \ov{\mc O}_\oa$.  We remark that $\ov{\mc O}$ and $\ov{\mc O}_\oa$ depend on the choice of $\nu$.

The following lemma establishes that $\mc O^{\vpre}$ has the natural structure
of an abelian category induced by an equivalence with $\ov{\mc O}$.

\begin{lem} \label{lem01}
	We have an equivalence $$\pi: \mc O^{\vpre}\xrightarrow{\cong} \ov{\mc O}.$$ In particular, $\{\pi(L(\la))|~\la\in \Lnua\}$  is an exhaustive list of simple objects in $\ov{\mc O}$.  Furthermore, $\pi(P(\la))$ is the indecomposble projective cover of $\pi(L(\la))$ in $\ov{\mc O}$, for any $\la \in \Lnua$.
\end{lem}
\begin{proof} We first show that $\pi$ is full and faithful.
	Take $X,Y\in \mc O^{\vpre}$ with a two step presentation
	\begin{align}
	&Q_1\rightarrow Q_2\rightarrow X\rightarrow 0, \label{eq::22}
	\end{align} where $Q_1,Q_2$ are {$\nu$-}admissible projective modules in $\mc O$.   Let $Y''$ be a submodule of $Y$ such that $Y''\in \mc  I^\nu$. Since $\Hom_{\mc O}(Q_1,Y'') = \Hom_{\mc O}(Q_2,Y'') =0$,  we have
\begin{align*}
	&\Hom_{\mc O}(Q_1,Y)= \Hom_{\mc O}(Q_1,Y/Y'')\quad\text{ and }\quad\Hom_{\mc O}(Q_2,Y)= \Hom_{\mc O}(Q_2,Y/Y'').
\end{align*} Applying the functors $\Hom_{\mc O}(-,Y)$ and $\Hom_{\mc O}(-,Y/Y'')$ to the sequence
\eqref{eq::22}, we obtain that $\Hom_{\mc O}(X,Y)= \Hom_{\mc O}(X,Y/Y'')$ by the Five lemma.  Consequently,   the map $\pi: \Hom_{\mc O}(X,Y) \rightarrow\Hom_{\ov{\mc O}}(\pi(X),\pi(Y))$ is an isomorphism.

Next, we show that $\pi$ is essentially surjective. For a given module $X\in \mc O$, let $X'$ be the minimal, with respect to inclusions, submodule of $X$ such that
$X/X'\in \mc I^\nu$. This implies that the top of $X'$ is a direct sum of simple modules whose  highest weights belong to $\Lnua$. Also, we have the following
short exact sequence  in $\ov{\mc O}$:
\begin{align*}
	&0\rightarrow \pi(X') \rightarrow \pi(X)\rightarrow \pi(X/X')\rightarrow 0.
\end{align*}
We have $\pi(X)\cong \pi(X')$ since $\pi$ is exact and $\pi(X/X')=0$.

Let $f: P\onto X'$ be an ($\nu$-admissible) projective cover of $X'$ in $\mc O$.
There is a submodule $K$ of the kernel $\ker(f)$ of $f$ such that $\pi(K)\cong \pi(\ker(f))$ and the top of $K$ is a direct sum of simple modules whose highest weights are in $\Lnua$. Let $g:Q\onto K$ be the projective cover of $K$ in $\mc O.$ This leads to the following  right exact sequence in $\ov{\mc O}$:
\begin{align*}
	&\pi(Q)\xrightarrow{\pi(g)}\pi(P)\rightarrow \pi(P/K)\rightarrow 0.
\end{align*} Since  $\pi(\ker(f))/\pi(K) \cong \pi(\ker(f)/K)=0$, it follows that we have the following right exact sequence $\ov{\mc O}$:
\begin{align*}
&\pi(Q)\xrightarrow{\pi(g)}\pi(P)\xrightarrow{\pi(f)} \pi(X')\rightarrow 0.
\end{align*} Consequently, we have $\pi(X) \cong \pi(P/K)$ and $P/K\in \mc O^{\vpre}$. This shows that $\pi$ is an equivalence.

 Since $\End_{\ov{\mc O}}(\pi(P(\la)))\cong \End_{\mc O^{\vpre}}(P(\la))$, for any $\la \in \Lnua$, is a local ring, it follows that $\pi(P(\la))$ is indecomposable in $\ov{\mc O}$ by \cite[Proposition 5.4]{Kr15}. Finally, it follows from \cite[Lemma A.1.3]{CP19} that $\pi(P(\la))$ is projective in $\ov{\mc O}$.  This completes the proof.
\end{proof}

For any $\la \in \Lnua$, we denote by $S(\la)\in \mc O$ the quotient of  $P(\la)$ by the sum of the images of all homomorphisms $P(\mu)\rightarrow \rad\, P(\la)$, for $\mu \in \Lnua$. Then the set $$\{\pi(S(\la))|~\la\in \Lnua\} = \{\pi(L(\la))|~\la\in \Lnua\}$$ is a complete and irredundant set of representatives of isomorphism classes of  simple objects in $\ov{\mc O}$; {see also Appendix \ref{sect::81}.}

\begin{prop} \label{prop::leftadj}
  The quotient functor $\pi$ has  both a left adjoint $\pi^L$ and a right adjoint $\pi^R$ such that  for every $V\in \ov{\mc O}$ the evaluation $V\rightarrow \pi \circ \pi^L(V)$ of the adjunction map  at $V$ is an isomorphism.
\end{prop}
\begin{proof}
	 First, we prove the existence of the right adjoint $\pi^R$ of $\pi$. An injective module $Q\in \mc O$ is said to be {$\nu$-}admissible if it is a direct sum of injective hulls of $L(\la)$ with $\la \in \Lnua$. We put $\mc O^{\nu-copres}$ to be the category  of injectively copresentable modules, that is, $\mc O^{\nu-copres}$ is the full subcategory of $\mc O$ consisting of modules that have a two step copresentation by {$\nu$-}admissible injective modules. By \cite[Lemma 2.1]{GL}, the quotient functor $\pi$ restricts to a full embedding $\pi|_{\mc O^{\nu-copres}}$  from $\mc O^{\nu-copres}$ to $\OI$. We claim that this full  embedding is essentially surjective. To see this, for any $M\in \mc O$, we consider its minimal quotient $M/M'$  by a submodule $M'$ such that $M'\in \mc I_\nu$. This implies that $M_1:=M/M'$ can be embedded into an {$\nu$-}admissible injective module $I_1$ in $\mc O$. Next, we consider the submodule $M_2$ of $I_1$ which contains $M_1$ and is maximal subject to the condition that  $M_2/M_1\in \mc I_\nu$. Set $I_2$ to be the injective hull of $I_1/M_2$ in $\mc O$, then we obtain a short exact sequence $0\rightarrow M_2\rightarrow I_1\rightarrow I_2$ in $\mc O$ such that $\pi(M_2)=\pi(M)$ and both $I_1,I_2$ are {$\nu$-}admissible injective modules. This shows that $\pi|_{\mc O^{\nu-copres}}: \mc O^{\nu-copres}\rightarrow \OI$ is an equivalence. As a consequence, the quotient functor $\pi$ admits a right adjoint by \cite[Proposition 2.2]{GL}.
		
	 We provide a self-contained proof of the existence of the left adjoint $\pi^L$ of $\pi$ as follows. Let $\iota:\Opres \hookrightarrow \mc O$ be the inclusion functor. We define $$\pi^L:=  \ov{\mc O}\xrightarrow{\cong}\Opres \xrightarrow{\iota}\mc O,$$
	namely, $\pi^L$ is the composition of $\iota$ and the inverse of the equivalence given in Lemma~\ref{lem01}.
   For any $M\in \mc O$, we let $M'$ be the minimal submodule of $M$ such that $M/M'\in \mc I_\nu$. Set $P^M\xrightarrow{\tau^M} M'$ to be the projective cover of $M'$ and let $K^M$ to be the submodule of $\ker(\tau^M)$ such that $\pi(K^M)\cong \pi(\ker(\tau^M))$ and the top  of $K^M$ is a direct sum of simple modules with highest weights in $\Lnua$. Let $X\in \OI$. By the proof of Lemma \ref{lem01},
  we know that
  \begin{align}
  &\pi(P^M/K^M)\cong \pi(M), \label{eq::47}
  \end{align}
   and  the natural epimorphism $P^M/K^M\rightarrow M'$ induces an isomorphism \begin{align}
  & \Hom_{\mc O}(\pi^L(X), P^M/K^M)\cong  \Hom_{\mc O}(\pi^L(X), M'). \label{eq::48}
  \end{align}
  Therefore, we get isomorphisms
	\begin{align}
	&\Hom_{\mc O}(\pi^L(X), M) \label{eq::49}\\
	&\cong \Hom_{\mc O}(\pi^L(X), M')~\text{ since } \pi^L(X)\in \Opres \\
	&\cong \Hom_{\mc O}(\pi^L(X), P^M/K^M)~\text{ by }\eqref{eq::48} \\
	&\cong \Hom_{\OI}(X, \pi(P^M/K^M))~\text{ by Lemma }\ref{lem01}.\\
	&\cong \Hom_{\OI}(X, \pi(M))~\text{ by }\eqref{eq::47}.\label{eq::413}
	\end{align}
 Let us denote the composition of these isomorphisms by $\eta_{X,M}: \Hom_{\mc O}(\pi^L(X), M)\xrightarrow{\cong}\Hom_{\OI}(X, \pi(M)).$ We shall show that $\eta_{X,M}$ gives rise to an adjunction between $(\pi^L, \pi)$. Since $\eta_{X,M}$ is natural in the variable $X\in \ov{\mc O}$, it suffices to show that $\eta_{X,M}$ is natural in the variable $M$. To see this, we let $f:M\rightarrow N$ be a homomorphism in $\mc O$ and $N'$ be the analogue of $M'$, that is,  $N'$ is the minimal, with respect to inclusions, submodule of $N$ such that $N/N'\in \mc I_\nu$. Note that $f$ induces a homomorphism from $M'$ to $N'$, which we will denote by $f'$. We consider the following diagram
	\begin{align}
	&\xymatrixcolsep{2pc} \xymatrix{
	\Hom_{\mc O}(\pi^L(X), M) \ar[r]^-{(1)}  \ar@<-2pt>[d]_{(2)} &  \Hom_{\mc O}(\pi^L(X), N) \ar@<-2pt>[d]_{(3)} \\
	\Hom_{\mc O}(\pi^L(X), M') \ar[r]^{(4)}  \ar@<-2pt>[d]_{(5)} & \Hom_{\mc O}(\pi^L(X), N')  \ar@<-2pt>[d]_{(6)} \\
  \Hom_{\mc O}(\pi^L(X), P^M/K^M) \ar[r]^{(7)} \ar@<-2pt>[d]_{(8)} & \Hom_{\mc O}(\pi^L(X), P^N/K^N) \ar@<-2pt>[d]_{(9)}\\
\Hom_{\OI}(X, \pi(P^M/K^M)) \ar[r]^{(10)} \ar@<-2pt>[d]_{(11)} & \Hom_{\OI}(X, \pi(P^N/K^N)) \ar@<-2pt>[d]_{(12)}\\
 \Hom_{\OI}(X, \pi(M)) \ar[r]^{(13)} & \Hom_{\OI}(X, \pi(N))} \label{eq::dia414}
	\end{align}
	
We shall explain the homomorphisms $(1)$-$(13)$ in the diagram \eqref{eq::dia414} and prove that they make this diagram commutative. First, the vertical homomorphisms in this diagram are given in \eqref{eq::49}-\eqref{eq::413}. Next, the homomorphisms $(1)$ and $(4)$ are induced by $f$, while $(2)$ and $(3)$ are given by restricting homomorphisms to $M'$ and  $N'$, respectively. Therefore, the first square is commutative.

Let us explain homomorphisms $(4)$-$(7)$ in the second square. The $f': M'\rightarrow N'$ lifts to a homomorphism from $f'': P^M\rightarrow P^N$ make the following diagram commutes:
  	\begin{align*}
  &\xymatrixcolsep{2pc} \xymatrix{
  	 P^M   \ar@<-2pt>[d]  \ar[r]^{f''} &  P^N \ar@<-2pt>[d] \\
  	M' \ar[r]^{f'}    & N'  }
  \end{align*}
  We note that $f''(K^M)$ is contained in the kernel of the canonical map $P^N\rightarrow N'$. Since the top of $f''(K^M)$ is an epimorphic  image of an $\nu$-admissible projective module, it follows that $f''(K^M)\subseteq K^N$. Therefore, we get a commutative diagram:
  	\begin{align}
  &\xymatrixcolsep{2pc} \xymatrix{
  	P^M/K^M   \ar@<-2pt>[d]  \ar[r]  &  P^N/K^N \ar@<-2pt>[d] \\
  	M' \ar[r]^{f'}    & N'  } \label{eq::dia416}
  \end{align} This makes the second square in \eqref{eq::dia414} commutes.

Finally, the homomorphisms $(7)$-$(10)$ in the third square form a commutative diagram by Lemma \ref{lem01}. Also, the last square commutes by the commutativity of the square \eqref{eq::dia416}. This completes the proof.
\end{proof}

\begin{rem} \label{rem::14} The construction of $\pi^R$ and $\pi^L$ are dual to each other. In particular, if $\mc O$ admits a simple-preserving duality $D$ then we have  $\pi^R\circ\pi(M)\cong D\circ \pi^L\circ \pi \circ D(M)$, for $M\in \mc O$.
\end{rem}

\subsection{Harish-Chandra ($\g,\g_\oa$)-bimodules} \label{Sect::HCbimod}

For a given $(\g,\g_\oa)$-bimodule $M$, we let $M^\text{ad}$ denote the left $\g_\oa$-module under the adjoint action of $\g_\oa$ given by
\begin{align}
&x \cdot m = xm -mx, \label{eqadjcation}
\end{align} for any $x\in \g_\oa$ and $m\in M$.

Let $\mc B$ denote the category of finitely generated $(\g,\g_\oa)$-bimodules $M$ for which
$M^\text{ad}$ is a (possibly infinite) direct sum of modules in $\mc F_\oa$.  Recall that,
for $\la \in \h^\ast$, we denote by $\chi_\la$ (respectively $\chi_\la^{\g_\oa}$) the  central character of $\mf g$ (respectively of $\g_\oa$) associated with $\la$.  We let $\mc B_\nu$ denote the full subcategory of $\mc B$ consistsing of all $M$ such that $M\ker(\chi_\nu^{\g_\oa})=0$.\footnote{
We emphasize that the notation $\mc B_\nu$ used here has different meaning compared to the same notation used in \cite{Ch21}, but it agrees with the one defined in \cite{Ch212}. In \cite{Ch21}, $\mc B_\nu$ is used to denote the full subcategory of $\mc B$ consisting of modules $X$ satisfying $X\ker(\chi_\nu^{\g_\oa})^n=0$ for $n\gg 0$.}

For any $\g_\oa$-module $M$ and any $\g$-module $N$, we let $\mc L(M, N)$ be the maximal $(\g,\g_\oa)$-submodule of $\Hom_\C(M, N)$ that is a direct sum of modules in $\mc F_\oa$ with respect to the adjoint action in \eqref{eqadjcation}.

 If we denote by
$\g\text{-}\mathrm{Mod}\text{-}\g_\oa$ the category of all
$(\g,\g_\oa)$-bimodules, the previous paragraph gives a   left exact bifunctor
$$\mc L(-,-) :~ \mf g_\oa\text{-}\mathrm{Mod} \times \mf g\text{-}\mathrm{Mod}
\rightarrow \g\text{-}\mathrm{Mod}\text{-}\g_\oa.$$

	Let $M$ be  a $\fg_{\oa}$-module such that the annihilator ideal of $M$ in $U(\g_\oa)$ is generated by the kernel $\ker(\chi_\nu^{\g_\oa})$. Suppose that
	\begin{enumerate}[(A)]
		\item the canonical monomorphism $U(\g_\oa)/\Ann_{U(\g_\oa)}(M)\hookrightarrow \mc L(M,M)$ is surjective; \label{ThmMS1a}
		\item the module $M$ is projective in~$\langle \mc F_\oa\otimes M\rangle$. \label{ThmMS1b}
	\end{enumerate}
	Then, by \cite[Theorem 3.1]{CC}, we have an equivalence
 \begin{align*}
 &\mc L(M,-):~\text{Coker}(\mc F\otimes \Ind M)\xrightarrow{\cong} \mc B_\nu,
 \end{align*}
with inverse
\begin{align*}
 &-\otimes_{U(\g_\oa)}M:~\mc B_\nu\xrightarrow{\cong} \text{Coker}(\mc F\otimes \Ind M).
 \end{align*}

For example, the module $M_0(\nu)$ satisfies both conditions \eqref{ThmMS1a} and \eqref{ThmMS1b} leading to the following lemma, see \cite[Lemma 20]{Ch212}.

\begin{lem} \label{lem::03}
	We have the following mutually inverse equivalences
	\begin{align*}
	&\mc L(M_0(\nu),-):~\Opres\xrightarrow{\cong} \mc B_\nu,\\
	&-\otimes_{U(\g_\oa)}M_0(\nu):~ \mc B_\nu\xrightarrow{\cong} \Opres.
	\end{align*}
\end{lem}

\subsection{Whittaker modules}\label{sec:whitt}

We recall the category of Whittaker modules for Lie superalgebras studied
in \cite{Ch21}. Originally, Whittaker modules for Lie algebras were studied
by Kostant in \cite{Ko78} and further investigated in more detail in \cite{MS,Mc,Mc2, B, BM}.


\subsubsection{Whittaker category} \label{Sect::341} The category $\mc N$ in \cite{Ch21}, originally introduced in \cite{MS} in the case of Lie algebras, consists of finitely-generated $\g$-modules that are locally finite over $U(\mf n^+)$ and over the center $Z(\g_\oa)$ of $U(\g_\oa)$. We then have a decomposition
\begin{align*}
	\mc N=\bigoplus_{\zeta\in {\bf ch} \mf n^+_\oa} \mc N(\zeta),
	\end{align*} where the full subcategory $\mc N(\zeta)$ consists of modules $M \in \mc N$ such that $x-\zeta(x)$ acts locally nilpotently on $M$, for any $x\in \mf n_\oa^+$.
	In the case $\g=\g_\oa$, the corresponding category is denoted by $\mc N_\oa$. \footnote{The notations $\mc N$ and $\mc N_\oa$ used here are different from \cite{Ch21,Ch212},  where the corresponding categories are denoted by $\widetilde{\mc N}$ and $\mc N$.}

\subsubsection{Simple and standard Whittaker modules}
Classification of simple Whittaker modules over reductive Lie algebras was completed by Mili{\v{c}}i{\'c} and Soergel in \cite{MS}. Such modules are classified by means of the so-called {\em standard Whittaker modules}.  An analogue of these modules for Lie superalgebras is considered in \cite{Ch21}.

  In \cite{Ko78}, Kostant constructed the following Whittaker modules: 
\begin{align} \label{eq::31}
	&Y_\zeta(\la, \zeta):=U(\mf l_\zeta)/\text{Ker}(\chi^{\mf l_\zeta}_\la) U(\mf l_\zeta) \otimes_{U(\mf n_\zeta)}\mathbb C_\zeta,
\end{align} where  $\chi^{\mf l_\zeta}_\la$ is the central character of $\mf l_\zeta$ and  $\C_\zeta$ is the one-dimensional module over the nilradical $\mf n_\zeta\subset \mf l_\zeta$ associated with the characters $\la$ and $\zeta$, respectively.  Note that our notation
$Y_\zeta(\la,\zeta)$ corresponds to  $Y_{\chi^{\g_\oa}_\la, \zeta}$ in \cite[Section~3.6]{Ko78}. The standard Whittaker modules over $\g_\oa$ and $\g$ are defined as follows:
\begin{align*}
&M_0(\la, \zeta):= U(\g_\oa)\otimes_{\mf p_\oa} Y_\zeta(\la, \zeta)
\quad\text{ and }\quad M(\la, \zeta):= {U}(\g)\otimes_{\mf p} Y_\zeta(\la, \zeta).
\end{align*} These standard Whittaker modules are objects in the category $\mc N_\oa$ and $\mc N$ defined in Section \ref{Sect::341}; see also \cite{Mc,MS, B, Ch21}. \footnote{In the present paper, we use the $0$-subscript convention to denote the respective $\g_\oa$-modules. We emphasize that the notations $M(\la,\zeta), M_0(\la,\zeta)$ and $L(\la,\zeta)$ correspond to the notations $\widetilde{M}(\la,\zeta), M(\la,\zeta)$ and $\widetilde{L}(\la,\zeta)$ used in \cite{Ch21,Ch212}.}

The following result is taken from \cite[Proposition 2.1]{MS} and \cite[Theorem 6]{Ch21}.
\begin{lem}\label{lem::MS21}
	We have the following:
	\begin{itemize}
		\item[(1)]
		For any $\la \in \mf h^\ast$,  $M(\la, \zeta)$ has simple top. This simple top is denoted by $L(\la,\zeta)$.
		\item[(2)] The set $\{{L}(\la, \zeta)|~\la \in \h^\ast\}$ is a complete and irredundant set of representatives of the isomorphism classes of simple objects in ${\mc N}(\zeta)$.
		\item[(3)] 	Let $\mu \in \h^\ast$, then $${L}(\la, \zeta)\cong  {L}(\mu, \zeta)\Leftrightarrow {M}(\la, \zeta)\cong {M}(\mu, \zeta)\Leftrightarrow W_\zeta \cdot \la =W_\zeta \cdot \mu.$$
	\end{itemize}
	
\end{lem}


\subsubsection{A Mili{\v{c}}i{\'c}-Soergel type equivalence}\label{sec:Wzeta} Set $$\mc W(\zeta):= \mathrm{Coker}(\cF\otimes \mathrm{Ind}(M_0(\nu,\zeta))).$$ 
In the case that $\zeta=0$, we have $\mc W(0)=\mc O$ by \cite[Corollary 3.2]{CC}. By \cite[Theorem 5.3, Lemma 5.11 and Proposition 5.5]{MS}, the standard Whittaker module $M_0(\nu, \zeta)$ satisfies both conditions \eqref{ThmMS1a} and \eqref{ThmMS1b} in Section \ref{Sect::HCbimod}.  Therefore, we have the following result, which  is taken from \cite[Theorem 16, Proposition 33]{Ch21}:
\begin{lem} \label{lem::5}
	There is an equivalence
	\begin{align}
		& -\otimes_{U(\g_\oa)} M_0(\nu,\zeta):~\mc B_\nu\xrightarrow{\cong} \mc W(\zeta), \label{eq::415}
	\end{align}
sending $\mc L(M_0(\nu), M(\la))$ to $M(\la, \zeta)$ and the top of $\mc L(M_0(\nu), M(\la))$ to $L(\la,\zeta)$, for any $\la \in \Lambda$.
\end{lem}
In the case that $\g=\g_\oa$, the equivalence \eqref{eq::415} was proved in \cite[Section 5]{MS}.  If $\g=\g_\oa$ and $\zeta=0$,  the equivalence  \eqref{eq::415}  goes back to \cite{BG}.
Lemma \ref{lem::5}  provides a connection between Harish-Chandra ($\g,\g_\oa$)-bimodules and Whittaker modules. We will show that the category $\ov{\mc O}$ admits a {\em properly stratified structure} in the sense of \cite[Section 2.6]{MaSt04}. Using Lemmas \ref{lem01}, \ref{lem::03} and \ref{lem::5}, we conclude that the category  $\mc W(\zeta)$ admits a properly stratified structure as well.

\subsection{Frobenius extensions} \label{Subsect::Fro}

In this subsection, we let $\g_\oa\subseteq \mf a\subseteq \g$ be a subalgebra such that it inherits a type-I grading in \eqref{typeI} with a Borel subalgebra $\mf b^{\mf a}$ such that $\mf b^{\mf a}_\oa =\mf b_\oa$. Let $\mc O(\mf a)\equiv \mc O(\mf a, \mf b^{\mf a})$.
We define the full subcategories $\mc O^{\vpre}(\mf a)$ and $\mc I^\nu(\mf a)$ of $\mc O(\mf a)$ in similar fashion. By the same argument as in Lemma \ref{lem01}, one can show that $\mc O^{\vpre}(\mf a)\cong \ov{\mc O}(\mf a):=\mc O(\mf a)/\mc I^\nu(\mf a)$, and so $\mc O(\mf a)^{\vpre}$ admits the structure of an abelian category.

\begin{lem}  \label{lem::1}
	 The induction functor $\Ind_{\mf a}^{\g}$, the co-induction functor $\Coind_{\mf a}^{\g}$ and the restriction functor $\Res_{\mf a}^{\g}$ give rise to two adjoint pairs of exact functors,
	\begin{align}
		&(\Ind_{\mf a}^{\g},~\Res_{\mf a}^{\g})\quad\text{ and }\quad(\Res_{\mf a}^{\g},~\Coind_{\mf a}^{\g}), \label{eq::338}
	\end{align} between $\mc O^{\nu\text{-pres}}(\mf a)$ and $\Opres$. Moreover,
\begin{align}
	&\Ind_{\mf a}^{\mf g}(-)\;\cong\; \Coind_{\mf a}^{\g}(-)\circ (\wedge^{\dim(\mf g/\mf a)}(\mf g/\mf a)\otimes-). \label{eq::339}
\end{align} 
\end{lem}
\begin{proof} 	First, let us argue that both, $\Ind_{\mf a}^{\g}(-), \Coind_{\mf a}^{\g}(-): \mc O^{\vpre}(\mf a)\rightarrow  \mc O^{\vpre}$  and $\Res_{\mf a}^{\g}(-):\mc O^{\vpre}\rightarrow  \mc O^{\vpre}(\mf a)$, are well-defined functors. To see this, recall that $L^{\mf a}(\la)$ and $P^{\mf a}(\la)$ denote the irreducible module of highest weight $\la$ and its projective cover in the category $\mc O(\mf a)$, respectively. Suppose that $\la \in \Lnua$. Then $L^{\mf a}(\la)$ is $\alpha$-free, for $\alpha \in {\Pi_\nu}$. Suppose that $P(\gamma)$ is a direct summand of $\Ind_{\mf a}^\g P^{\mf a}(\la)$. Observe that
	\begin{align}
		&{\dim}\Hom_\g(\Ind_{\mf a}^\g P^{\mf a}(\la), L(\gamma)) = {\dim}\Hom_{\mf a}(P^{\mf a}(\la), \Res^\g_{\mf a} L(\gamma)) = [\Res_{\mf a}^\g L(\gamma): L^{\mf a}(\la)], \label{eq::26}
	\end{align} which implies that $L(\gamma)$ is $\alpha$-free, for all $\alpha\in \Pi_\nu$. Hence $\gamma\in\Lambda(\nu)$. This shows that both functors  $\Ind_{\mf a}^{\g}(-), \Coind_{\mf a}^{\g}(-): \mc O^{\vpre}(\mf a)\rightarrow  \mc O^{\vpre}$  are well-defined. Also,  if $P^{\mf a}(\gamma)$ is a direct summand of $\Res^\g_{\mf a} P(\la)$ then
	\begin{align}
	&{\dim}\Hom_{\mf a}(\Res_{\mf a}^\g P(\la), L^{\mf a}(\gamma)) = {\dim}\Hom_{\mf g}(P(\la), \Coind_{\mf a}^{\mf g} L^{\mf a}(\gamma)) = [\Coind_{\mf a}^\g L^{\mf a}(\gamma): L(\la)],
\end{align}which implies that the module $L^{\mf a}(\gamma)$ is $\alpha$-free, for all $\alpha\in {\Pi_\nu}$. This proves that the functor $\Res_{\mf a}^{\g}(-):\mc O^{\vpre}\rightarrow  \mc O^{\vpre}(\mf a)$ is well-defined as well.  Therefore, we obtain two adjoint pairs of functors  $(\Ind_{\mf a}^{\g}, \Res_{\mf a}^{\g})$ and $(\Res_{\mf a}^{\g}, \Coind_{\mf a}^{\g})$ between the abelian categories $\mc O^{\vpre}(\mf a)$ and $\mc O^{\vpre}$, as claimed. Equation \eqref{eq::339} follows from \eqref{eq::adj}.

Finally, the assertion about exactness of these functors follows from the adjunction \eqref{eq::339} and  \cite[Corollary 2.1.2]{Co}.  \end{proof}

By Lemma \ref{lem01}, the induction, the co-induction and the restriction functors induce two adjoint pairs of exact functors between $\ov{\mc O}(\mf a)$ and $\ov{\mc O}$. By slight abuse of notation, we still denote these functors by $\Ind_{\mf a}^{\g}(-), \Coind_{\mf a}^{\g}(-)$ and $\Res_{\mf a}^{\g}(-)$.
It follows from \eqref{eq::338} and \eqref{eq::339} that these functors give a structure of {\em Frobenius extension} between $\mc O^{\vpre}(\mf a)\cong \ov{\mc O}(\mf a)$ and  $\mc O^{\vpre}\cong \ov{\mc O}$ in the sense of \cite[Section 2]{Co}, see also \cite{BF}.

 \begin{cor}\label{cor::5} Let $M\in \ov{\mc O}(\mf a)$ and $N\in \ov{\mc O}$. Then,  for any $d\geq 0$, we have
 	\begin{align*}
 	&\Ext^d_{\ov{\mc O}}(\Ind_{\mf a}^{\g} M, N)\cong \Ext^d_{\ov{\mc O}(\mf a)}(M, \Res_{\mf a}^{\g} N),\\
 	&\Ext^d_{\ov{\mc O}(\mf a)}(\Res_{\mf a}^{\g} M, N)\cong \Ext^d_{\ov{\mc O}}(M, \Coind_{\mf a}^{\g} N).
 	\end{align*}
 \end{cor}
\begin{proof}
	The conclusion follows from Lemma \ref{lem::1} and \cite[Proposition 4]{CoM}.
\end{proof}

\begin{cor}
 \emph{$\ov{\mc O}$} has enough injective objects.
\end{cor}
\begin{proof} Since $\ov{\mc O}$ is a Frobenius extension of $\ov{\mc O}_\oa$ and the latter category has enough injective objects, the conclusion follows from \cite[Proposition 2.2.1]{Co}.
	\end{proof}

\section{Stratified structure of $\ov{\mc O}$} \label{sect::5}
From now on, we assume that $\mf l_\nu$ is the Levi subalgebra in a parabolic decomposition of $\g$; see Section \ref{sec:lambdanu}. We remark that this condition is always satisfied in the case when $\g=\gl(m|n), \mf{osp}(2|2n),\pn$; see, e.g.,  \cite[Section 3.3.1]{Ch21}.   By letting $\g_1$ and $\g_{-1}$ act trivially on $\g_\oa$-modules,
we define the {\em Kac functor} and the {\em dual Kac functor}, respectively, as follows:
\begin{align*}
&K(-):=\Ind_{\g_{\geq 0}}^{\g}(-), ~\widehat K(-):=\Coind_{\g_{\leq 0}}^{\g}(-): \mc O_\oa \rightarrow \mc O.
\end{align*} By the universal property in Lemma \ref{lem::unipro}, these two functors induce functors from $\ov{\mc O}_\oa$ to $\ov{\mc O}$. By a slight abuse of notations, we still denote these induced functors by $K(-)$ and $\widehat{K}(-)$. In this section, we study the stratified structure of $\ov{\mc O}$.

\subsection{Standard and costandard objects} \label{sect::51}


Consider the BGG category $\mc O(\mf l_\nu, \mf l_\nu\cap \mf b_\oa)$ of $\mf l_\nu$-modules with respect to the Borel subalgebra $\mf l_\nu\cap \mf b_\oa$.
For a given $\la \in \Lnua$, we set $$\overline\Delta_0'(\la):= \Ind_{\fp_\oa}^{\g_\oa}L(\mf l_\nu,\la)\quad\text{ and }\quad\Delta_0'(\la):= \Ind_{\fp_\oa}^{\g_\oa}P(\mf l_\nu,\la).$$
Here $L(\mf l_\nu,\la)$ denotes the simple $\mf l_\nu$-module of
highest weight $\la$  and  $P(\mf l_\nu,\la)$
denotes the indecomposable projective cover of $L(\mf l_\nu,\la)$
in $\mc O(\mf l_\nu, \mf l_\nu\cap \mf b_\oa)$. We note that $\ov{\Delta}'_0(\la)$ is the Verma module $M_0(\la)$.

Recall that we have the simple-preserving duality $(\cdot)^\vee$ on $\mc O_\oa$, mentioned in Section~\ref{sect::24}. Set $\nb_0'(\la):= (\Delta_0'(\la))^\vee$ and $\ov{\nb}_0'(\la):= (\ov{\Delta}_0'(\la))^\vee.$ We define  in $\ov{\mc O}$ the following
objects:
\begin{itemize}
\item the {\em standard} object $\Delta(\la):= \pi(\Ind_{\g_{\geq 0}}^\g \Delta_0'(\la))$,
\item the {\em proper standard} object
$\overline \Delta(\la):= \pi(\Ind_{\g_{\geq 0}}^\g \overline\Delta_0'(\la))$,
\item the {\em costandard} object
$\nabla(\la):= \pi(\Coind_{\g_{\leq 0}}^\g \nabla_0'(\la))$, and
\item the  {\em proper costandard} object
$\overline \nabla(\la):= \pi(\Coind_{\g_{\leq 0}}^\g \overline\nabla_0'(\la))$.
\end{itemize}
Set $$\Delta_0(\la):= \pi(\Delta'_0(\la)),~\ov{\Delta}_0(\la):= \pi(\ov{\Delta}'_0(\la)),~\ov{\nb}_0(\la):= \pi(\ov{\nb}'_0(\la))\,\text{ and }\,\Delta_0(\la):= \pi(\Delta'_0(\la)).$$ These objects are structural objects for the properly stratified structures described in \cite{FKM00, MaSt04}.  We have $\Delta(\la) \cong K(\Delta_0(\la)),~\nb(\la) \cong \widehat{K}(\nb_0(\la))$, etc.

We have the following natural inclusions and projections:
\begin{align*}
&\pi(L(\la)) \hookrightarrow \nb(\la), \ov{\nb}(\la),\\
&\Delta(\la), \ov{\Delta}(\la) \onto \pi(L(\la)).
\end{align*}

\begin{lem} \label{lem3}
 For any $\la,\mu \in \Lnua$ and a non-negative integer $d$, we have
 \begin{align}
 	&\Ext_{\ov{\mc O}}^d(\Delta(\la), \ov \nb(\mu)) =\begin{cases} \C, &\mbox{ if $\mu =\la$ and $d=0$;}\\
 	0,&\mbox{ otherwise.}
 	\end{cases} \label{lem3eq::1} \\
 	&\Ext_{\ov{\mc O}}^d(\ov \Delta(\la), \nb(\mu)) =\begin{cases} \C, &\mbox{ if $\mu =\la$ and $d=0$;}\\
 	0,&\mbox{ otherwise.}
 	\end{cases} \label{lem3eq::2}
 \end{align}
\end{lem}
\begin{proof}
	Let $\pi_{\leq 0}: \mc O(\g_{\leq 0})\rightarrow\ov{\mc O}(\g_{\leq 0})$ denote the quotient functor.  
 For any $d\geq 0$, using Corollary \ref{cor::5},	we calculate
\begin{align*}
	&\Ext_{\ov{\mc O}}^d(\Delta(\la), \ov \nb(\mu))\cong \Ext_{\ov{\mc O}(\g_{\leq 0})}^d(\Res_{\g_{\leq 0}}^\g(\Delta(\la)), \pi_{\leq 0}(\ov \nb_0'(\mu)))\\
    &\cong \Ext_{\ov{\mc O}(\g_{\leq 0})}^d(\pi_{\leq 0}(\Res_{\g_{\leq 0}}^\g \Ind_{\g_{\geq 0}}^\g\Delta_0'(\la)), \pi_{\leq 0}(\ov \nb_0'(\mu)))\\
	&\cong\Ext_{\ov{\mc O}(\g_{\leq 0})}^d(\Ind_{\g_\oa}^{\g_{\leq 0}}(\Delta_0(\la)), \pi_{\leq 0}(\ov \nb_0'(\mu)))\\
	&\cong\Ext_{\ov{\mc O}_\oa}^d(\Delta_0(\la),  \ov \nb_0(\mu)).
\end{align*} Now Equation~\eqref{lem3eq::1} follows from classical results; see \cite{MaSt04} and \cite{FKM00}. Equation \eqref{lem3eq::2} can be proved in similar fashion.
\end{proof}
\begin{prop} \label{prop4} For $\la,\mu\in \Lnua$, we have the following:
	\begin{itemize}
		\item[(1)] The standard object $\Delta(\la)$ has a filtration, all subquotients of which are isomorphic to $\ov \Delta(\la)$.
		\item[(2)]  $[\ov \Delta(\la): \pi(L(\la))]=1$.
		\item[(3)] $[\ov \Delta(\la): \pi(L(\mu))]>0$ implies that $\mu<\la$.
	\end{itemize}
\end{prop}
\begin{proof} For a reductive Lie algebra $\g=\g_\oa$,  Claims (1)-(3) are established in \cite[Section 2.6]{MaSt04}.
	
	Suppose that $\g$ is an  arbitrary quasireductive Lie superalgebra of type I. Claim (1) follows by applying the Kac functor $K(-)$. 
	Since $\ov \Delta(\la)= \pi(M(\la))$ in $\ov{\mc O}$,
	applying the quotient functor $\pi$ we have
\begin{align*}
&[\ov \Delta(\la): \pi(L(\mu))] = [M(\la): L(\mu)].
\end{align*}  Claims (2) and (3)  follow.
\end{proof}

We let $[\ov{\mc O}]$ denote the Grothedieck group of $\ov{\mc O}$. Then, by definition, the set of simple objects $\{[\pi(L(\la))]|~\la\in \Lnua\}$ forms a $\Z$-basis in $[\ov{\mc O}]$.  From Proposition \ref{prop4}, it follows that elements in the following two sets
\begin{align}
	&\{[\ov{\Delta}(\la)]|~\la\in \Lnua\}\quad\text{ and }\quad\{[\ov{\nb}(\la)]|~\la\in \Lnua\}, \label{eq::61}
\end{align} are $\Z$-independent in $[\ov{\mc O}]$, respectively.

We now describe the endomorphism algebra $\End_\OI(\Delta(\la))$ of the standard object $\Delta(\la)$. In the case of the reductive Lie algebra $\g=\g_\oa$ with a regular weight $\la\in \Lnua$,  this endomorphism algebra is isomorphic to the coinvariant algebra for $\mf l_\nu$ by  \cite[Theorem 6.1 and Remark 6.2]{MaSt08}:
	 \begin{cor}For any $\la\in \Lnua$, we have
		\begin{align*}
		&\End_{\OI}(\Delta(\la))\cong\End_{\mf l_\nu}(P(\mf l_\nu, \la)).
		\end{align*}
\end{cor}
\begin{proof} Set $\Delta'(\la):=K(\Delta_0'(\la)) = \Ind_{\mf p}^{\g} P(\mf l_\nu,\la)$. By the proof of Proposition \ref{prop::45} in Appendix \ref{sect::81},  $\Delta'(\la)$  is the quotient of $P(\la)$ modulo the sum of images of  homomorphisms $P(\mu)\rightarrow P(\la)$, for $\mu\in \Lnua$ with $\mu >\la$.  Therefore, we have  $\End_{\OI}(\Delta(\la)) = \End_{\mc O}(\Delta'(\la))$ by Lemma \ref{lem01}. Applying the parabolic induction, we get an inclusion
	\begin{align}
	&\End_{\mf l_\nu}(P(\mf l_\nu, \la)) \hookrightarrow \End_\OI(\Delta(\la)). \label{eq::59}
	\end{align}

On the other hand, applying the functor $\Hom_{\g}(-, \Delta'(\la))$ to the canonical quotient $P(\la)\rightarrow \Delta'(\la)$, we obtain
\begin{align*}
&\dim \End_\OI(\Delta(\la)) \leq \dim \Hom_{\g}(P(\la), \Delta'(\la)) =[\Delta'(\la): L(\la)],
\end{align*}
which coincides with the multiplicity  $[P(\mf l_\nu, \la):L(\mf l_\nu, \la)] =\dim\End_{\mf l_\nu}(P(\mf l_\nu, \la))$ by Proposition \ref{prop4}.
Consequently, the inclusion in \eqref{eq::59} is an isomorphism.
\end{proof}

\subsection{BGG reciprocity}

 An object $M$ in $\mc \OI$ is said to have a $\Delta$-flag  if $M$ has a filtration subquotients of which are isomorphic to objects in $\{\Delta(\la)|~\la\in \Lnua\}$.    Similarly, we define objects that admit $\ov \Delta$-flag, $\nb$-flag and $\ov \nb$-flag, respectively. We denote by $\mc F(  \Delta)$, $\mc F(\ov \Delta)$,  $\mc F(\nb)$ and $\mc F(\ov \nb)$ the full subcategories of $\OI$ of objects which admit a $\Delta$-flag,
 or a $\ov \Delta$-flag, or a $\nb$-flag or a  $\ov \nb$-flag, respectively. Similarly we define $\mc F(\Delta_0)$, $\mc F(\ov{\Delta}_0)$, and so on.  The following statement is
 the key homological property of stratified categories.

\begin{thm} \label{thm5}
	We have \begin{itemize}
		\item[(1)] $\mc F(\ov \Delta) = \{M\in \mc \OI|~\Ext^d_\OI(M, \mc F(\nb))=0, \,\text{for any }d\geq 1\}.$
		\item[(2)] $\mc F(\Delta) = \{M\in \mc \OI|~\Ext^d_\OI(M, \mc F(\ov \nb))=0, \,\text{for any }d\geq 1\}.$
	\end{itemize}
\end{thm}
To prove Theorem \ref{thm5}, we need the following lemma.
\begin{lem}\label{lem::8} Let $\la, \mu \in \Lnua$.  \begin{itemize}
		\item[(1)] Suppose that $$\Ext^1_\OI(M,\nb(\mu))=\Hom_{\OI}(M, \pi(L(\gamma))) =0,$$ for any $\gamma \leq \mu$. Then $\Ext^1_{\OI}(M, \pi(L(\mu))=0$.
		\item[(2)] Suppose that $$\Ext^1_\OI(M,\ov \nb(\mu))=\Hom_{\OI}(M, \pi(L(\gamma))) =0,$$ for any $\la \geq \mu$ and $\la>\gamma$. Then $\Ext^1_{\OI}(M, \pi(L(\mu)))=0$, for any $\la \geq \mu$.
	\end{itemize}
\end{lem}
\begin{proof} First, we prove Claim~(1).
Consider a short exact sequence
\begin{align}
	&0\rightarrow \pi(L(\mu)) \rightarrow \nb(\mu) \rightarrow E(\mu)\rightarrow 0 \label{eq::388}
\end{align} in $\OI$. Applying the functor $\Hom_{\OI}(M,-)$ to \eqref{eq::388}, we obtain the exact sequence
\begin{align*}
&\Hom_{\OI}(M, E(\mu)) \rightarrow \Ext^1_{\OI}(M, \pi(L(\mu)))\rightarrow \Ext^1_{\OI}(M, \nb(\mu)).
\end{align*} By assumptions, we have $\Ext^1_{\OI}(M, \nb(\mu))=0$. Also, if $[E(\mu):\pi(L(\gamma))]>0$, then $\mu \geq \gamma$, and so $\Hom_\OI(M, E(\mu)) =0.$ This implies that $\Ext^1_\OI(M,\pi(L(\mu))) =0$.

Claim~(2) can be proved similarly.
\end{proof}

\begin{proof}[Proof of Theorem \ref{thm5}] We shall only prove Claim~(1). The proof of Claim~(2) is similar and will be omitted.
	
	   By Lemma \ref{lem3}, we only need to show that $M \in \mc F(\ov \Delta)$ provided that $\Ext^d_{\ov{\mc O}}(M, \mc F(\nb))=0$, for all $d\geq 1$. To see this, we pick    $\la\in \Lnua$ such that $\la$ is minimal with the property $\Hom_{\OI}(M, \pi(L(\la)))\neq 0.$
    We let $N(\la)$ be the radical of $\ov \Delta(\la)$, namely,  we have the following short exact sequence in $\OI$: \begin{align}
   	&0\rightarrow N(\la) \rightarrow \ov \Delta(\la) \rightarrow \pi(L(\la)) \rightarrow 0. \label{eq::38}
   \end{align}
  We are going to show that \begin{align}
  	&\Hom_{\OI}(M, \ov \Delta(\la)) \cong \Hom_{\OI}(M, \pi(L(\la))). \label{eq::39}
  \end{align} To see this, we apply the functor $\Hom_{\OI}(M,-)$ to  \eqref{eq::38} and obtain the following exact sequence:
  \begin{align*}
  	&0\rightarrow \Hom_{\OI}(M, N(\la))\rightarrow \Hom_{\OI}(M, \ov \Delta(\la))\rightarrow \Hom_{\OI}(M, \pi(L(\la)))\rightarrow \Ext^1_{\OI}(M, N(\la)).
  \end{align*} It follows from  Proposition \ref{prop4} that $[N(\la): \pi(L(\mu))] =0$ unless $\mu <\la$.  By the minimality of $\la$  and Lemma \ref{lem::8}, we have $\Hom_{\OI}(M, N(\la))=0$ and $\Ext^1_{\OI}(M, N(\la))=0$. Consequently, \eqref{eq::39} holds and the quotient map $M\rightarrow \pi(L(\la))$ factors through the canonical quotient $\ov \Delta(\la)\rightarrow \pi(L(\la))$. 
  Then, we obtain a short exact sequence
 \begin{align*}
 	&0\rightarrow X \rightarrow M \rightarrow \ov \Delta(\la) \rightarrow 0
 \end{align*}in  $\OI$. For any $d\geq 1$ and $Y\in \mc F(\nb)$, applying the functor $\Hom_\OI(-, Y)$, we obtain an exact sequence
\begin{align*}
	&\Ext^d_\OI(M , Y)\rightarrow \Ext^d_\OI(X, Y) \rightarrow \Ext^{d+1}_\OI(\ov \Delta(\la), Y).
\end{align*} It follows that $\Ext^d_\OI(X, Y)=0$ and, by induction, we conclude that $X\in \mc F(\ov \Delta)$. This completes the proof of Claim~(1). 
 \end{proof}

\begin{cor}
	$\mc F(\ov \Delta)$ and $\mc F(\Delta)$ are closed under taking direct summands.
\end{cor}

\begin{proof}
The conclusion follows from the additive characterization of
$\mc F(\ov \Delta)$ and $\mc F(\Delta)$ given by  Theorem \ref{thm5}.
\end{proof}

\begin{prop}[BGG reciprocity] \label{prop::BGG}
	Let $\la \in \Lnua$. Then $\pi(P(\la))$ has a $\Delta$-flag in $\ov{\mc O}$. In particular, we have the following multiplicity formula:
	\begin{align}
	&(\pi(P(\la)): \Delta(\mu)) = [\ov \nb(\mu): \pi(L(\la))]. \label{eq::BGGre}
	\end{align}
Furthermore, if $\Phi(\g_1) = -\Phi(\g_{-1})$, then we have $$(\pi(P(\la)): \Delta(\mu)) = [\ov{\Delta}(\mu): \pi(L(\la))]. $$
	\end{prop}
\begin{proof}
	To show that $\pi(P(\la))\in \mc F(\Delta)$, it suffices to show that $\pi(\Ind P_0(\la))\in \mc F(\Delta)$. We note that $\Ind P_0(\la)\cong \Ind_{\g_{\geq 0}}^\g \Ind_{\g_\oa}^{\g_{\geq 0}} P_0(\la)$ has a filtration subquotients of which are of the form $K(P_0(\la+\gamma))$, for some $\gamma \in \Phi(\g_1)$, satisfying  $\la +\gamma \in \Lnua.$ Since each $\pi_0(P_0(\la+\gamma))$ has a $\Delta_0$-filtration, we know that $\pi(K(P_0(\la+\gamma)))$ has a $\Delta$-flag. This implies the first claim.
	
 By Lemma \ref{lem3}, for any exact sequence $0\rightarrow M' \rightarrow M\rightarrow M/M'\rightarrow 0$ in $\mc F(\Delta)$, we have a short exact sequence \begin{align*}
 	&0\rightarrow \Hom_{\ov{\mc O}}(M/M',\ov\nb(\mu)) \rightarrow \Hom_{\ov{\mc O}}(M,\ov\nb(\mu)) \rightarrow \Hom_{\ov{\mc O}}(M',\ov\nb(\mu))\rightarrow 0.
 	\end{align*}   In particular,	Equation \eqref{eq::BGGre} follows from Lemma \ref{lem3} as
	\begin{align*}
	&[\ov \nb(\mu): \pi(L(\la))] = \dim \Hom_{\ov{\mc O}}(\pi(P(\la)), \ov \nb(\mu)) = (\pi(P(\la)): \Delta({\mu})).
	\end{align*}
	
	Finally, suppose that $\Phi(\g_1) = -\Phi(\g_{-1})$.  We calculate the following characters:
	\begin{align*}
	&\ch (\Coind_{\g_{\leq 0}}^{\g} \ov{\nb}_0'(\eta)) = \ch(U(\g_1^\ast))\ch(M_0(\eta)) = \ch(U(\g_{-1}))\ch(M_0(\eta)) =\ch(\Ind_{\g_{\geq 0}}^{\g} \ov{\Delta}_0'(\eta)),
	\end{align*} for any $\eta \in \h^\ast$.
	 It follows that $\ov{\nb}(\mu)$ and $\ov{\Delta}(\mu)$ have the same composition factors in $\ov{\mc O}$, for all $\mu \in \Lnua$.   The conclusion follows.
\end{proof}

\begin{rem}
Since $(\Delta(\mu): \ov{\Delta}(\mu)) = [P(\mu, \mf l_\nu): L(\mu, \mf l_\nu)] = |W_\nu\cdot \mu|,$ for any $\mu \in   \Lnua$,  it follows that
\begin{align*}
	&(\pi(P(\la)): \ov{\Delta}(\mu)) = |W_\nu\cdot \mu|\cdot  [\ov \nb(\mu): \pi(L(\la))].
\end{align*}
\end{rem}

\begin{rem}
Descriptions of simple, standard and proper standard objects in $\Opres$
can be found in Appendix~\ref{sect::81}.
\end{rem}

\section{Tilting modules and Ringel duality} \label{sect::6}

In this section, we study tilting modules in $\ov{\mc O}$. We continue to assume that $\nu\in \h^\ast$ is dominant and integral. We set  $w_0^\nu:=w_0^\zeta$ to be the longest element in $W_\nu$ as defined  in  Section  \ref{sec:qsymm}. { Let $\Lnub$ be the set of all  $W_\nu$-dominant and integral weights.}

\subsection{Tilting modules} \label{sect::tiltings} As in the the general setup of properly stratified categories (see, e.g., \cite{FM06}),
we  consider the category  $\mc T:= \mc F(\Delta)\cap \mc F(\ov \nb)$ of {\em tilting modules}, respectively the category $\mc T':= \mc F(\ov \Delta)\cap \mc F(\nb)$ of {\em co-tilting modules}  in $\ov{\mc O}$.  Similarly, we  consider $\mc T_0 : = \mc F(\Delta_0)\cap \mc F(\ov \nb_0)$ and $\mc T'_0 : = \mc F(\ov{\Delta}_0)\cap \mc F(\nb_0)$. In this section, we show that $\mc T =\mc T'$ and classify all indecomposable tilting objects in this category.

\begin{lem}   \label{lem::11}
	The functors $\Res(-): \mc T\rightarrow \mc T_0$ and $\Ind(-):\mc T_0\rightarrow \mc T$ are well-defined.
\end{lem}
\begin{proof}
 We note that $\Res\Delta(\la) \cong  \pi_0(\Res K(\Delta_0'(\la)))\cong \pi_0(U(\g_{-1})\otimes \Delta_0'(\la))\in \mc F(\Delta_0)$, since $\mc  F(\Delta_0)$ is closed under tensoring with finite-dimensional modules. Similarly, we have $\Res (\ov \nb(\la)) \in \mc F(\ov \nb_0)$. Consequently, $\Res(\mc T)\subseteq  \mc T_0$.

 Let $\la \in \Lnua$. Then \begin{align} \label{eq::1}
 	&\pi(\Ind \Delta_0'(\la))\cong \Ind_{\g_{\geq 0}}^\g\pi_{\geq 0}(U(\g_1)\otimes\Delta_0'(\la)),
 \end{align} where $\pi_{\geq 0}: \mc O(\g_{\geq 0})\rightarrow \ov{\mc O}(\g_{\geq 0})$ denotes the quotient functor.
 Since $\mc F(\Delta_0)$ is closed under tensoring with finite-dimensional modules,
 the module  $\pi_{\geq 0}(U(\g_1)\otimes \Delta_0'(\la))$ has a $\Delta_0$-flag whose subquotients are of the form $\Delta_0(\mu)$ with $\g_1$ acting trivially. Here $\mu\in\Lambda(\nu)$ is the $W_\nu$-anti-dominant conjugate of $\la+\gamma$, where $\gamma$ is a weight in $\wedge(\g_1)$. This means that $\pi(\Ind \Delta_0'(\la))$ has a $\Delta$-flag. Then $\Ind(\mc F(  \Delta_0))\subseteq \mc F(  \Delta)$. Similarly, $\Ind(\mc F(\ov \nb_0))\subseteq \mc F(\ov \nb)$.
\end{proof}

\begin{lem} \label{lem::9}
For any $\la \in \Lnua$, there exists a unique (up to isomorphism) indecomposable tilting module $T^\OI(\la)\in \ov{\mc O}$ with a short exact sequence
\begin{align}
&0\rightarrow \Delta(\la) \rightarrow T^\OI(\la) \rightarrow C(\la) \rightarrow 0, \label{eq::2}
\end{align}  in $\ov{\mc O}$, where $C(\la) \in \mc F(\Delta).$

Furthermore, $T^\OI(\la)$ is the unique indecomposable  co-tilting module  such that there is a short exact sequence \begin{align}
	&0\rightarrow C'(\la) \rightarrow T^\OI(\la) \rightarrow \nabla(\la) \rightarrow 0, \label{eq::3}
\end{align}  in $\ov{\mc O}$, where $C'(\la) \in \mc F(\nabla).$
\end{lem}

\begin{proof}

	First, we prove uniqueness following  the argument in \cite[Proposition 5.6]{So98}.  Suppose that $T^\OI(\la)$ and $T$ are two indecomposable tilting modules in $\OI$ such that \eqref{eq::2} holds and there is a short exact sequence \begin{align*}
		&0\rightarrow \Delta(\la) \xrightarrow{i} T \rightarrow C \rightarrow 0,
	\end{align*} where $C \in \mc F(\Delta).$ Applying the functor $\Hom_\OI(-,T^\OI(\la))$, we get an exact sequence
\begin{align*}
	&0\rightarrow \Hom_\OI(C, T^\OI(\la)) \rightarrow   \Hom_\OI(T, T^\OI(\la)) \rightarrow  \Hom_\OI(\Delta(\la), T^\OI(\la))\rightarrow 0,
\end{align*} since $\Ext^1_{\ov{\mc O}}(C, T^\OI(\la))=0.$ Therefore, if we denote the inclusion $\Delta(\la)\hookrightarrow T^\OI(\la)$ by $\iota$, then there exists $\alpha\in \Hom_\OI(T,T^\OI(\la))$ such that $\alpha\circ i = \iota$. Similarly, there also exists  $\beta\in \Hom(T^\OI(\la),T)$ such that $\beta \circ \iota = i$. By the Fitting lemma for arbitrary abelian categories given in \cite[Lemma 5.3]{Kr15}, the homomorphisms $\alpha\circ \beta$ and $\beta\circ \alpha$ are either invertible or nilpotent. Since $\alpha\circ\beta(\Delta(\la)) =\Delta(\la) = \beta \circ\alpha(\Delta(\la))$, we may conclude that both $\alpha\circ \beta$ and $\alpha\circ \beta$ are isomorphisms. This proves that $T\cong T^\OI(\la)$.

For any $\la\in \h^\ast$, recall that $T_0(\la)$ denotes the indecomposable tilting module in $\mc O_\oa$ of highest weight $\la$. By \cite{FKM00}, the indecomposable tilting objects in $\ov{\mc O}_\oa$ are $\{\pi(T_0(\la))|~\la\in \Lnub\}$, up to isomorphism. For a given $\la\in \Lnua$, we let $T_0^{\ov {\mc O}_\oa}(\la) =\pi_0(T_0(w_0^\nu\cdot \la))$ denote the indecomposable tilting object in $\ov{\mc O}_\oa$. There is an embedding from $\Ind \Delta_0(\la)$ to $\Ind T_0^{\OI_\oa}(\la)$  such that the cokernel of this embedding has a $\Delta$-flag. It follows from \eqref{eq::1} that $\Ind\Delta_0(\la)$ has a $\Delta$-flag starting at $K(\Delta_0(\la+2\rho_1))=\Delta(\la+2\rho_1)$.

Finally, we note that $T^\OI(\la)$ is also a co-tilting module since $\mc T_0=\mc T_0'$ by \cite{FKM00}. Therefore, the uniqueness in the second claim can be proved using an identical argument. This completes the proof.
 \end{proof}

 We give a description of  indecomposable tilting modules of $\OI$ in the following proposition.
\begin{prop} \label{prop::15}
For any $\la \in \Lnua$, we have $T^\OI(\la) = \pi(T(w_0^\nu\cdot \la))$.
\end{prop}

Before we prove Proposition \ref{prop::15}, we need to do some preparation in the
next subsection.

\subsection{Ringel duality} \label{sect::62}

We denote by $w_0^W$ the  longest element in $W$.
  \subsubsection{} \label{sect::621}
For any simple reflection $s\in W$, we let ${\bf T}_s: \mc O\rightarrow \mc O$ denote Arkhipov's
{\em twisting functor} as introduced in \cite[Section 3.6]{CMW}; see also \cite{Ar, AS, CoM1}. As in the classical case \cite{KM2}, twisting functors satisfy braid relations; see \cite{CoM1}. This allows us to define ${\bf T}_w$, for any $w\in W$. We use ${\bf T}^0_w$ to denote the corresponding twisting functor on $\mc O_\oa$.

In \cite{Ri}, Ringel developed a duality theory for quasi-hereditary algebras which is now known as Ringel duality. In \cite{So98}, Soergel proved Ringel self-duality of category
$\mc O$ by establishing an  equivalence between the full subcategories of projective modules and tilting modules. This is done by applying the functor ${\bf T}_{w_0^W}$.
In fact, for a given $\mu\in \Lambda$, we have
$${\bf T}_{w_0^W}^0P_0(\mu) = T_0(w_0^W \cdot \mu).$$

Building on \cite{So98}, Brundan  established
Ringel self-duality for category $\mathcal{O}$ for   Lie superalgebras
in \cite{Br04};  see also \cite{Ma,CCC} for further results.
In the formulation of \cite[Theorem~3.7]{CCC},
for any $\mu \in \h^\ast$, we have
${\bf T}_{w_0^W}P(\mu) = T^r(\mu'),$
 where  $T^r(\mu')$  denotes the indecomposable tilting module of highest weight $$\mu':=w_0^W\mu -2\rho_\oa-2w_0^W\rho_{-1},$$ with respect to the  {\em reverse Borel subalgebra} $\mf b^r:= \mf b_\oa+\g_{-1}$.

  \subsubsection{}
  We define a map $\hat{\cdot}: \mf h^\ast\rightarrow \mf h^\ast$ by letting $\hat \mu:=-w_0^W\mu$, for any $\mu \in \h^\ast$. {Recall that we denote  the set of all integral and $W_\nu$-dominant weights by  $\Lnub$.}  We note that the map $\mu \mapsto w_0^W\cdot \mu$ gives a bijection between $\Lnub$ and $\Lambda(\widehat{\nu})$.

\begin{lem} \label{lem::16} We have
  $$ \{T^r(\la)|~\la \in \Lnub\} = \{T(\la)|~\la \in \Lnub\} .$$
  In particular, the twisting functor ${\bf T}_{w_0^W}$ gives rise to an equivalence between the additive closures $\add(\{P( \mu)|~\mu \in \Lnua\})$ and $\add(\{T( \mu)|~\mu \in \Lnub\})$.
\end{lem}
\begin{proof}
Suppose that $\la,\mu \in \Lambda$ are such that
$P(\la) = P^r(\mu),$
where  $P^r(\mu)$ denotes the projective cover of the irreducible module $L^r(\mu)$ of highest weight $\mu$ with respect to $\mf b^r$. For a given   root $\alpha\in \Pi_\nu$, we  calculate
\begin{align*}
	&\langle \mu +\rho_\oa, \alpha^\vee \rangle \leq 0 \Leftrightarrow L^r(\mu) \text{ is }\alpha\text{-free} \Leftrightarrow  L(\la) \text{ is }\alpha\text{-free}\Leftrightarrow \langle \la+\rho_\oa,\alpha^\vee \rangle \leq 0.
\end{align*}

By \cite[Theorem 3.7]{CCC}, there is a duality ${\bf D}: \mc O\rightarrow \mc O$ such that
\begin{align*}
&{\bf D T}_{w_0^W} P(\kappa) = T(-\kappa -2\rho_\oa+2\rho_{-1}),~
{\bf D T}_{w_0^W} P^r(\kappa) = T^r(-\kappa -2\rho_\oa+2\rho_{1}),
\end{align*}
for any $\kappa\in \Lambda$. We calculate that
\begin{align*}
	&\{T(\la)|~\la\in \Lambda \text{ with } \langle\la+\rho_\oa,\alpha^\vee \rangle\geq 0 \}\\
	&=\{{\bf DT}_{w_0^W}P(\la)|~\la\in \Lambda \text{ with }\langle\la+\rho_\oa,\alpha^\vee \rangle\leq 0\}\\
	&=\{{\bf DT}_{w_0^W}P^r(\la)|~\la\in \Lambda \text{ with }\langle\la+\rho_\oa,\alpha^\vee \rangle\leq 0\}\\
	&=\{T^r(\la)|~\la\in \Lambda \text{ with } \langle\la+\rho_\oa,\alpha^\vee \rangle\geq 0 \}.
\end{align*}
This completes the proof.
\end{proof}

\begin{proof}[Proof of Proposition \ref{prop::15}]
	We first show that  the following set \begin{align*}
		&\{\pi(T(\la))|~\la \in \Lnub\}
	\end{align*} is an exhaustive list of pairwise non-isomorphic indecomposable tilting modules in $\mc \OI.$ To see this,  we note that $\mc T = \add(\{\Ind(\mc F(\Delta_0)\cap \mc F(\ov \nb_0))\})$ by Lemmas \ref{lem::11} and \ref{lem::9}. By \cite{FKM00}, the set  $$\{\pi_0(T_0(\la))|~\la \in \Lnub\},$$ is a complete list of indecomposable tilting objects in $\OI_\oa$. It follows that $$\mc T = \add(\{\pi(\Ind(T_0(\mu)))|~\mu \in \Lnub\}).$$  By \cite[Theorem 3.5-(i)]{CCC}, we know that $\mc T$ is a full subcategory of the category $\add(\{\pi(T(\mu))|~\mu\in \Lambda\})$.



	
	On the other hand, we have $\add(\{\Ind P_0(\mu)|~\mu \in \Lnua\}) = \add(\{P(\mu)|~\mu\in \Lnua\})$ by \eqref{eq::26}.
	We calculate
	\begin{align*}
	&\add(\{\Ind T_0(\mu)|~\mu\in \Lnub\})\\
	&=\add(\{\Ind {\bf T}^0_{w_0^W}P_0(\mu)|~\mu\in  {\Lambda}(\widehat \nu)\})\\
	&=\add(\{{\bf T}_{w_0^W}\Ind P_0(\mu)|~\mu\in  {\Lambda}(\widehat \nu)\})\\
	&=\add(\{{\bf T}_{w_0^W} P(\mu)|~\mu\in  {\Lambda}(\widehat\nu)\})\\
	&=\add(\{T^r( \mu)|~\mu \in \Lnub\})\\
	&=\add(\{T( \mu)|~\mu \in \Lnub\}). 
	\end{align*}
	The last equality follows from Lemma \ref{lem::16}.
	
 Next,  we shall show that $\pi(T(\la))$ is indecomposable, for any $\la \in \Lnub$.
 By \cite[Lemma 18]{FKM00}, we have $T_0(\la)\in \mc O^{\vpre}_\oa$. By Lemma \ref{lem::1}, it follows that $\Ind T_0(\la)\in \mc O^{\vpre}$.
 By Lemma \ref{lem01}, we obtain an isomorphism  $$\pi:\Hom_{\mc O}(\Ind T_0(\la), \Ind T_0(\mu))\xrightarrow{\cong}\Hom_{\OI}(\pi(\Ind T_0(\la)), \pi(\Ind T_0(\mu))),$$ for $\la, \mu \in \Lnub$. In particular, $$\End_{\OI}(\pi(T(\la)))\cong\End_{\mc O}(T(\la)).$$ Since the latter is a local ring, it follows  from \cite[Proposition 5.4]{Kr15} that $\pi(T(\la))$ is indecomposable.

 Finally, for any $\la\in \Lnub$,  there is an inclusion $M(\la)\hookrightarrow T(\la)$ with co-kernel having a Verma flag. Since $\pi(M(\la)) = \ov\Delta(w_0^\nu\cdot \la)$, we know that $\pi(T(\la))$ admits a $\ov \Delta$-flag starting at $\ov \Delta(w_0^\nu\cdot \la)$. The conclusion  follows from Lemma \ref{lem::9}.
\end{proof}

\begin{cor} The category $\mc T$ is equivalent to the full subcategory
$\mc P$ in $\ov{\mc O}$ consisting of all projective objects.
\end{cor}
\begin{proof}
	 Since we have  $$\pi: \add(\{P( \mu)|~\mu \in \Lnua\})\xrightarrow{\cong}\mc P\quad\text{ and }\quad\pi: \add(\{T( \mu)|~\mu \in \Lnub\})\xrightarrow{\cong}\mc T,$$ the conclusion follows from Lemma \ref{lem::16}.
\end{proof}

\section{Categorification of the $q$-symmetrized Fock space} \label{sect::7}

 In  this section, unless stated otherwise, we assume that $\g$ is one of
the Lie superalgebras $\gl(m|n), \mf{osp}(2|2n)$ or $ \pn$. We recall that $\nu\in \h^\ast$ is a (possibly singular) dominant integral weight and $\zeta\in {\bf ch}\mf n_\oa^+$ is such that $W_\nu = W_\zeta$.

\subsection{Multiplicity problem for standard Whittaker modules} The problem of computing the composition factors of standard Whittaker modules has been solved by Backelin \cite{B} and Mili{\v{c}}i{\'c} and Soergel \cite{MS}  in the case of semisimple Lie algebras. The answer is formulated in terms of the Kazhdan-Lusztig combinatorics of category $\mc O$.
The main tool used by Mili{\v{c}}i{\'c} and Soergel  \cite{MS} was an equivalence from a certain subcategory of Harish-Chandra bimodules to a certain category of Whittaker modules, which restricts to the equivalence \eqref{eq::415} in Lemma~\ref{lem::5}. Backelin, in turn, used a certain functor $\Gamma_\zeta$ from $\mc O$ to $\mc N(\zeta)$. For quasireductive Lie superalgebras, the analogues of Backelin's functor $\Gamma_\zeta$ and Mili{\v{c}}i{\'c}-Soergel's equivalence were constructed in \cite{Ch21}.

\subsection{The Backelin functor $\Gamma_\zeta$}
For  given $M\in \mc O$ and $\la \in \h^\ast$, we recall from \eqref{eq::wtspace} that $M_\la$ denotes the weight subspace of $M$ corresponding to $\la$. Then the completion $\ov M:= \prod_{\la \in \mf h^\ast} M_\la$ admits the  structure of a $\g$-module in a natural way.
 We recall the super version of Backelin's functor from \cite[Section 5.2]{Ch21}, which was initially studied by Backelin \cite{B} in the setting of the category $\mc O$ for a reductive Lie algebra:
\begin{align*}
&  \Gamma_\zeta:\mc O\rightarrow \mc N(\zeta),~
M\mapsto  \{m\in \ov M|~x-\zeta(x)\text{ acts nilpotently on $m$ for any }x\in \mf n_\oa^+ \}.
\end{align*}  For semisimple Lie algebras, our notation  $\Gamma_\zeta$
corresponds to $\ov{\Gamma}_f$ in \cite{B, AB}.

By \cite[Theorem 20]{Ch21}, the functor $ {\Gamma}_\zeta$ is exact and, for any $\la \in \Lambda$, we have
	\begin{align}
		&\Gamma_\zeta(M(\la))\cong M(\la,\zeta); \label{qe::1}\\
		&\Gamma_\zeta({L}(\la)) \cong \left\{ \begin{array}{ll} {L}(\la, \zeta),\quad \text{if $\la\in \Lnua$;} \\
			0, \quad\text{otherwise}.\end{array} \right.   \label{qe::3}
	\end{align}
  Furthermore, $\Gamma_\zeta$ functorially commutes with projective functors; see also \cite[Proposition 3]{AS}. Denote by $\Gamma_\zeta^0$ the original Backelin's functor studied in \cite{B}. By the proofs of \cite[Proposition 3]{AB} and \cite[Theorem 20]{Ch21},  we have
  \begin{align}
  &\Ind ( \Gamma_\zeta^0(M)) \cong  \Gamma_\zeta(\Ind(M)), \label{eq::73} \\
   &K(\Gamma_\zeta^0(M)) \cong  \Gamma_\zeta(K(M)), \label{eq::74}
  \end{align} for any  $M\in \mc O_\oa$.
We will show that the target category $\Gamma_\zeta(\mc O)$ of $\Gamma_\zeta$ is equivalent to $\OI$, see Corollary \ref{coro::univprop}.

\subsection{A Mili{\v{c}}i{\'c}-Soergel type equivalence}
\subsubsection{The functor $F_\zeta$} We define a functor $F_\zeta(-):\mc O\rightarrow \mc W(\zeta)$ by letting
$$F_\zeta(-):=\mc L(M_0(\nu), -) \otimes_{U(\g_\oa)} M_0(\nu,\zeta): \mc O\rightarrow \mc B_\nu \xrightarrow{\cong} \mc W(\zeta).$$
By \cite[Theorem 26]{Ch212}, this functor has the following properties:
\begin{thm} \label{prop18}
The following holds:
\begin{itemize}
	\item[(1)] $F_\zeta(-)$ is exact.
	\item[(2)] $F_\zeta(-)$ functorially commutes   with projective functors.
	\item[(3)] Let $\la\in \Lnua$ and $\mu\in \Lambda$. Suppose that $\la\in W_\nu\cdot \mu$. Then
	\begin{align*}
	&F_\zeta(M(\mu))=F_\zeta(M(\la))=M(\la,\zeta).
	\end{align*}
	\item[(4)]	We have
	\begin{align}
		&F_\zeta(L(\mu))\cong\begin{cases} L(\mu,\zeta), &\mbox{ for $\mu \in \Lnua$,}\\
			0,&\mbox{ otherwise.}
		\end{cases} \label{eq::twoeq}
	\end{align}
\end{itemize}
\end{thm}
The following lemma provides an alternative proof of \cite[Theorem 6.2]{B} and \cite[Theorem 20]{Ch21}, where the cases of arbitrary weights are considered.
\begin{lem}
	Let $\la,\mu\in \Lambda$ and $\la_0\in W_\nu\cdot \la\cap \Lnua,~\mu_0\in W_\nu\cdot \mu\cap \Lnua$. Then we have
	\begin{align*}
	&[M(\la,\zeta): L(\mu, \zeta)] = [M(\la_0):L(\mu_0)].
	\end{align*}
\end{lem}
\begin{proof}
 This is a direct consequence of Theorem \ref{prop18}.
\end{proof}

	In the following theorem, we show that the functor $F_\zeta(-)$ is a realization of the quotient functor $\pi$.
	\begin{thm} \label{prop::univprop}
	The functor $F_\zeta(-): \mc O\rightarrow \mc W(\zeta)$ satisfies the universal property of the Serre quotient category $\OI$.
	\end{thm}
	\begin{proof}
		By the universal property in Lemma \ref{lem::unipro}, we get an induced functor $F'_\zeta: \OI \rightarrow \mc W(\zeta)$ such that $F_\zeta = F'_\zeta\circ \pi$. We shall show that $F'_\zeta$ is an equivalence. By Lemma \ref{lem::03}, $\mc L(M_0(\nu),-)$ restricts to an equivalence from $\Opres$ to $\mc B_\nu$. It follows that  $F_\zeta$ restricts to an equivalence from $\Opres$ to $\mc W(\zeta)$. It remains to show that $F'_\zeta$ is full and faithful. By Lemma \ref{lem01}, the homomorphism space between two objects in $\ov{\mc O}$ is of the form $\Hom_{\OI}(\pi(M), \pi(N))$, for some $M,N\in \Opres$. Then, we get two isomorphisms
		\begin{align*}
			&\pi: \Hom_{\mc O}(M,N) \xrightarrow{\cong}  \Hom_{\OI}(\pi(M),\pi(N)),\\
			&F_\zeta: \Hom_{\mc O}(M,N) \xrightarrow{\cong}  \Hom_{\mc N}(F_\zeta(M),F_\zeta(N)).
		\end{align*}  This implies that $F'_\zeta: \Hom_{\OI}(\pi(M),\pi(N)) \rightarrow \Hom_{\mc N}(F_\zeta(M),F_\zeta(N))$ is an isomorphism. The theorem follows.
	\end{proof}

\begin{cor} \label{coro::univprop}
	The functors $\Gamma_\zeta$ and $F_\zeta$ are isomorphic as functors from $\mc O$ to $\mc N(\zeta)$.
		As a consequence, we have $\Gamma_\zeta(\mc O) = \mc W(\zeta)\cong \ov{\mc O}$. 	
	 In particular, the Backelin functor $\Gamma_\zeta: \mc O\rightarrow \mc W(\zeta)$ satisfies the universal property of the Serre quotient category $\OI$.
\end{cor}
\begin{proof}
	Since both exact functors $\Gamma_\zeta$ and $F_\zeta$  functorially commute with projective functors and $\Gamma_\zeta(M(\la))\cong F_\zeta(M(\la))$, for any $\la\in \h^\ast$, we have $\Gamma_\zeta \cong F_\zeta$ by Lemma \ref{lem::fpfO}.
		Consequently, we  have  $\Gamma_\zeta(\mc O) = F_\zeta(\mc O) =\mc W(\zeta)$. Finally, it follows from Lemmas 	 \ref{lem01}, \ref{lem::03} and \ref{lem::5} that $\mc W(\zeta)\cong \OI$. This completes the proof.
\end{proof}

 \subsection{Applications of the quotient functor realization of $\Gamma_\zeta$}
 In this subsection, we give some applications of the realization of $\Gamma_\zeta$ in terms of the Serre quotient functor.
\subsubsection{Left and right adjoints of $\Gamma_\zeta$} \label{re::37}	
	 In the case when $\g=\g_\oa$ is a semisimple Lie algebra, Arias and Backelin \cite{AS} studied the restriction of $\Gamma_\zeta$ to a  block $\mc O_\la$ containing $L(\la)$, for a dominant, integral and regular weight $\la \in \h^\ast$. It is proved in \cite[Corollary 2]{AB} that $\Gamma_\zeta: \mc O_\la \rightarrow \Gamma_\zeta(\mc O_\la)$ has both a left adjoint $\Gamma_\zeta^L: \Gamma_\zeta(\mc O_\la)\rightarrow \mc O_\la$ and a right adjoint $\Gamma_\zeta^R: \Gamma_\zeta(\mc O_\la)\rightarrow \mc O_\la$ such that the evaluation $V\rightarrow \Gamma_\zeta\Gamma_\zeta^L(V)$ of the adjunction map at each $V\in \Gamma_\zeta^L(\mc O_\la)$ is an isomorphism. In the present paper this is generalized to the setup of Lie superalgebras.

	In \cite[Section 3.2]{AB},  Arias and Backelin calculated  the left adjoint $\Gamma_\zeta^L$ and the right adjoint $\Gamma_\zeta^R$ on two standard Whittaker modules $M(\la,\zeta)$  in the case when $\g$ is a semisimple Lie algebra and $\la$ is regular. We now generalize this calculation by computing the images of $\Gamma^L_\zeta$ and $\Gamma_\zeta^R$ on all simple and standard Whittaker modules in $\mc W(\zeta)$ for Lie superalgebras of type I as follows.
	
	 Let $w_0^\nu$ denote the longest element in $W_\nu$.  Recall the twisting functor $\mathbf T_{w_0^\nu}:\mc O\rightarrow \mc O$  associated to $w^\nu_0$ from Section \ref{sect::621}. If $\g =\g_\oa$ and $M\in \mc O$ such that the projective cover of $M$ is $\nu$-admissible, then $\mathbf T_{w_0^\nu}(M)$ is isomorphic to  the partial coapproximation of $M$ in the sense of \cite[Section 2.5]{KM2}, namely, we have  $\mathbf T_{w_0^\nu}(M)\cong \pi^L\pi(M)$; see \cite{KM2, MaSt07}.
	Such a realization of twisting functor has been generalized to Lie superalgebras in \cite[Proposition 5.9]{Co16} and \cite[Theorems 9, 12]{Ch212}, including the series in  \eqref{eq::claA}-\eqref{eq::claP}.  To summarize, we have the following consequence. 	
	\begin{cor}
	Let $\g$ be a quasireductive Lie superalgebra of type I. Then the Backelin functor $\Gamma_\zeta: \mc O\rightarrow \mc W(\zeta)$ admits both a left adjoint $\Gamma_\zeta^L$ and a right adjoint $\Gamma_\zeta^R$. If $\mc O$ admits a simple preserving duality $D$, then we have \begin{align}
		&\Gamma_\zeta^R\circ \Gamma_\zeta(-) \cong D\circ \Gamma_\zeta^L\circ \Gamma_\zeta\circ D(-). \label{eq::isolr}
	\end{align}	Furthermore, suppose that $\g$ is one of the Lie superalgebras from the series   in \eqref{eq::claA},   \eqref{eq::claC} or \eqref{eq::claP}. Then, for any $\la\in \Lnua$ we have
	\begin{align}
	&\Gamma_\zeta^L(L(\la,\zeta))\cong {\bf T}_{w_0^\nu}L(\la), \label{eq::Cor39eq1}\\
	&\Gamma_\zeta^L(M(\la,\zeta))\cong {\bf T}_{w_0^\nu}M(\la) \label{eq::Cor39eq2}.
	\end{align} 

	\end{cor}
\begin{proof}
 The existence of left and right adjoints of $\Gamma_\zeta$ is proved in  Proposition \ref{prop::leftadj} and Corollary \ref{coro::univprop}. The isomorphism  in \eqref{eq::isolr} follows from the construction of $\pi^L, \pi^R$ in Proposition \ref{prop::leftadj}; see   Remark \ref{rem::14}.

 Next, we shall establish the isomorphisms in  \eqref{eq::Cor39eq1} and \eqref{eq::Cor39eq2}. To see this, recall that $\Gamma_\zeta(M(\la))\cong M(\la,\zeta)$ and $\Gamma_\zeta(L(\la))\cong L(\la,\zeta)$. By  Proposition~\ref{prop::leftadj}, Corollary \ref{coro::univprop} and  Proposition~\ref{prop::45}, we have
  \begin{align*}
 	&\Gamma_\zeta^L(M(\la,\zeta))\cong \Gamma_\zeta^L\Gamma_\zeta(M(\la))\cong \pi^L\pi(M(\la)),\\
 		&\Gamma_\zeta^L(L(\la,\zeta))\cong \Gamma_\zeta^L\Gamma_\zeta(L(\la))\cong \pi^L\pi(L(\la)).
 \end{align*}
Since the projective cover of $M(\la)$ is {$\nu$-}admissible, we have $\pi^L\pi(M(\la))\cong {\bf T}_{w_0^\nu}M(\la)$ by \cite[Theorems 9,12]{Ch212}. Similarly, we have $ \pi^L\pi(L(\la))\cong {\bf T}_{w_0^\nu}L(\la)$.
This completes the proof.
\end{proof}

\begin{ex}[Arias-Backelin] Consider the case that $\g$ is a semisimple Lie algebra. In this example, we recover   the calculations of  $\Gamma_\zeta^L(M(\la,\zeta))$ and $\Gamma_\zeta^R(M(\la,\zeta))$ given in \cite[Section 3.2]{AB}, where the cases of dominant and anti-dominant integral and regular weights $\la$ were considered. Recall that $w_0^W$ denotes the longest element in $W$ and we have a simple-preserving duality $(\cdot)^\vee$ on $\mc O$.
	
	Let $\la$ be a dominant, integral and regular  weight. Then we have a short exact sequence $0\rightarrow M(\la) \rightarrow Q_1\rightarrow Q_2$ for some projective-injective modules $Q_1,Q_2$ in $\mc O$; see, e.g., \cite[Section 3.1]{KSX}. This shows that $\Gamma_\zeta^R(M(\la,\zeta))\cong \pi^R(\pi(M(\la)))\cong M(\la)$.   Next, since $w_0^\nu\cdot \la\in \Lnua$, it follows that  $\Gamma_\zeta^L(M(\la,\zeta))\cong \Gamma_\zeta^L(M(w_0^\nu\cdot \la,\zeta))\cong {\bf T}_{w_0^\nu}M(w_0^\nu\cdot\la)$. By \cite[Section 5]{KM2}, ${\bf T}_{w_0^\nu}M(w_0^\nu\cdot\la)$ is isomorphic to the quotient of $P(w_0^\nu\cdot\la)$ by the sum of all homomorphic images from $P(w_0^\nu\cdot\la)$ to the kernel of the canonical quotient $P(w_0^\nu\cdot\la)\rightarrow M(w_0^\nu\cdot\la).$

 Next, we consider $\Gamma_\zeta^L(M(w_0^W\cdot \la,\zeta))$ and $\Gamma_\zeta^R(M(w_0^W\cdot \la,\zeta))$ as follows. Since $w_0^W\cdot \la\in \Lnua$, we have $\Gamma_\zeta^L(M(w_0^W\cdot \la,\zeta))\cong {\bf T}_{w_0^\nu}(M(w_0^W\cdot \la))\cong M(w_0^\nu w_0^W\cdot \la)^\vee$ by \cite[Theorem 2.3]{AS}. Finally, by Remark \ref{rem::14} we have  \begin{align*}
 	&\Gamma_\zeta^R(M(w_0^W\cdot \la,\zeta))\cong \Gamma_\zeta^R(\Gamma_\zeta(M(w_0^W\cdot \la)))\cong (\Gamma_\zeta^L\Gamma_\zeta(M(w_0^W\cdot \la)^\vee))^\vee \cong  M(w_0^\nu w_0^W\cdot \la).
 \end{align*}
\end{ex}


    	\subsubsection{A McDowell type category and representations of finite $W$-algebras}\label{rem::40} In this section, we continue to assume  that $\g$ is a semisimple Lie algebra.  For a given dominant and integral weight $\nu$ and $\eta\in \bf \ch \mf n^+$, we denote by $\mc N(\ker(\chi_\nu), \eta)$ the full subcategory of $\mc N(\eta)$  consisting of all modules annihilated by the kernel of the central character $\chi_\nu$ associated to  $\nu$. The study of this type of category goes back at least to McDowell \cite{Mc}. The case of the principal block $\mc N(\ker(\chi_0), 0)$ was also considered in the work of Soergel \cite{So86}.

    	An interesting connection with the representation theory of finite $W$-algebras is the following equivalence constructed  by Webster  \cite[Proposition 7]{We11}, building on the earlier work of Losev \cite[Theorem 4.1, Proposition 4.2]{Lo09} and \cite[Theorem 1.2.2(iii)]{Lo10jams}:
    		\begin{align}
    		&\mc O'(\chi_\nu,\mc W_e, \mf p_\eta)\cong \mc N(\ker(\chi_\nu), \eta). \label{eq::WeLo}
    	\end{align} Here $\mc W_e$ denotes the finite $W$-algebra associated to $\g$ and a given  nilpotent element $e\in \g$ which is regular in the Levi $\mf l_\eta$ of the parabolic subalgebra $\mf p_\eta$ determined by $\eta$. Then   $\mc O'(\chi_\nu,\mc W_e, \mf p_\eta)$ denotes the  category of modules in the category $\mc O$ for $\mc W_e$ from \cite[Section 4]{Lo09} on which the center $Z(\mc W_e)\cong Z(\g)$ acts semisimply via the central character $\chi_\nu$. We refer to \cite[Section 2]{We11} and \cite[Section 4]{Lo09} for more details. We remark that the equivalence  \eqref{eq::WeLo} is stated   for regular weights $\nu$ in \cite{We11}, however, by the same arguments as in the proof of \cite[Propoisiton 7]{We11} the equivalence \eqref{eq::WeLo} holds for the case of singular weights as well.  

    	Let $\nu$ be regular.
    		For $\mu\in W\cdot \nu$, the structure of Whittaker modules $\Gamma_\zeta(M_0( \mu)^\vee)$ has been studied by  Brown and Romanov in \cite[Sections 6,7]{BR}. It is clear that $\Gamma_\zeta(M_0(\mu)^\vee)$ and $\Gamma_\zeta(M_0(\mu))$ have the same composition factors (see also \cite[Theorem 6.4]{BR}). From Corollary \ref{coro::univprop} and results in Section \ref{sect::5}, we know that $\Gamma_\zeta(M_0(\mu)^\vee)$ plays the role of a proper costandard object in the stratified structure of $\mc W(\zeta)$. In particular,
    	$\Gamma_\zeta(M_0(\mu)^\vee)$ has simple socle, as shown in \cite[Theorems 6.4, 6.8]{BR}. In \cite[Theorem 8.2]{BR} it is proved that $\Gamma_\zeta(M_0(\mu)^\vee)$
    	appears in the BGG reciprocity for the block $\mc N(\ker(\chi_\nu),\eta)$. 

    	It is natural to investigate the structure of $\mc N(\ker(\chi_\nu),\eta)$ for  singular  $\nu$. Note that $\mc N(\ker(\chi_\nu),\eta)$ is equivalent to the full subcategory of $\mc B_\nu$ from Section \ref{Sect::HCbimod} consisting of bimodules $X$ such that  $\ker(\chi_\la)X=0$ for some dominant and integral weight $\la\in \h^\ast$ with $W_\la=W_\eta$.  Therefore, it is equivalent to the   block $\mc O^{\vpre}_\la$ of $\mc O$ associated to $\la$, namely, it consists of modules in $\mc O^{\vpre}$ on which   $\ker({\chi_\la})$ acts nilpotently.  We summarize our discussion  in the following corollary:
    	\begin{cor} \label{cor::41}
    	Let $\g$ be a semisimple Lie algebra with $\eta \in {\bf ch}\mf n^+$ and  $\la,\nu\in \h^\ast$ be dominant and integral weights such that $W_\la=W_\eta$. Then $\mc N(\ker(\chi_\nu), \eta)$ (and hence also the category $\mc O'(\chi_\nu,\mc W_e, \mf p_\eta)$ above) is equivalent to  the (possibly singular) block $\mc O^{\vpre}_\la$.
    	
    	In particular, $\mc N(\ker(\chi_\nu), \eta)$ admits a stratified structure with proper standard objects  $\Gamma_\zeta(M(\mu))$,  proper costandard objects  $\Gamma_\zeta(M(\mu)^\vee)$, standard objects  $\Gamma_\zeta(\Ind_{\mf p}^{\mf g}P(\mf l_\nu,\mu))$, and costandard objects  $\Gamma_\zeta(\Ind_{\mf p}^{\mf g}P(\mf l_\nu,\mu)^\vee)$ for $\mu\in \Lnua\cap W\cdot \nu.$
    	Furthermore, it is a highest weight category if $\nu$ is regular.
    	\end{cor}
    	
    	In the case when $\nu$ is regular,  Corollary \ref{cor::41} has been obtained in \cite{BR}; see also \cite[Remark 1.1, Corollary 7.4, Theorem 8.2]{BR}. 
    		In general, the full regular block of $\mc N(\eta)$ is equivalent to a (possibly singular) block of the {\em thick} category $\mc O$, namely, the category of finitely generated $\g$-modules which are locally finite over both $\mf b$ and $Z(\g)$; see \cite[Sections 2.6.2, 2.6.3]{AB}. These categories do not have enough projectives.

		\subsubsection{Realization of Soergel's  functor $\mathbb V$ in terms of  $\Gamma_\zeta$} \label{ref::41} Let $\mc O_\la$ be the  block of $\mc O$ over a semisimple Lie algebra $\g$  associated to a regular,  antidominant and integral weight $\la$. As is well-known,  Soergel's
	combinatorial functor $\mathbb V$ introduced in \cite{So90}  plays a significant role in the representation theory of {a} semisimple Lie algebra. It is an exact functor sending $M\in \mc O_\la$ to $\Hom_\g(P(\la),M)$, which is  a right module over $\End_\g(P(\la))$. In the case when $W_\zeta =W$, Backelin proved in \cite[Corollary 5.4]{B} that  $\Gamma_\zeta(-)$ and $\mathbb V(-)$ are isomorphic, up to an equivalence from  \cite{Ko78} between their target categories; see also \cite[Proposition 2]{AB}.	
			
			It would be natural to establish an analog of such a realization in the super case. Indeed,  set $\nu:=-\rho_\oa$ and define the algebra $B:=\bigoplus_{\la, \mu\in \Lnua} \Hom_{\mf g}(P(\mu), P(\la))$ for a quasireductive Lie superalgebra of type I as in Section \ref{sect::92}.
		   Then we can formulate an analog of Soergel's functor $\mathbb V$ as the functor $\mathbb V^{sup}: M\mapsto \bigoplus_{\la\in \Lnua}\Hom_\g(P(\la),M)$ sending modules in $\mc O$ to (right) $B$-modules.
		    We have the following corollary:
		   \begin{cor} \label{cor::42}  Let $\g$ be a quasireductive Lie superalgebra of type I with non-singular $\zeta$ (i.e., $W_\zeta=W$). Retain the notations above.
		  Then there is a full embedding $\iota$ from $\mc W(\zeta)$ to the category of right $B$-modules such that $\iota\circ \Gamma_\zeta(-)\cong \mathbb V^{sup}$.
		   \end{cor}
	   \begin{proof}
	   	By Lemma \ref{lem::116} and Corollary \ref{lem::51}, the functor $\mathbb V^{sup}$  satisfies the universal property of Serre quotient functor. The conclusion follows from Corollary \ref{coro::univprop}.   \end{proof}

 \subsection{Categorification}

In  this section, we assume that $\g=\gl(m|n)$. Denote by $\ep_1,\cdots,\ep_{m+n}$ the basis in $\h^*$ dual to the standard basis of the Cartan subalgebra $\h$ of diagonal matrices.

Recall $\mathbb T^{m|n}$ and $\mathbb T^{m|n}_\zeta$ from \eqref{eq:Fock} and \eqref{eq:Fock:z}. Let $S_\zeta$ be the $q$-symmetrizer defined in \eqref{eq:q:symm}.

 Recall the Casimir element $\Omega$ for $\g$, defined as follows:
 \begin{align*}
 &\Omega = \sum_{i,j=1}^{m+n} (-1)^{|e_{ij}|}e_{i,j}\otimes e_{j,i}\in \g\otimes \g.
 \end{align*}
 Let ${\bf V}:= \C^{m|n}$ be the natural module over $\g$ with its dual $\bf V^\ast$. Then the action of $\Omega$ and the action of $\g$ commute on both ${\bf V}\otimes M$ and ${\bf V}^\ast\otimes M$, for any $M\in \mc O$. We define $\texttt{F}_i$ (respectively $\texttt{E}_i$) to be the endofunctor on $\mc O$ defined via taking the $i$-eigenspace (respectively $n-m-i$-eigenspace) of the action of $\Omega$ on ${\bf V}\otimes M$ (respectively ${\bf V}^\ast\otimes M$). These give exact {functors} and they induce an action of the Lie algebra $\mathfrak{sl}_\infty$ both on  $[\mc O]$ and on $[\ov{\mc O}]$.

We first recall the standard monomial basis and dual canonical basis of $\mathbb T^{m|n}$ from \eqref{eq:Mf} and {Proposition} \ref{prop:ex:can1}, respectively. Denote their $\Z[q,q^{-1}]$-span by $\mathbb T^{m|n}_{\Z[q,q^{-1}]}$ and denote its specification at $q=1$ by $\mathbb T^{m|n}_{\Z}$. Their topological completion are denoted by $\widehat{\mathbb T}^{m|n}_{\Z[q,q^{-1}]}$ and $\widehat{\mathbb T}^{m|n}_{\Z}$, respectively. Similar notations apply to $\mathbb T^{m|n}_\zeta$.

The relevance of the canonical and dual canonical basis $\{T_f|f\in\Z^{m|n}\}$ and $\{L_f|f\in\Z^{m|n}\}$ in the representation theory of $\gl(m|n)$ is given in the following Brundan-Kazhdan-Lusztig conjecture \cite{Br1} that was first established in \cite{CLW2}, and then later upgraded to a $\Z$-graded version in \cite{BLW}. Before stating it, let us make the following identification between the integral weight lattice $\sum_{i=1}^{m+n}\Z\ep_i$ of $\gl(m|n)$ and $\Z^{m|n}$ via the bijection: $\la\rightarrow f_\la\in\Z^{m|n}$, where
\begin{align}\label{eq:wt:bij}
f_\la(i):=(\la+\rho,\ep_i),
\end{align}
for $\la\in \sum_{i=1}^{m+n}\Z\ep_i$. Here $\rho=\sum_{i=1}^m(m-i+1)\ep_i+\sum_{j=m+1}^{m+n}(m-j)\ep_j$ is the super Weyl vector.

Define $\psi:\mathbb T^{m|n}_\Z \rightarrow [\mc O]$ to be the $\Z$-linear isomorphism uniquely determined by $\psi(M_{f_\la}):=[M(\la)]$. Then the topological completion $\widehat{\mathbb T}^{m|n}_\Z$ of $\mathbb T^{m|n}_\Z$ induces a topological completion $\widehat{[\mc O]}$ of $[\mc O]$ via $\psi$. By slight abuse of notation, we will still denote the induced map by
\begin{align}\label{map:psi}
\psi:\widehat{\mathbb T}^{m|n}_\Z \rightarrow \widehat{[\mc O]}.
\end{align}

\begin{thm}\label{thm:BKL:conj}\emph{(}{\cite[Conjecture 4.32]{Br1}, \cite[Section 8]{CLW2}}\emph{)} Let $\psi$ be as above. We have $\psi(L_{f_\la})=[L(\la)]$ and $\psi(T_{f_\la})=[T(\la)]$, for $\la\in\sum_{i=1}^{m+n}\Z\ep_i$. In particular, the character of the irreducible module $L(\la)$ and that of the tilting module $T(\la)$ are given respectively by
\begin{align*}
\ch L(\la) = \sum_{\mu}\ell_{f_\mu,f_\la}(1)\ch M(\mu),\\
\ch T(\la) = \sum_{\mu}t_{f_{\mu},f_\la}(1)\ch M(\mu).
\end{align*}
Furthermore, the $\mf{sl}_\infty$-actions on $\widehat{\mathbb T}^{m|n}_\Z$ (obtained by specializing the $U_q(\mf{sl}_\infty)$-action to $q=1$) and on $\widehat{[\mc O]}$ are compatible. That is, $\psi$ is an isomorphism of $\mf{sl}_\infty$-modules.
\end{thm}

Recall that we have the standard basis $\widetilde{N}_f$ and the dual canonical basis $\mc L_f$ on $\widehat{\mathbb T}^{m|n}_\zeta$ from \eqref{eq:Nf} and Section \ref{sec:qsymm}.
 The following theorem gives a categorification of $q$-symmetrized Fock space via  category $\ov{\mc O}$.

\begin{thm}\label{thm:T:to:Obar}
	\begin{itemize}
		\item[(1)] The $\Z$-linear isomorphism $$\psi_\zeta: \widehat{\mathbb T}^{m|n}_{\zeta,\Z} \xrightarrow{\cong} \widehat{[\ov{\mc O}]},$$ uniquely determined by $\psi_\zeta(\widetilde{N}_{f_\la})=[\ov \nb(\la)]$, satisfies
		\begin{align*}
		\psi_\zeta(N_{f_\la})=[\Delta(\la)],\quad\psi_\zeta(\mc T_{f_\la})=[T^{\OI}(\la)],	\quad\psi_\zeta( \mc L_{f_\la})=[\pi(L(\la))],
		\end{align*}
		for all $\la \in \Lnua$.
Furthermore, $\psi_\zeta$ is compatible with the action of $\mf{sl}_\infty$, i.e., $\psi$ is an isomorphism of $\mf{sl}_\infty$-modules.
{}				
\item[(2)] Both functors $  \Gamma_\zeta(-)$ and $F_\zeta(-)$ categorify the map of $\phi_\zeta:\widehat{\mathbb T}^{m|n}_\Z\rightarrow\widehat{\mathbb T}^{m|n}_{\zeta,\Z}$ in \eqref{def:phiz}.
	\end{itemize}
\end{thm}
\begin{proof}
First, we observe that, if $\mu\not\in\Lambda(\nu)$, then $\pi(L(\mu))=0$.

Let $\la\in\Lambda(\nu)$. Since, for any $\tau\in W_\nu$, the kernel of the surjection $\Coind_{\g_{\le 0}}^\g\ov{\nabla}'(\tau\cdot\la)\onto\Coind_{\g_{\le 0}}^\g\ov{\nabla}'(\la)$ cannot contain a composition factor with highest weight in $\Lambda(\nu)$ and $\pi$ is exact, we conclude that $\pi(\Coind_{\g_{\le 0}}^\g\ov{\nabla}'(\la))=\pi(\Coind_{\g_{\le 0}}^\g\ov{\nabla}'(\tau\cdot\la))$.

Now, consider the linear map $\psi:\widehat{\mathbb T}^{m|n}_\Z\rightarrow \widehat{[\mc O]}$ from \eqref{map:psi}. The quotient functor $\pi:\mc O\rightarrow \ov{\mc O}$ induces a linear map for the completed Grothendieck groups, which we shall also denote by $\pi:\widehat{[\mc O]}\rightarrow\widehat{[\ov{\mc O}]}$ by slight abuse of notation. Recall also the map $\phi_\zeta:\widehat{\mathbb T}^{m|n}\rightarrow\widehat{\mathbb T}^{m|n}_\zeta$ defined in \eqref{def:phiz}. We shall denote the specialization at $q=1$ also by $\phi_\zeta:\widehat{\mathbb T}^{m|n}_\Z\rightarrow\widehat{\mathbb T}^{m|n}_{\zeta,\Z}$. Finally, recall the linear isomorphism $\psi_\zeta:\widehat{\mathbb T}^{m|n}_{\zeta,\Z}\rightarrow\widehat{[\ov{\mc O}]}$ above. For $f_\la\in\Z^{m|n}$, we let $\tau\in W_\nu=W_\zeta$ such that $f_\la\cdot\tau\in\Lambda(\nu)$. We compute
\begin{align*}
\pi\circ\psi(M_{f_\la})=\pi([\Delta'(\la)])=\pi([\nabla'(\la)])=[\ov{\nabla}(\tau\cdot\la)].
\end{align*}
On the other hand,
\begin{align*}
\psi_\zeta\circ\phi_\zeta(M_{f_\la})=\psi_\zeta(\widetilde{N}_{f_\la\cdot\tau})=[\ov{\nabla}(\tau\cdot\la)].
\end{align*}
Thus, we have a commutative diagram:
$$\xymatrixcolsep{3pc} \xymatrix{
			\widehat{\mathbb T}^{m|n}_{\Z}  \ar[r]^-{\psi}  \ar@<-2pt>[d]_{\phi_\zeta} &   \widehat{[\mc O]}   \ar@<-2pt>[d]_{\pi} \\
			\widehat{\mathbb T}^{m|n}_{\zeta,\Z} \ar[r]^-{\psi_\zeta}  &  [\widehat{\ov{\mc O}}]}$$
Using Proposition \ref{prop::15}, Theorem \ref{thm:szeta:fock} and Theorem \ref{thm:BKL:conj} we compute for $\la\in\Lambda(\nu)$:
\begin{align*}
\psi_\zeta(\mc T_{f_\la})&=\psi_\zeta\circ\phi_\zeta(T_{f_\la\cdot w_0^\nu})=\pi\circ\psi(T_{f_\la\cdot w_0^\nu})=\pi([T(w_0^\nu\cdot\la)])=[T^{\ov{\mc O}}(\la)].\\
\psi_\zeta(\mc L_{f_\la})&=\psi_\zeta\circ\phi_\zeta(L_{f_\la})=\pi\circ\psi(L_{f_\la})=\pi([L(\la)]).
\end{align*}
Finally, both maps $\phi_\zeta$ and $\pi$ are $\mf{sl}_\infty$-module homomorphisms, and thus so is $\psi_\zeta$. This completes the proof of Part (1).

Part (2) now follows from Part (1), Theorem \ref{prop::univprop} and Corollary~\ref{coro::univprop}.
\end{proof}

\begin{cor}\label{cor:duality1}
	For any $\la,\mu \in \Lnua$, we have
	\begin{align*}
	&(T^{\OI}(\la): \Delta(\mu)) = \left(P(-(w_0^\nu(\la+\rho)-\rho): M(-w_0^\nu(\mu+\rho)-\rho)\right).
	\end{align*}
	\end{cor}
	
\begin{proof}
By Theorems~\ref{thm:T:to:Obar} and \ref{pi:can:basis}, we deduce that
\begin{align*}
\left(T^{\OI}(\la):\Delta(\mu)\right)=&t_{f_\mu\cdot w_0^\nu,f_\la\cdot w_0^\nu}(1)\\
=&\left(T(w_0^\nu(\la+\rho)-\rho):M(w_0^\nu(\la+\rho)-\rho)\right)\\
=&\left(P(-w_0^\nu(\la+\rho)-\rho):M(-w_0^\nu(\la+\rho)-\rho)\right).
\end{align*}
The second equality is due to Theorem \ref{thm:BKL:conj}, while the last equality follows from Ringel duality for $\mc O$ \cite[(7.4)]{Br04}, which reads $\left(T(\kappa):M(\gamma)\right)=\left(P(-\kappa-2\rho):M(-\gamma-2\rho)\right)$, for $\kappa,\gamma\in\h^*$.
\end{proof}

\begin{rem} \label{rem::44}
In the special case of $W_\zeta=\mf{S}_m\times\mf{S}_n$, by the super version of Skryabin's theorem \cite[Remarks 3.9--3.10]{Zh}, the category of Whittaker modules is equivalent to the finite-dimensional module category over the finite $W$-superalgebra corresponding to the Lie superalgebra $\gl(m|n)$ associated to the (even) principal nilpotent element.
	More precisely, in this case, there is a {\em Whittaker co-invariant functor} $H_0$ from $\mc O$ to a certain category of  modules over above $W$-algebra which is reminiscent of Backelin's functor see \cite[Section 3]{BrG}.  One of the most significant property of $H_0$ is that it is the Serre quotient functor, namely, $H_0$ satisfies the universal property of the Serre quotient (see \cite[Theorem 4.8]{BrG}  and \cite[Theorem 4.10]{BLW}).

If we denote by $\mc W_\zeta$ the target category consisting of all $W$-algebra modules of the form $H_0(M)$ for $M\in \mc O$, then by Proposition \ref{prop::univprop} there is a unique equivalence $E_\zeta(-): \mc W(\zeta)\xrightarrow{\cong} \mc W_\zeta$ making the following diagram commute:
	$$\xymatrixcolsep{3pc} \xymatrix{
  &	\mc O  \ar[ld]_-{F_\zeta(-)}  \ar@<-2pt>[rd]^{H_0(-)} &       \\
	\mc W(\zeta) \ar[rr]_-{E_\zeta(-)}^{\cong} & &  \mc W_\zeta}$$

By the super Skryabin equivalence, $E_\zeta(-)$ is the {\em Whittaker functor} $Wh(-)$ which gives {\em Whittaker vectors} of modules as in \cite[Section 3.7]{Zh}, and which was first introduced by Kostant \cite{Ko78} in the setting of semisimple Lie algebras.  Therefore, our functor $F_\zeta(-)$, which is isomorphic to $\Gamma_\zeta(-)$, can be considered as a generalization of the Whittaker co-invariants functor in the singular setting $\zeta$, i.e.,  $W_\zeta\neq W$.

It is proved in \cite{BrG} that $H_0(-)$ has properties similar to Soergel's functor $\mathbb V$ for semisimple Lie algebras.  In particular,  $H_0(-)$ is fully faithful on projective modules by \cite[Corollary 4.9]{BrG};  see also \cite[Theorem~4.10]{BLW}. This allows one to show that these functors satisfy an analogue of the Soergel's {\em Struktursatz} in the setting of the category $\mc O$ for Lie algebras. Namely, $\Gamma_\zeta(-)\cong  F_\zeta(-)$ are fully faithful on projective modules. It perhaps worth pointing out that this holds for any type I Lie superalgebras.

\end{rem}

\begin{cor} \label{cor::47} Let $\g$ be a quasireductive Lie superalgebra of type I. Suppose that $\zeta$ is non-singular, namely, $W_\zeta=W$. Then the functors
	\begin{align*}
	 {\Gamma_\zeta(-)\cong~}F_\zeta(-):\mc O\rightarrow W(\zeta),\qquad
	\mc L(M(\nu),-):\mc O\rightarrow \mc B_\nu,
	\end{align*}  are fully faithful on projectives.
\end{cor}
\begin{proof} We recall the functor $\mathbb V^{sup}$  from Section \ref{ref::41}.
By \cite[Theorem 7.2]{AM}, $\mathbb V^{sup}$ is fully  faithful on projective modules. The conclusion   follows by Corollary \ref{coro::univprop} or Corollary \ref{cor::42}.
\end{proof}

	\subsection{Examples} \label{sect::gl11ex} In this section, we consider the Whittaker category  $\mc W(\zeta)$ and its structural modules in the case of the general linear Lie superalgebra $\gl(m|n)$.
	
	Following \cite{Ko78, MS}, a character $\zeta$ is said to be {\em non-singular} (or {\em regular}) if $W_\zeta =W$. Throughout this subsection, we assume that $\zeta$ is non-singular with  $W_\nu=W_\zeta (=W)$.  In this case, we give a explicit combinatorial description of the algebras that are Morita equivalent to $\mc W(\zeta)$, for $\g=\gl(1|n)$.

\subsubsection{Structural modules in $\mc W(\zeta)$ for $\g=\gl(1|2)$.}	In this section, we will give a detailed description in the case of the general linear Lie superalgebra  $\gl(1|2)$.

	For any $\la\in \Lnua$, we define the following Whittaker modules
	\begin{align*}
	&P(\la,\zeta):= \Gamma_\zeta(P(\la)), \\
	&\Delta(\la,\zeta):= \Gamma_\zeta(K(\Ind_{\mf p_\oa}^{\g_\oa} P(\mf l_\nu, \la))).
	\end{align*}  By Proposition \ref{coro::univprop}, $P(\la,\zeta)$, $\Delta(\la,\zeta)$ and $M(\la,\zeta)$ are the projective  cover of $L(\la,\zeta)$, the corresponding standard object and the corresponding proper standard object in the category $\mc W(\zeta)$, respectively.  In addition, $P(\la,\zeta)$ is an indecomposable tilting module in $\mc W(\zeta)$ by Proposition \ref{prop::15}.   We have the following description of $P(\la,\zeta)$:

\begin{prop} Let $\la \in \Lnua$.
		\begin{itemize}
			\item[(1)]	Suppose that $\la$ is typical. Then we have
			\begin{align}
				&M(\la,\zeta) =L(\la,\zeta),~P(\la,\zeta) = \Delta(\la,\zeta), \label{eq::717}
			\end{align}  moreover, $\Delta(\la,\zeta)$ has a proper standard filtration of length $|W\cdot \la|.$
		\item[(2)] Suppose that $\la$ is atypical with  $\langle \la+\rho, \alpha^\vee\rangle =0$, for some odd root $\alpha$.  Then we have the following non-split short exact sequences:
		\begin{align}
			&0\rightarrow L(\la-\alpha,\zeta)\rightarrow M(\la,\zeta) \rightarrow L(\la,\zeta)\rightarrow 0, \label{eq::718}\\
			&0\rightarrow \Delta(\underline{\la+\alpha})\rightarrow P(\la,\zeta) \rightarrow \Delta(\la,\zeta) \rightarrow 0, \label{eq::719}
		\end{align} where  $ \underline{\la+\alpha}$ denotes the unique anti-dominant weight in $W\cdot (\la+\alpha)$.
		\end{itemize}
\end{prop}

\begin{proof}
	We first show that both $M(\la,\zeta)$ and $P(\la,\zeta)$ have simple socles. To see this, observe that $P(\la,\zeta)$ is an indecomposable injective module, and so it has simple socle. Also, we note that $M(\la,\zeta)\cong K(Y_\zeta(\la,\zeta))$, where $Y_{\zeta}(\la,\zeta)$ is the simple Whittaker $\g_\oa$-module introduced by Kostant in \cite{Ko78}; see \eqref{eq::31} for the definition. Consequently, $M(\la,\zeta)$ has simple socle by \cite[Lemma~3.2]{CM}.
		
	The identities in \eqref{eq::717} follow from \cite{Gor2}.  The short exact sequence in  \eqref{eq::718} is  taken from \cite[Proposition 24]{Ch21}, or, alternatively, it follows from the exactness of $\Gamma_\zeta$ and the character formulas in \cite[Section 9.4]{ChWa08}. The exact sequence in  \eqref{eq::719} is a direct consequence of Proposition \ref{prop::BGG}. The exact sequences in \eqref{eq::718} and \eqref{eq::719} are non-split since $M(\la,\zeta)$ and $P(\la,\zeta)$ have simple socles.
\end{proof}

	\subsubsection{The category $\mc W(\zeta)$ for  $\g=\gl(1|n)$.}
		In the general case $\g=\gl(m|n)$ with arbitrary $\zeta$, the category $\mc W(\zeta)$ is equivalent to the   category of finite-dimensional locally unital right modules over a locally unital algebra  by Lemma \ref{lem01} and Lemma \ref{lem::116}. In the case when $\zeta$ is non-singular, $\mc W(\zeta)$ is equivalent to a certain infinite {\em tower of cyclotomic quiver Hecke algebras}; see \cite[Section 4, Lemma 4.6]{Br}.
	
Consider the special cases $\g= \gl(1|n)$ with non-singular $\zeta$. By the analysis in  \cite[Examples 4.7, 4.8]{Br}, the category $\mc W(\zeta)$ is equivalent to the category of finite dimensional locally unital right modules over the path algebra of the following infinite quiver
	\begin{center}
		\hskip -3cm \setlength{\unitlength}{0.16in}
		\begin{picture}(24,3)
			\put(7.8,1){\makebox(0,0)[c]{$\bullet$}}
			\put(10.2,1){\makebox(0,0)[c]{$\bullet$}}
			\put(15.2,1){\makebox(0,0)[c]{$\bullet$}}
			\put(12.8,1){\makebox(0,0)[c]{$\bullet$}}
			\put(9,1.1){\makebox(0,0)[c]{$\longrightarrow$}}
			\put(9,0.7){\makebox(0,0)[c]{$\longleftarrow$}}
			\put(11.5,1.1){\makebox(0,0)[c]{$\longrightarrow$}}
			\put(11.5,0.7){\makebox(0,0)[c]{$\longleftarrow$}}
			\put(14,1.1){\makebox(0,0)[c]{$\longrightarrow$}}
			\put(14,0.7){\makebox(0,0)[c]{$\longleftarrow$}}

			
			\put(6.5,0.95){\makebox(0,0)[c]{$\cdots$}}
			\put(16.5,0.95){\makebox(0,0)[c]{$\cdots$}}
			
			\put(9,1.8){\makebox(0,0)[c]{\tiny$x_{i-1} $}}
			\put(9,0){\makebox(0,0)[c]{\tiny$y_{i-1}$}}
			\put(11.5,1.8){\makebox(0,0)[c]{\tiny$x_{i} $}}
			\put(11.5,0){\makebox(0,0)[c]{\tiny$y_{i}$}}
			\put(14,1.8){\makebox(0,0)[c]{\tiny$x_{i+1} $}}
			\put(14,0){\makebox(0,0)[c]{\tiny$y_{i+1}$}}
		\end{picture}
	\end{center}
	modulo the relations
	\begin{align*}
		&x_{i+1}\cdot x_i =y_{i}\cdot y_{i+1} =0, \\ &(y_{i+1}x_{i+1})^{n-(n-1)\delta_{i,-1}} =-(x_iy_i)^{n-(n-1)\delta_{i,0}},
	\end{align*} for all $i\in \mathbb{Z}$. In particular, in the case of $\gl(1|2)$, we get the defining relations
	\begin{align*}
		&x_{i+1}\cdot x_i =y_{i}\cdot y_{i+1} =0, \\ &y_{i+1}x_{i+1} =-x_iy_i, \text{for $i\neq 0,1$,} \\
		&y_0x_0=-(x_{-1}y_{-1})^2,\\
		&(y_1x_1)^2 = -x_0y_0,
	\end{align*} for all $i\in \mathbb{Z}$.

	From Section~\ref{App::9}, the category $\mc W(\zeta)$, for $\g=\gl(m|n)$,  has a graded lift. In particular, in the case when $\g=\gl(1|n)$ and $\zeta$ is non-singular, by \cite[Examples 4.7, 4.8]{Br} the algebra described above is positively graded  with the grading
	\begin{align}
	&\deg(x_i)=\deg(y_i) = 1+(n-1)\delta_{i,0}, \label{eq::grad}
	\end{align} for any $i\in \Z$. By Lemma \ref{lem::18} it follows  that all simple, standard and
	proper standard modules in $\mc W(\zeta)$ are gradable with respect to the grading given by \eqref{eq::grad}.

\section{Appendix A. Structural modules in $\Opres$} \label{sect::81}

The goal of this section is to describe the simple, standard and proper standard objects in $\Opres.$

For a given module $M\in \mc O$, let $\text{Tr}_{\nu}(M)$ denote the  sum of the  images of all homomorphisms
from the {$\nu$-}admissible projective modules to $M$.  For a fixed $\la \in \Lnua$,  denote by $A(\la)$ the kernel of the
canonical epimorphism $P(\la) \onto M(\la)$. Consider the following modules: \begin{align}
	&S(\la):=   P(\la)/\text{Tr}_{\nu}(\rad   P(\la)),\label{eq::81}\\
	&D(\la):=   P(\la)/\text{Tr}_{\nu}(A(\la)), \label{eq::82}
\end{align}
Define $Q(\la)$ to be the quotient of $P(\la)$ modulo the sum of the images of all
homomorphisms $P(\mu)\rightarrow P(\la)$, where $\mu\in\Lnua$ and $\mu>\la$.
Denote the natural projection by
\begin{align}\label{varpi:map}
\varpi: P(\la)\rightarrow Q(\la).
\end{align}

By construction, the modules $S(\la), D(\la)$ and $Q(\la)$ lie in $\Opres$. We recall that $\mathbf T_{w_0^\nu}$ denotes the twisting functor on $\mc O$ associated to $w^\nu_0$; see Section \ref{sect::621}.
The following proposition describe these modules as structural objects with respect to
the stratified structure on $\Opres$:

\begin{prop} \label{prop::45} Suppose that $\g$ is one of the Lie superalgebras from the series   in \eqref{eq::claA}-\eqref{eq::claP}.
  For any $\la \in \Lnua$, we have
  \begin{align}
  	&S(\la)\cong {\bf T}_{w_0^\nu}L(\la)\text{ and } \pi(S(\la))\cong \pi(L(\la)), \label{eq::84} \\
  	&D(\la)\cong {\bf T}_{w_0^\nu}M(\la)\text{ and } \pi(D(\la))\cong \ov \Delta(\la), \label{eq::85} \\
  	&\pi(Q(\la))\cong \Delta(\la). \label{eq::86}
  \end{align}
\end{prop}

	We will prove  Proposition \ref{prop::45} using the following lemmas.
	\begin{lem} \label{lem:::46}
		For any weight $\la\in \h^\ast$, there is a direct  summand $P$ of $\Res P(\la)$ such that  $P\cong P_0(\la)$ and $U(\g)P =P(\la)$.
	\end{lem}
	\begin{proof}
		We first decompose $\Res P(\la) =\bigoplus_{i=1}^{\ell}P_0(\mu_i)$. Consider the canonical quotients
		\begin{align}
			&p: P(\la) \onto   L(\la),\label{eq::04} \\
			&q:K(L_0(\la))\onto  L(\la), \label{eq::05}
		\end{align} where $p$ factors through $q.$
		By \eqref{eq::05}, the $\g_\oa$-module $\Res L(\la)$ has a quotient isomorphic to $L_0(\la)$. Therefore, $\Res P(\la)$ has
		$L_0(\la)$ as a quotient as well. This implies that there is $1\leq i\leq \ell$ such that $\mu_i=\la$, and, if we let $P:=P_0(\mu_i)$, then $P\not \subseteq \ker(p)$. The fact that $U(\g)P =P(\la)$ then follows by combining  the uniqueness of the maximal submodule of $P(\la)$ with $p(P)\neq 0$.
	\end{proof}

For any $\la \in \Lnua$, define
\[\Lnua_{>\la}:=\{\mu \in \Lnua \text{ such that } \mu>\la\}.\]

\begin{lem} \label{lem:::47}
	For each $\la \in \Lambda(\nu)$, there is a short exact sequence of $\g$-modules as follows:
	\begin{align*}
		&0\rightarrow X \rightarrow \Ind P_0(\la) \rightarrow K(P_0(\la)) \rightarrow 0,
	\end{align*} Here $X = U(\g_{-1})\cdot (\g_1U(\g_1))\otimes P_0(\la)\subset \Ind P_0(\la)$ has a Kac flag, subquotients of which are isomorphic to direct sums of the  modules of the form $K(P_0(\gamma))$, where $\gamma \in \xnula$.
	
	Furthermore, $X$ is an epimorphic image of a direct sum of modules of the form $P(\gamma)$, where $\gamma \in \xnula.$
\end{lem}
\begin{proof} For any $k>1$,
	we may observe that
	\begin{align*}
		&\Hom_{\g_0}(\Lambda^k \mf g_1\otimes P_0(\la), L_0(\gamma)) = [ \Lambda^k \mf g_1^\ast\otimes L_0(\gamma): L_0(\la)],
	\end{align*} which implies that $\Lambda^k(\g_1)\otimes P_0(\la)$ decomposes into a  direct sum of modules $P_0(\gamma)$, with $\gamma\in \xnula$. In particular, the $\g_{\geq 0}$-module $\Ind_{\g_0}^{\g_{\geq 0}}P_0(\la)$ has a filtration, subquotients of which are direct sums of  $P_0(\gamma)$  such that $\g_1P_0(\gamma)=0$ and $\gamma\in \xnula\cup \{\la\}$. It follows that the module
	\begin{align*}
		&\Ind P_0(\la) \cong \Ind_{\g_{\geq 0}}^\g \Ind_{\g_0}^{\g_{\geq 0}}  P_0(\la),
	\end{align*} admits a Kac flag, subquotients of which are of the form $K(P_0(\gamma))$, where we have $\gamma \in \xnula\cup \{\la\}$. By \cite[Theorem 51]{CCM}, each $K(P_0(\gamma))$ has simple top, which is, automatically, isomorphic to $L(\gamma)$. Consequently,  $K(P_0(\gamma))$ is a quotient of $P(\gamma)$. This completes the proof.
\end{proof}

\begin{proof}[Proof of Proposition \ref{prop::45}]	The isomorphisms in \eqref{eq::84}, \eqref{eq::85} are proved in \cite[Theorem~12]{Ch21}. It remains to prove the isomorphism in \eqref{eq::86}. To see this, we first recall the definitions of the Kac functor $K(-)$ and the module $\Delta_0(\la)$ from Sections \ref{sect::5}, \ref{sect::51}. We shall show that $K(\Delta_0'(\la))\cong Q(\la)$.
	
	Let $Q_0(\la)$ denote the quotient of $P_0(\la)$ by the sum of the traces of all $P_0(\mu)$, where $\mu \in \xnula$, in  $P_0(\la)$. 

Let $P$ be the $\g_0$-submodule of $\Res P(\la)$ described in Lemma \ref{lem:::46}. Recalling the map $\varpi$ from \eqref{varpi:map}, we see that $\varpi(P)\neq 0$. In particular, the $\g_0$-submodule $\varpi(P)$ is a quotient of $P\cong P_0(\la)$. We are going to show that there is an  $\g_\oa$-epimorphism
	\begin{align}
		&Q_0(\la) \onto \varpi(P), \label{eq:::06}
	\end{align} equivalently, we shall show that $[\varpi(P):L_0(\gamma)]=0$, for any $\gamma \in \xnula$.
	
	We first claim that $\g_1 \varpi(P)=0,$ namely, that $\g_1 P\in \ker(\varpi)$. To see this, we consider the epimorphism  $\Ind(P) \onto U(\g)P$ sending $x\otimes v$ to $xv$, for any $x\in U(\g)$ and $v\in P$.
	By Lemma \ref{lem:::47}, the $\g$-module $U(\g_{-1})\g_1U(\g_1)P$ is an epimorphic image of a direct sum of $P(\gamma)$'s, where $\gamma \in \xnula$. Consequently, we have $\g_1P\subseteq \ker(\varpi)$, as desired.
	
	Next,  we suppose that $[\varpi(P):L_0(\gamma)]>0$, for some weight $\gamma$.  Then there is a $\g_0$-submodule $Y$ of $\varpi(P)$ having a quotient $L_0(\gamma)$. Consider
	\begin{align*}
		&\Hom_\g(K(Y), Q(\la)) =\Hom_{\g_0}(Y, Q(\la)^{\g_1}) \supseteq \Hom_{\g_0}(Y, \varpi(P)) \neq 0,
	\end{align*} where $Q(\la)^{\g_1}$ denotes the $\g_1$-invariants of $Q(\la)$. This leads to a non-zero homomorphism from $P(\gamma)$ to $Q(\la)$. Consequently, $\gamma$ cannot be a weight in $\xnula$. This proves \eqref{eq:::06}.

	Finally, we note that $U(\g)\varpi(P)=Q(\la)$ by Lemma \ref{lem:::46}. Then we consider
	\begin{align*}
		&\Hom_\g(K(\Delta'_0(\la)), Q(\la)) =\Hom_{\g_0}(\Delta'_0(\la), Q(\la)^{\g_1}).
	\end{align*}
	Since $Q_0(\la)$ is isomorphic to the parabolically induced $\g_\oa$-module $\Delta'_0(\la)$ from the projective-injective $\mf l_\nu$-module $P(\mf l_\nu, \la)$, we obtain an epimorphism from $K(\Delta'_0(\la))$ to $Q(\la)$.
	
	Finally, it remains to show  that $K(\Delta_0'(\la))$ is an epimorphic image of  $Q(\la)$. By \cite[Theorem 51]{CCM}, the top of $K(\Delta_0'(\la))$ is simple, and so it is isomorphic to  $L(\la)$. Then we get an epimorphism $f: P(\la)\onto K(\Delta_0'(\la))$. It remains to show that, if $\eta\in \Lnua$ is such that $[K(\Delta_0'(\la)): L(\eta)]\neq 0$, then $\eta \not > \la$. We will show the stronger statement that $\eta \leq \la$. To see this, we observe that
	\begin{align*}
	&\ch K(\Delta_0'(\la)) = \sum_{\mu\in W_\nu\cdot \la} \ch M(\mu),
	\end{align*}  which implies that there is $\mu \in W_\nu\cdot \la$ such that  $[M(\mu):L(\eta)]\neq 0$. Since $\mu \in W_\nu\cdot \la$ and $\la, \eta\in \Lnua$, we  have $[M(\mu)/M(\la):L(\eta)]=0$ and thus $[M(\mu):L(\eta)] = [M(\la):L(\eta)]$. 	The conclusion follows.
\end{proof}

\begin{ex}
	Consider $\g=\mf{sl}(2)$ with $W_\nu =W$. Let $\la \in \Lnua$. If $\la$ is regular, then $S(\la)=D(\la)$ is isomorphic to the  the dual Verma module of highest weight $w_0^\nu\cdot\la$ and $Q(\la)= P(\la)$. If $\la$ is singular, then $S(\la) =D(\la) =Q(\la)$.
\end{ex}

\section{Appendix B. A realization of $\OI$ and its graded version.} \label{App::9}

In this appendix, we realize $\OI$ as the category of  finite-dimensional (locally unital) modules over a locally unital algebra.
In the case $\g=\gl(m|n)$, we study the graded version of $\OI$ and show that all structural modules in $\OI$ have graded lifts.

\subsection{Morita equivalence for $\mc O$}
We define $A$ to be the locally unital algebra of $\mc O$, that is
\begin{align*}
	&A:=\bigoplus_{\la, \mu\in \Lambda} \Hom_{\mf g}(P(\la), P(\mu)).
\end{align*}
For $\la \in \Lnua$, denote by $1_\la$ the identity endomorphism on $P(\la)$. Then $\{1_\la|~\la\in \Lambda\}$ forms a complete
set of mutually orthogonal idempotents and $A = \bigoplus_{\la,\mu\in\Lambda}1_\mu A 1_\la$.  Define $\text{mof}$-$A$ to
be the category of finite-dimensional locally unital right $A$-modules.  This means that the objects of $\text{mof}$-$A$ are
finite dimensional  right $A$-modules $M$ such that
$M= \bigoplus_{\la\in \Lambda} M1_\la$ and morphisms in $\text{mof}$-$A$
are homomorphism of $A$-modules. Then we have a Morita type equivalence  (c.f.{~\cite[Section 2.1]{BLW}})
\begin{align*}
	&T(-):\mc O\xrightarrow{\cong }\text{mof-}A,
\end{align*} via the functor $T(M):=\bigoplus_{\la \in \Lambda} \Hom_\g(P(\la), M)$, where $M\in \mc O$.  We may note that $T(P(\la)) = 1_\la A$ and $\dim 1_\la A, \dim A1_\la<\infty$.

In particular,  for a given $\la\in \Lambda$ (respectively $\mu\in \Lambda$), there are only finitely many $\mu\in \Lambda$ (respectively $\la\in \Lambda$) such that $\Hom_\g(P(\la), P(\mu))\neq 0$.  Hence $A= \bigoplus_{\la,\mu\in \Lambda}\Hom_A(1_\la A,1_\mu A)$.

\subsection{Morita equivalence for $\OI$}\label{sect::92} We define a subalgebra $B$ of $A$ as follows
\begin{align}
	&B:=\bigoplus_{\la, \mu\in \Lnua} \Hom_{\mf g}(P(\mu), P(\la))=\bigoplus_{\la,\mu\in \Lnua} 1_\la A 1_\mu.
\end{align}
Similarly, we define $\mofB$ to be the category of all finite dimensional locally unital right $B$-modules.

\begin{lem}\label{lem::116} There is an equivalence
	$\ov{\mc O}\cong \emph{\mofB}$.
\end{lem}
\begin{proof} Let $\ov{\mofA}$ denote the Serre quotient of $\mofA$ by the Serre subcategory $I_\nu$ of all modules $M\in \mofA$ which satisfy $M1_\la =0$, for all $\la \in \Lnua$. We observe that $I_\nu=T(\mc I^\nu)$. Therefore, $\ov{\mofA}$ is equivalent to $\ov{\mc O}$ through the Morita equivalence functor $T$. Therefore we have a commutative diagram
	$$\xymatrixcolsep{3pc} \xymatrix{
		\mc O \ar[r]^-{ T}  \ar@<-2pt>[d]_{\pi} &   \mofA  \ar@<-2pt>[d]_{\pi^A} \\
		\ov{\mc O} \ar[r]^-{\cong}  &  \ov{\mofA},}$$
	where $\pi^A$ denotes the quotient functor from $\mofA$ to $\overline{\mofA}$.
	
 In order to complete the proof we define a functor $\pi': \mofA\rightarrow \mofB$ by letting
	\begin{align*}
		&\pi': M\mapsto \bigoplus_{\la\in \Lnua} M1_\la.
	\end{align*} We are now in the position to invoke \cite[Lemma A.2.1]{CP19} which says that
 $\pi'$ induces an equivalence $\widetilde{\pi'}: \ov{\mofA}\rightarrow\mofB$ such that $\pi'=  \widetilde{\pi'} \circ \pi^A$.
\end{proof}

By the proof of Lemma \ref{lem::116}, there are  functors $T^B: \ov{\mc O} \xrightarrow{\cong} \mofB$ and $\pi': \mofA \rightarrow \mofB$ that make the following  diagram  commutative:
$$\xymatrixcolsep{3pc} \xymatrix{
	\mc O \ar[r]^-{T}  \ar@<-2pt>[d]_{\pi} &   \mofA  \ar@<-2pt>[d]_{\pi'} \\
	\ov{\mc O} \ar[r]^-{T^B}  &  \mofB.}$$

\begin{cor} \label{lem::51}
The exact functor
\begin{align}
&\pi'\circ T: \mc O\rightarrow \mofB, ~M\mapsto \bigoplus_{\la\in \Lnua}\Hom_\g(P(\la),M) \label{eq::96}
\end{align} 	
satisfies the universal property for the Serre quotient $\OI$.
\end{cor}

\subsection{Graded category $\mc O$ for $\gl(m|n)$} We now consider the case that $\g=\gl(m|n)$. Recall the simple-preserving contragredient duality $(-)^\vee:\mc O\rightarrow \mc O$ from Section \ref{sect::24}. In this case, the algebra $A$ admits a positive grading such that $A$ is a standard Koszul algebra. We refer to \cite[Sections 3,4]{Br} and \cite{BLW} for the details. With respect to this grading, the idempotents $1_\la$, for $\la\in \Lambda$, have degree zero.

 \begin{ex} For $\g=\gl(1|1)$, the corresponding category $\mc O$ is equivalent to the category of finite-dimensional locally unital right modules over the path algebra of the following infinite quiver
 \begin{center}
 	\hskip -3cm \setlength{\unitlength}{0.16in}
 	\begin{picture}(24,3)
 		\put(7.8,1){\makebox(0,0)[c]{$\bullet$}}
 		\put(10.2,1){\makebox(0,0)[c]{$\bullet$}}
 		\put(15.2,1){\makebox(0,0)[c]{$\bullet$}}
 		\put(12.8,1){\makebox(0,0)[c]{$\bullet$}}
 		\put(9,1.1){\makebox(0,0)[c]{$\longrightarrow$}}
 		\put(9,0.7){\makebox(0,0)[c]{$\longleftarrow$}}
 		\put(11.5,1.1){\makebox(0,0)[c]{$\longrightarrow$}}
 		\put(11.5,0.7){\makebox(0,0)[c]{$\longleftarrow$}}
 		\put(14,1.1){\makebox(0,0)[c]{$\longrightarrow$}}
 		\put(14,0.7){\makebox(0,0)[c]{$\longleftarrow$}}

 		
 		\put(6.5,0.95){\makebox(0,0)[c]{$\cdots$}}
 		\put(16.5,0.95){\makebox(0,0)[c]{$\cdots$}}
 		
 		\put(9,1.8){\makebox(0,0)[c]{\tiny$x_{i-1} $}}
 		\put(9,0){\makebox(0,0)[c]{\tiny$y_{i-1}$}}
 		\put(11.5,1.8){\makebox(0,0)[c]{\tiny$x_{i} $}}
 		\put(11.5,0){\makebox(0,0)[c]{\tiny$y_{i}$}}
 		\put(14,1.8){\makebox(0,0)[c]{\tiny$x_{i+1} $}}
 		\put(14,0){\makebox(0,0)[c]{\tiny$y_{i+1}$}}
 	\end{picture}
 \end{center}
 modulo the relations $x_iy_i=y_{i+1}x_{i+1},~x_{i+1}x_{i} =y_{i}y_{i+1}=0,$
  for all $i\in \Z$. The grading is given by $\deg(x_i)=\deg(y_i)=1.$
 \end{ex}

 We denote by  $\gmofA$ the category of graded finite-dimensional unital  right $A$-modules with homogenous homomorphisms of degree zero. Also, we denote by $\mathbb F^A$ the functor from $\gmofA$ to $\mofA$ forgetting the grading.  A object $\dot X\in \gmofA$ is called a {\em graded lift} of  $X\in \mofA$ (respectively $X\in \mc O$)  provided that $\mathbb F^A(\dot X)\cong X$ (respectively $\mathbb F^A(\dot X)\cong T(X)$). Let $\langle 1 \rangle$ denote the degree shift functor defined on graded modules, namely,  for a given graded module $M$, $M\langle 1 \rangle$  has the same module structure but $(M\langle 1 \rangle)_{n}:=M_{n-1}$, for any homogenous subspace $M_n$. For $m\in \Z$, we set $\langle m \rangle:= \langle 1\rangle^m.$

 Due to positivity of the grading, for $\la \in \Lambda$,  we can choose a graded lift $\dot L(\la)$ of $L(\la)$,
 a graded lift $\dot M(\la)$ of $M(\la)$, a graded lift $\dot M(\la)^\vee$ of $M(\la)^\vee$, a graded lift $\dot P(\la)$ of $P(\la)$
 and a graded lift $\dot I(\la)$ of $I(\la)$, such that the canonical maps
 \begin{align*}
 	&\dot P(\la) \onto \dot M(\la) \onto \dot L(\la),\\
 	&\dot L(\la) \hookrightarrow \dot M(\la)^\vee \onto \dot I(\la)
 \end{align*} are  homogenous of degree zero.

 \subsection{Graded category $\OI$ for $\gl(m|n)$}
 Since $B$ is a positively graded algebra, we can denote by   $$\dotoo := \gmofB$$ the category of graded finite-dimensional unital   right $B$-modules
 whose morphisms are all homogenous homomorphisms of degree zero. We consider $\dot{\ov{\mc O}}$ as a graded version of $\ov{\mc O}$.

 Let $\mathbb F^B: \gmofB \rightarrow \mofB$ be the functor which forgets the grading. A module $X$ in $\mofB$ (respectively $X \in \ov{\mc O}$) is said to have a graded lift $\dot X\in \dotoo$ if $\mathbb F^B(\dot X)\cong X$ (respectively $T^B(X)\cong \mathbb F^B(\dot X)$).

 \begin{lem} \label{lem::117} Let $\dot{I_\nu}$ be the Serre subcategory of $\gmofA$ generated by $\{T(L(\la))\langle i\rangle|~\la\in \Lnua,~i\in \Z\}$. Then, the functor $$\dot \pi: \gmofA \rightarrow \dotoo,~ M\mapsto \bigoplus_{\la\in \Lnua} M1_\la,$$ gives rise to an equivalence  $\gmofA/I_\nu\cong\dotoo.$
 \end{lem}
 \begin{proof}
 	The proof is similar to the proof of Lemma \ref{lem::116}.
 \end{proof}

  Let  $\la \in \Lnua$. We recall modules $S(\la), D(\la)$ and $Q(\la)$ from Section \ref{sect::81}; see also Proposition \ref{prop::45}. We note that  $\pi(M(\la)^\vee)\cong \ov{\nb}(\la)$.  

 \begin{lem} \label{lem::18} Let $\la \in \Lnua$. Then the modules $S(\la)$, $D(\la)$ and $Q(\la)$ from Proposition \ref{prop::45}, and $M(\la)^\vee$ have graded lifts in $\gmofA$.
 \end{lem}
 \begin{proof}
 The proof is similar to the proof of \cite[Lemma 5.5]{BLW}.
 \end{proof}
 For a given $\la\in \Lnua,$ we let $\dot Q(\la)$ be the quotient of $\dot P(\la)$  by the submodule generated by the image of all homomorphisms $\dot P(\mu)\langle m \rangle\rightarrow \dot P(\la)$, for $m\in \Z$ and $\mu \in \Lnua$.
 From Lemma \ref{lem::18}, it is easy to see that $\dot Q(\la)$ is a graded lift  of  $Q(\la)$. 

 Let $\dot{X} \in \gmofA$ be a graded lift of $X\in \mofA$. We may observe that $\dot{\pi}(\dot{X})$ is a graded lift of $\pi'(X)\in \mofB$ since $$\mathbb F^B(\dot \pi (\dot X)) \cong \mathbb F^B(\bigoplus_{\la\in\Lnua} \dot X1_\la) \cong \bigoplus_{\la\in \Lnua} X1_\la \cong \pi'(X).$$ Therefore, for any $\la\in \Lnua$, \begin{align*}
 	&\dot P^B(\la):=\dot \pi(\dot{P}(\la)),~\dot\Delta^B(\la):=\dot \pi(\dot{Q}(\la)), \\
 	&\dot{\ov \nb}^B(\la):=\dot \pi(\dot{M}(\la)^\vee),~\dot{L}^B(\la):=\dot \pi(\dot{L}(\la)),	
 \end{align*} are the graded lifts of $\pi(P(\la)), \Delta(\la), \ov \nb(\la)\text{ and }\pi(L(\la))$, respectively. In particular, we have the following graded BGG reciprocity.

 \begin{lem}[Graded BGG reciprocity]
 	Let $\la,\mu\in \Lnua$. Then $\dot \pi(\dot{P}(\la))$ has a $\dot\Delta$-flag.  Furthermore, for any  $m\in \Z$, we have
 	\begin{align*}
 		&(\dot P^B(\la): \dot{\Delta}^B(\mu)\langle m \rangle) = [ \dot{\ov{\nb}}^B(\mu)\langle m \rangle: \dot{L}^B(\la)].
 	\end{align*}
 \end{lem}
 \begin{proof}
 	
 	The first assertion is proved using Lemma \ref{lem3} and an argument similar to the proof of \cite[Lemma 8.6]{MaSt04}. For the proof of the graded BGG reciprocity, we note the following:
 	\begin{align*}
 		&\bigoplus_{i\in \Z}\Ext^1_{\dotoo}(\dot{\Delta}^B(\la)\langle i\rangle, \dot{\ov \nb}^B(\mu))
 		\cong \Ext^1_{\mofB}(T^B{\Delta}(\la), T^B{\ov \nb}(\mu))\cong  \Ext^1_{\ov{\mc O}}({\Delta}(\la), {\ov \nb}(\mu)),
 	\end{align*}  which is zero by Theorem \ref{thm5}.
 	Similarly, we calculate
 	\begin{align*}
 		&\bigoplus_{i\in \Z}\Hom_{\dotoo }(\dot{\Delta}^B(\la)\langle i\rangle, \dot{\ov \nb}^B(\mu))
 		\cong \Hom_{\mofB}(T^B{\Delta}(\la), T^B{\ov \nb}(\mu))
 		\cong  \Hom_{\ov{\mc O}}({\Delta}(\la), {\ov \nb}(\mu)),
 	\end{align*}   which is zero, for $\la\neq \mu$, and is isomorphic to $\C$, for $\la=\mu$, by Lemma \ref{lem3}. Since $\dot \Delta^B(\la)$ has a quotient isomorphic to $\dot L^B(\la)$, it follows that $$\Hom_{\dotoo}(\dot{\Delta}^B(\la), \dot{\ov \nb}^B(\la)\cong \C.$$
 	The claim of the lemma follows now by a standard argument; see, e.g., \cite[Section~3.11]{Hu08}.
 \end{proof}

 For $\la,\mu \in \Lnua$, we define $$[\dot{\ov \nb}^B(\la):  \dot L^B(\mu)]_q:= \sum_{m\geq 0} [\dot{\ov \nb}^B(\la): \dot L^B(\mu)\langle m \rangle] q^m.$$ Similarly, we define $(\dot P^B(\la): \dot \Delta^B(\mu))_q$. We have the following graded version of Corollary \ref{cor:duality1}.
 \begin{prop}
 	For $\la,\mu\in\Lambda(\nu)$, we have
 	\begin{align*}
 		&(\dot P^B(\mu): \dot \Delta^B(\la))_q=[\dot{\ov \nb}^B(\la):\dot L^B(\mu)]_q =\texttt{t}_{-f_\la\cdot w_0^\nu,-f_\mu\cdot w_0^\nu}(q).
 	\end{align*}
 \end{prop}
 \begin{proof}
 	We calculate
 	\begin{align*}
 		&\sum_{m\ge 0}[\dot{\ov \nb}^B(\la)\langle m\rangle: \dot L^B(\mu)]q^m \\
 		&=\sum_{m\ge 0}[\dot \pi(\dot{M}(\la)^\vee): \dot \pi( \dot L(\mu)\langle -m\rangle)]q^m\\
 		&=\sum_{m\ge 0}[\dot M(\la)^\vee: \dot L(\mu)\langle -m\rangle]q^m \\
 		&=\sum_{m\ge 0}[\dot M(\la): \dot L(\mu)\langle m\rangle]q^m,
 	\end{align*}
 which is equal to the matrix coefficient of the matrix inverse to $\left(\ell_{f_{\mu},f_\la}(q^{-1})\right)$ by \cite[Section 5.9]{BLW}; see also \cite[Theorem 3.6]{Br}. Now, by \cite[Corollary 2.24]{Br1} this inverse matrix is precisely $\left( t_{-f_{\la},-f_{\mu}}(q)\right)$. By Proposition \ref{pi:can:basis}, the polynomial $t_{-f_{\la},-f_{\mu}}(q)$ is equal to $\texttt{t}_{-f_\la\cdot w_0^\nu,-f_\mu\cdot w_0^\nu}(q)$, which proves the claim.
 \end{proof}



\vspace{2mm}




\end{document}